\documentclass[12pt]{article}
\usepackage{hyperref}
\usepackage{color}

\usepackage{epsfig}
\usepackage{amsfonts}
\usepackage{amssymb}
\usepackage{amsmath}
\usepackage{amsthm}
\usepackage{latexsym}
\usepackage{graphicx}
\usepackage{bm}
\usepackage{tabularx}
\usepackage{tikz}
\usepackage{booktabs}
\usepackage{enumerate}
\usepackage{caption}
\usepackage[titletoc]{appendix}
\usepackage[numbers]{natbib}
\usepackage{float}
\usepackage[tight,footnotesize]{subfigure}
\usepackage{mwe}

\catcode`\@=11

\newcommand{\shortdot}[1]{\raisebox{-0.4pt}{$\stackrel{\bullet}{#1}$}}
\newcommand{\updot}[1]{\raisebox{0.9pt}{$\stackrel{\bullet}{#1}$}} 

\DeclareBoldMathCommand\boldlangle{\left\langle} 
\DeclareBoldMathCommand\boldrangle{\right\rangle}

\theoremstyle{plain}
\newtheorem{theorem}{Theorem}[section]
\newtheorem{lemma}[theorem]{Lemma}
\newtheorem{corollary}[theorem]{Corollary}
\newtheorem{proposition}[theorem]{Proposition}

\theoremstyle{definition}
\newtheorem{definition}{Definition}[section]
\newtheorem{assumption}[definition]{Assumption}

\theoremstyle{remark}

\newtheorem{innercustomthm}{Proposition}
\newenvironment{customthm}[1]
  {\renewcommand\theinnercustomthm{#1}\innercustomthm}
  {\endinnercustomthm}
\begin{document}

\title{A Stochastic Analysis of Bike Sharing Systems}
\author{ 
  Shuang Tao \\ School of Operations Research and Information Engineering \\ Cornell University
\\ 293 Rhodes Hall, Ithaca, NY 14853 \\  st754@cornell.edu  \\ 
\and
  Jamol Pender \\ School of Operations Research and Information Engineering \\ Cornell University
\\ 228 Rhodes Hall, Ithaca, NY 14853 \\  jjp274@cornell.edu  \\ 
 }

\maketitle
\begin{abstract}
 As more people move back into densely populated cities, bike sharing is emerging as an important mode of urban mobility.  In a typical bike sharing system, riders arrive at a station and take a bike if it is available. After retrieving a bike, they ride it for a while, then return it to a station near their final destinations.  Since space is limited in cities, each station has a finite capacity of docks, which cannot hold more bikes than its capacity.  In this paper, we study bike sharing systems with stations having a finite capacity.  By an appropriate scaling of our stochastic model, we prove a mean field limit and a central limit theorem for an empirical process of the number of stations with $k$ bikes.  The mean field limit and the central limit theorem provides insight on the mean, variance, and sample path dynamics of large scale bike sharing systems.  We also leverage our results to estimate confidence intervals for various performance measures such as the proportion of empty stations, the proportion of full stations, and the number of bikes in circulation.  These performance measures have the potential to inform the operations and design of future bike sharing systems.      

\end{abstract}


\section{Introduction}

Bike sharing  is an emerging mode of eco friendly transportation that have launched in over 400 cities around the world (\citet{nair2011fleet}).  In the United States, we are witnessing a transition where more people are deciding to live in large and densely populated cities.  As more people transition from the sprawling suburbs into densely populated cities, bike sharing programs will continue to grow in popularity since they provide easy transportation for citizens of these large cities.  As more people use these bike sharing systems(BSS), less people will drive motor vehicles on the road.  This reduction of vehicles on the road due to BSS has the potential to also reduce the growing traffic congestion in these growing cities.  BSS also promote healthy living as biking is a great form of exercise. They are environmentally friendly and they have the potential to reduce carbon emissions if operated correctly and efficiently, see for example \citet{hampshire2012analysis, nair2011fleet, nair2013large, nair2016equilibrium, o2015data, o2015smarter, jian2016simulation, freund2017minimizing} and their references within for more information on BSS.

For a typical system, riders simply arrive at a station and select a bike if there is one available for them to take.  If there is no bike available for the rider, the rider will leave the system.  Otherwise, the rider will take a bike and ride it for a while before returning it to another station near their final destination, if there is available space. If no space is available, the rider must find a nearby station to return the bike.  If there were an infinite supply of docks to store the bikes, then our model would be reduced to a network of infinite server queues, which is more tractable to analyze.  However, since the number of bike docks have finite capacity, the model become less tractable especially for systems with a large number of stations.  

Much of the complexity inherent in BSS lies in the scarcity of resources to move all riders around each city at all times of the day. Riders can encounter the scarcity of resources in two fundamental ways.  First, a rider can encounter insufficient resources by finding an empty station with no bikes when a rider needs one.  Secondly, a rider can find a station full with bikes when attempting to return a bike.  Thus, from a managerial point of view, having stations with no bikes or too many bikes are both problematic for riders.  There are several reasons why bike stations either have no or too many bikes.  One main reason why stations might have too many bikes or too few bikes is that the system is highly inhomogeneous.  Not only is the arrival rate a non-constant function of the time of day (see Figure~\ref{Fig_Avg_Trips}), but also riders do not evenly distribute themselves amongst the available stations.  For example, many stations that are located in residential areas have fewer bikes available for riding as many riders take bikes to more commercial areas.  Another example that illustrates the inhomogeneous dynamics is that riders tend to take bikes from up-hill stations to go to down-hill stations, however, very few riders take bikes from down-hill stations to go up-hill.  Thus, as more bikes flow from residential to commercial areas during rush hours, or up-hill to down-hill stations, this causes the system to be more imbalanced over time. 

Since BSS are quite complex, researchers have been inspired to study these systems in great depth.  The subsequent analysis of BSS has generated many insights on these systems, especially for rebalancing the fleet of bikes.  Although there is a large community that studies these systems, few analytical models have been proposed, especially stochastic analytical models.  This is primarily because the stochastic models for BSS are often very complex and are rather intractable to analyze without making strong assumptions.  Nevertheless, the analysis of such stochastic and mathematical models could provide insights on the behavior of these BSS and how to manage them effectively. In fact, a deeper analysis of  stochastic models for such systems could help researchers understand the impact of different incentive algorithms for taking or returning bikes, which can lead to significant improvements in the overall system performance.   

Most BSS can be viewed as closed queueing networks. One of the first papers to model the bike sharing system as a closed queueing network is by  \citet{george2011fleet}. However, in the bike sharing context, the customers are replaced by bikes. The number of bikes, also known as the fleet size, is fixed and remains constant.  In this model, the bikes can go to two types of stations.  The first type of station is a single-server queue where the service times are the user inter-arrival times to this station. The second type of station is an infinite-server queue where the service times are the trip times on a route from station $i$ to station $j$. The main drawback of the model by \citet{george2011fleet} is that it is based on infinite capacity queues.  This means that the model does not take into account the finite capacity of the stations and the related strategies of the users to return their bikes.  To overcome this major drawback, the model proposed in this paper allow the finite capacity at stations.  Thus, we are able to model the real system where customers are blocked from returning bikes to the stations that is nearest to their destination.  

In our model, we model the bike sharing system as state dependent $M/M/1/K_i$ queueing networks.  When joining a saturated single-server queue, the user reattempts in another queue, after a time with the same distribution as the trip time, until he returns his bike.  Although this model seems to model the behavior of the network, it is not practical since it scales with the number of stations.  Thus, we follow an approach developed by \citet{fricker2012mean} to study the bike sharing network's empirical process instead.  The empirical process still allows us to derive important performance measures of the original system, however, it scales with the maximum station capacity and not the number of stations, which is more practical for large networks like CitiBike.

\begin{figure}[htbp]
\centering
\includegraphics[width=1\textwidth]{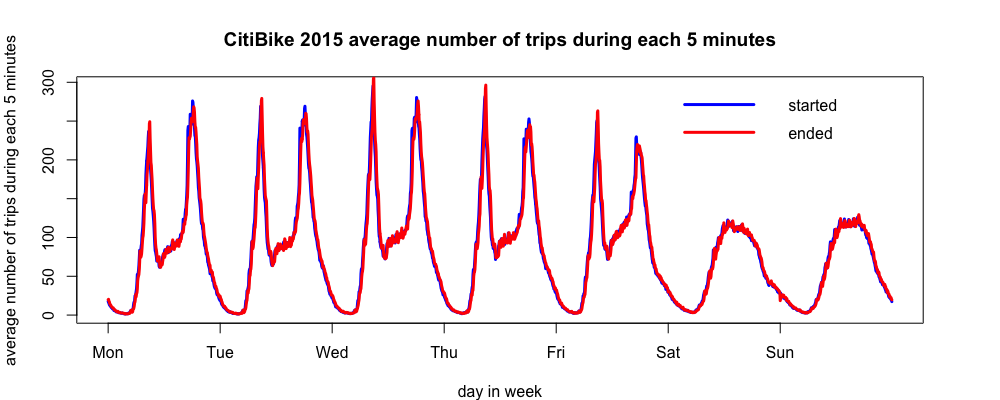}
\captionsetup{justification=centering}
\caption{CitiBike average number of trips during each 5 minutes (Jan 1st-Dec 31st, 2015).\\ The red line represents the average number of trips that started during each 5 minutes. The blue line represents the average number of trips that ended during each 5 minutes.}
\label{Fig_Avg_Trips}
\end{figure}

 \subsection{Main Contributions of Paper}

The contributions of this work can be summarized as follows:    
\begin{itemize}
\item We construct a stochastic bike sharing queueing model that incorporates the finite capacity of stations. Since our model is difficult to analyze for a large number of stations, we propose to analyze an empirical process that describes the proportion of stations that have a certain number of bikes. 
\item We prove a mean field limit and a central limit theorem for our stochastic bike sharing empirical process, showing that the mean field limit and the variance of the empirical process can be described by a system of $\frac{1}{2}(K+4)(K+1)$ differential equations where $K$ is the maximum station capacity.    
\item Using the mean field limit of the empirical process, we are able to approximate the mean proportion of empty and full stations.  Furthermore, with the central limit theorem of the empirical process, we are able to construct confidence intervals around the mean field limit for the same performance measures.
\item We compare the mean field limit and the central limit theorem to a simulation and show that the differential equations approximate the mean and variance of the empirical process extremely well. 
\end{itemize} 


\subsection{Organization of Paper}

The remainder of this paper is organized as follows. Section~\ref{history} provides a brief history of bike sharing programs and a literature review on the research streams concerning BSS. Section~\ref{Sec_Bike_Model} introduces our bike sharing model and notation of the paper.  In Section~\ref{Fluid_Limit}, we derive amd prove the mean field limit of the empirical measure process of the distribution of stations with different bikes and utilization rate. In Section~\ref{Diffusion_Limit}, we derive the diffusion limit and prove a functional CLT for our model. We show that the diffusion process is a centered Gaussian OU process and we also obtain a closed form expression of the diffusion limit process. In Section~\ref{Ext}, we extend our analysis in Section~\ref{Sec_Bike_Model} to a broader case with non-uniform routing probabilities and capacities, and derive the mean field limit and diffusion limit in this extended case. In Section~\ref{simulation}, we discuss the simulation results of our model, with both stationary and non-stationary arrival processes. We also give a comparison between real CitiBike empirical measure and simulated ones using our model to show how well our model is in capturing reality.  Finally, in Section~\ref{conc}, we give concluding remarks and provide some future directions of research that we intend to pursue later.

\section{History and Literature Review}\label{history}

The literature that focuses on the analysis and operations of BSS is increasing rapidly. Early research that studied the history of BSS includes \citet{shaheen2010bikesharing}, \citet{hampshire2012analysis}, \citet{nair2013large}, \citet{schuijbroek2017inventory}, and \citet{demaio2009bike}.  These papers provide a history of bike sharing and how it has evolved over time.  The beginning of bike sharing can be traced back to the first generation of \emph{white bikes} (or free bikes) in Amsterdam, The Netherlands as early as 1965. However, this first generation of BSS failed due to a large amount of bike theft.  The launch of Bycyklen in Copenhagen in 1995 marked the second generation of BSS.  Bycyklen was the first BSS to implement docking stations and coin-deposit systems to unlock bikes.  These coin-deposit systems helped with the bike theft problem and thus made the BSS more reliable. However, even with these improvements, the Bycyklen could not eliminate bike thefts mainly due to the fact that customer still remained anonymous and there were no time limits on how long a customer could use a bike.  

The failures of the second generation of BSS inspired the present-day or the third generation of BSS by combining docks with information technology. These new systems incorporate information technology for bicycle reservations, pickups, and drop-offs.  This new technology has enabled many BSS to keep better track of bicycles and the users that use them, thus eliminating virtually all bike theft.  One example of a third generation BSS is the Paris’ bike sharing program Velib.  Velib was launched in Paris in July 2007 and has emerged as the most prominent example of a successful bike sharing program in the modern world. As a result of the success of Velib,  many cities like New York City (Citi Bike), Chicago (Divvy), and even Ithaca (Big Red Bikes) have implemented large-scale BSS and bike sharing has become a widely used form of transportation in these cities.  For the interested reader,  \citet{laporte2015shared} provides a comprehensive survey of the vehicle/bike sharing literature.

Rebalancing is currently the biggest stream of research concerning BSS.  In rebalancing operations, there are two methods of rebalancing: (1) deploying a truck fleet or (2) providing user incentives, where deploying a truck fleet is often referred to as bike repositioning. Both methods involve static and dynamic cases.  Static repositioning usually is moving bikes during the night when traffic flow is low, while dynamic repositioning is moving bikes during the day based on current state of the system. Most research on this area focuses on the static case, partly because it is easier to model and also because the impact of repositioning is more important during the night (\citet{jian2016simulation}). \citet{raviv2013static} is one of the first papers to study static repositioning of BSS, using mixed integer linear programming by maximizing customer demand satisfaction. \citet{benchimol2011balancing} consider a similar problem, where a single truck repositions bikes to bring the inventory of each station to a predetermined value.  However, their objective is to minimize the routing cost as opposed to maximizing customer satisfaction. In the case of dynamic repositioning, \citet{chemla2013self} and \citet{pfrommer2014dynamic} consider the case when the trucks respond in real time to the current state of the system.  However, \citet{contardo2012balancing} and \citet{ghosh2017dynamic} consider the situation where the time-dependent demand is known a priori and the rebalancing operations are computed in an off-line fashion. Yet, none of these papers really explore stochastic dynamics and they for the most part mainly exploit optimization techniques to tackle the problem.

Unlike the rebalancing literature, our paper falls into the performance analysis literature with an emphasis in supply analysis. We focus on analyzing the most salient performance measures such as the mean, variance, covariance and sample path dynamics of the bike distributions in a large-scale BSS.  There is not much literature that explores the fluctuations of BSS around the mean field limit.  In this paper, we prove a mean field limit and a functional central limit theorem under some smoothness conditions and show that the diffusion limit is characterized by a multi-dimensional Ornstein-Uhlenbeck (OU) process.  The functional central limit theorem not only gives us information about the sample path fluctuations of the queue length process, but it also allows us to construct approximate confidence intervals for various performance measures such as the proportion of empty stations, the proportion of full stations, and the mean number of bikes in circulation. Unlike the previous literature, our paper also considers non-stationary arrivals to stations, which is much more realistic given the user patterns we observe in the historical data from Citi Bike. In Figure~\ref{Fig_Avg_Trips}, we plot the empirical mean of the number of trips (5 minute intervals) during the week for the time period Jan 1st-Dec 31st, 2015. We observe from Figure~\ref{Fig_Avg_Trips} that the arrival rate is non-stationary and clearly reflects the morning rush and evening rush during the peak times. Thus, analyzing the non-stationary dynamics is crucial for understanding the impact of system inhomogeneity since it can provide useful guidelines for rebalancing operations.


\section{Bike-Sharing Queueing Model} \label{Sec_Bike_Model}

\begin{figure}[htbp]
\centering
\includegraphics[scale=0.27]{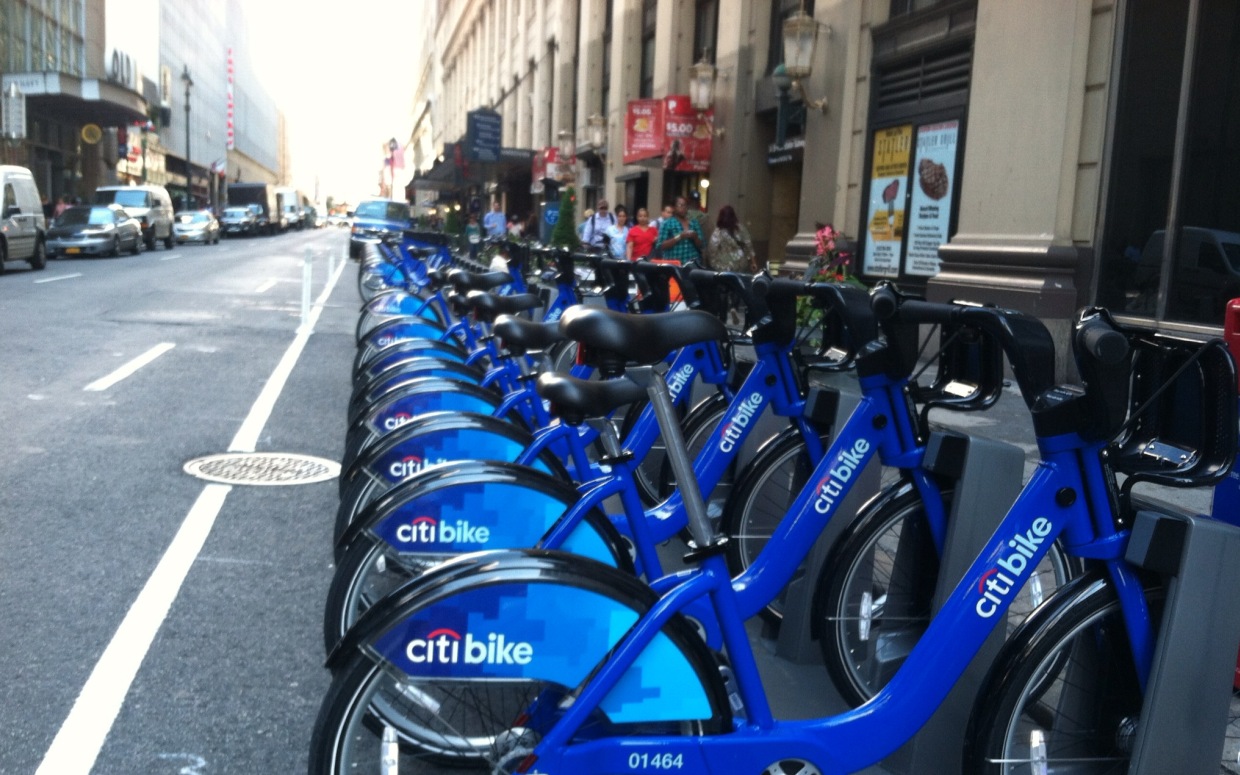}
\caption{Figure of a typical CitiBike Station.} \label{Fig_Shuang_station}
\end{figure}

In this section, we construct a Markovian bike sharing queueing model where customers can pick-up and drop-off bikes at each station if there is available capacity.  Figure \ref{Fig_Shuang_station} provides an illustration of a typical Citi Bike station in New York City (NYC).  As one can see in Figure \ref{Fig_Shuang_station}, the bikes are attached to docks and the number of docks is finite with roughly 40 docks.  Figure \ref{Fig_station_map} shows a map of Citi Bike stations, the nation's largest bike sharing program, with over 10,000 bikes and 600 stations across Manhattan, Brooklyn, Queens and Jersey City. Citi Bike was designed for quick, affordable and convenient trips, and has become an essential part of the transportation infrastructure in NYC.

\begin{figure}[htbp]
\centering
\includegraphics[scale=0.21]{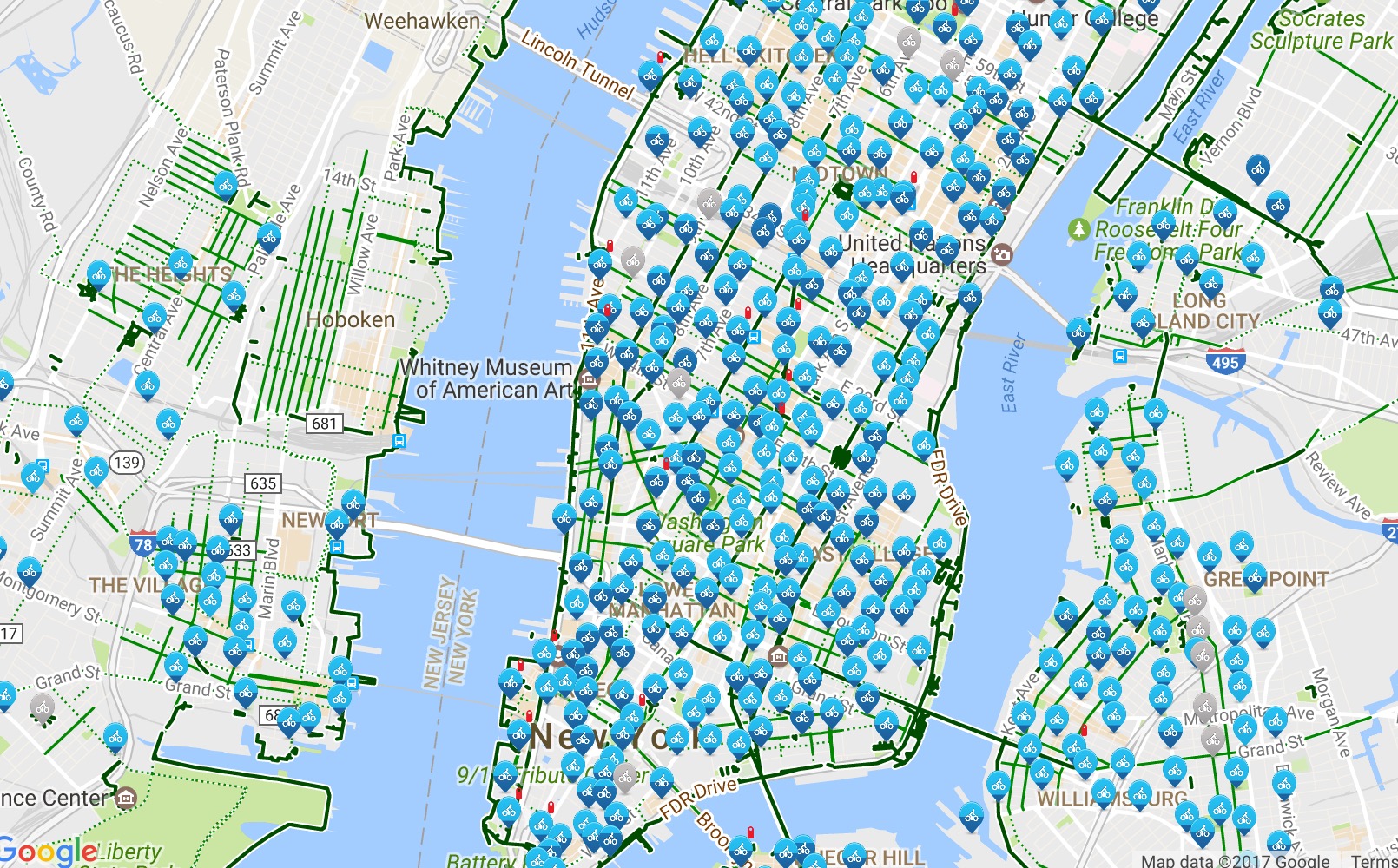}
\caption{Citi Bike stations map} \label{Fig_station_map}
\end{figure}
 
Motivated by the Citi Bike bike sharing system, we consider a bikes sharing system with $N$ stations and a fleet of $M$ bikes in total. We assume that the arrival of customers to the stations are independent Poisson processes with rate $\lambda_i$ for station $i$.  When a customer arrives at a station, if there is no bikes available, they will then leave the system and are immediately blocked and lost. Otherwise, the customer will take a bike and ride to station $j$ with probability $P_j$.  We assume that the travel time of the rider is exponentially distributed with mean $1/\mu$, for every transition from one station to another.  Since we are concerned with finite capacity stations, we assume that station $i$ has a bike capacity of $K_i$, which is assumed to be finite for all stations.   Thus, when a customer arrives at station $j$, if there are less than $K_j$ bikes in this station, he returns his bike and leaves the system. If there are exactly $K_j$ bikes (i.e. the station is full), the customer randomly chooses another station $k$ with probability $P_k$ and goes to that station to drop the bike off.  As before, it takes a time that is exponentially distributed with mean $1/\mu$. Finally, the customer rides like this again until he can return his bike to a station that is not full. 

Below, we provide Table \ref{notation} for the reader's convenience so that they understand the notation that we will use throughout the paper.  

\begin{table}[ht]
\caption{Notation}\label{notation} 
\centering 
\begin{tabular}{c c } 
\hline\hline 

\hline 
$N$ & Number of stations \\ 
$M$ & Total number of bikes \\
$K_{i}$ & Capacity at station $i$\\
$\lambda_i$ & Arrival rate at station $i$\\
$1/\mu$ & Mean travel time\\
$P_i$ & Routing probability to station $i$\\
$X_i(t)$ & Number of bikes at station $i$ at time $t$\\
$R_i=\mu P_i /\lambda_i$ & Utilization at station $i$  \\
$r_i=R_i/\max_{i}R_i$ & Relative utilization at station $i$\\
$\gamma$ & Average number of bikes at each station\\
\hline 
\end{tabular}
\label{table:nonlin} 
\end{table}

To avoid cumbersome notation, throughout Sections \ref{Sec_Bike_Model}, \ref{Fluid_Limit} and \ref{Diffusion_Limit}, we \textbf{assume without loss of generality} that the service rate $\mu$ is equal to 1 and that the routing probability from each station is uniform, i.e. $P_i=1/N$. We also assume that capacities across all stations are equal, i.e. $K_i=K$, for  $i=1,...,N$. With our notation in hand, we are ready to develop our stochastic model for our bike sharing network.  At first glance, these assumptions seem restrictive, however, we explain in Section \ref{Ext-routing}  how our model can be extended seamlessly to more complex settings with non-uniform routing probabilities and capacities.  We should mention, however, the extension to inter-station transition probabilities i.e probabilities that depends on the departing station as well as the returning station are non-Markovian models.  In Section \ref{Ext-repositioning} we discuss more in detail how one can extend the state space, by tracking the number of in-transit bikes coming from each station, to make the queueing model Markovian, however this extended model makes the problem even more high-dimensional and adds difficulty to the analysis.

We define $\mathbf{X}(t)=(X_{1}(t),\cdots, X_{N}(t))$, where $X_{i}(t)$ is the number of bikes at station $i$ at time $t$. Then $X_{i}(t)$ is a continuous time Markov chain(CTMC), in particular a state dependent $M/M/1/K_{i}$ queue. In this model, the rate of dropping off bikes at station $i$ is equal to $\mu P_{i}\left(M-\sum_{k=1}^{N}X_{k}(t)\right)\mathbf{1}\{X_{i}(t)<K_{i}\}$, and the rate of retrieving bikes at station $i$ is equal to $\lambda_{i}\mathbf{1}\{X_{i}(t)>0\}$.  Using these rates, we can construct the functional forward equations for our stochastic bike sharing model.  This construction is given below in the following proposition:

 \begin{proposition}\label{functional_forward}
 For any integrable  function $f: \mathbb{Z}_{+}^{N}\rightarrow \mathbb{R}$, the CTMC $\mathbf{X}(t)$ satisfy the following functional forward equation,
 \begin{eqnarray}
 & &\lefteqn{ \updot{\mathbb{E}}[f(\mathbf{X}(t)) | \mathbf{X}(0) = \mathbf{x}]  \equiv  \frac{d}{dt} \mathbb{E}[f(\mathbf{X}(t)) | \mathbf{X}(0) = \mathbf{x}]  } \nonumber \\ 
 &=&  \sum_{i=1}^{N} \mathbb{E}\left[\left(f(\mathbf{X}(t)-\mathbf{1}_{i})-f(\mathbf{X}(t)\right)\lambda_{i}\mathbf{1}_{\{X_{i}(t)>0\}}\right]\nonumber  \\
& & +\sum_{i=1}^{N} \mathbb{E}\left[\left(f(\mathbf{X}(t)+\mathbf{1}_{i})-f(\mathbf{X}(t)\right)\mu P_{i}\left(M-\sum_{k=1}^{N}X_{k}(t)\right)\mathbf{1}_{\{X_{i}(t)<K_{i}\}} \right] .
 \end{eqnarray} 
 \end{proposition}

\begin{proof}
The proof is found in the Appendix.  
\end{proof}
 
\begin{corollary}\label{funcforeqns}
The time derivatives of the mean, variance, and covariance of $\mathbf{X}(t)$ are given by
\begin{eqnarray}
 \updot{\mathbb{E}}[X_i(t)]   &=& \mathbb{E}\left[ \mu P_{i}\left(M-\sum_{k=1}^{N}X_{k}(t)\right)\mathbf{1}\{X_{i}(t)<K_{i}\}\right]  - \lambda_{i} \mathbb{P}\left[ X_{i}(t) > 0   \right] ,
 \\
\updot{\mathrm{Var}}[X_i(t)]  &=& \updot{\mathbb{E}}[X_i^{2}(t)]-2\updot{\mathbb{E}}[X_i(t)]\mathbb{E}[X_i(t)]\nonumber \\
 &=& 2 \mathrm{Cov}\left[X_i(t),\mu P_{i}\left(M-\sum_{k=1}^{N}X_{k}(t)\right)\mathbf{1}\{X_{i}(t)<K_{i}\} \right]  \nonumber \\
 &-& 2 \mathrm{Cov}\left[X_i(t),  \lambda_{i}\mathbf{1}\{ X_{i}(t) > 0 \}  \right]\nonumber \\
  & +& \mathbb{E}\left[ \mu P_{i}\left(M-\sum_{k=1}^{N}X_{k}(t)\right)\mathbf{1}\{X_{i}(t)<K_{i}\} \right] +  \lambda_{i} \mathbb{P}\left[ X_{i}(t) > 0   \right]  ,
\end{eqnarray}

\begin{eqnarray}
 \updot{\mathrm{Cov}}[X_i(t), X_j(t)] 
&=&\mathrm{Cov}\left[\mu P_{j}\left(M-\sum_{k=1}^{N}X_{k}(t)\right)\mathbf{1}_{\{X_{j}(t)<K_{j}\} }-\lambda_j\mathbf{1}_{\{X_{j}(t)>0\} },X_{i}(t) \right]\nonumber \\
&+ &  \mathrm{Cov}\left[\mu P_{i}\left(M-\sum_{k=1}^{N}X_{k}(t)\right)\mathbf{1}_{\{X_{i}(t)<K_{i}\} }-\lambda_i\mathbf{1}_{\{X_{i}(t)>0\} },X_{j}(t) \right],\nonumber\\
\end{eqnarray}
for $i,j=1,\cdots,N$ and $i\neq j$.
\end{corollary}
%
  
   \subsection{Intractability of Individual Stations Model}\label{intractability}
   
Although the functional forward equations given in Corollary \ref{funcforeqns} describe the exact dynamics of the mean, variance and covariance of the bike sharing system, the system of differential equations are not closed.  This non-closure property of the functional forward equations in this model arises from the fact that the bike sharing system has finite capacity.  More importantly, it also implies that we need to know a priori the full distribution of the whole stochastic process $\mathbf{X}(t)$ in order to calculate the mean or variance or any moment for that matter. Work by \citet{ Massey2013, Pender2014, Pender2015, EP, pender2016sampling} could yield useful and accurate closure approximations for making the system closed.  Moreover, withe exception of \citet{EP}, there  are no error bounds on the accuracy of various closure approximations.  Thus, it is not clear how well the closure approximations would perform over a variety of parameter settings.  Finally if we even wanted to solve these equations and knew the entire distribution of $\mathbf{X}(t)$ a priori, there are still $O(N^2)$ differential equations(around 180,900 equations in the CitiBike case) that would need to be numerically integrated.  This is very computationally expensive and thus, we must take a different approach to analyze our bike sharing system. 
 
Moreover, if we want to analyze the limiting behavior of $\{X_{i}\}_{i=1}^{N}$ as  CTMCs, as we let $N$ go to $\infty$, the mean field limit would become infinite dimensional, which is quite complicated.  However, if we instead analyze an empirical measure process for $\mathbf{X}(t)$, we can use the finite capacity nature of the bike sharing system to our advantage and have a finite dimensional CTMC for the empircal measure process.  

   \subsection{An Empirical Measure Model}
 
Following the model of \citet{fricker2012mean}, we construct an empirical measure process that counts the proportions of stations with $n$ bikes and utilization $r$.  This empirical measure process is given below by the following equation: 
\begin{equation}
Y_{t}^{N}(r,n)=\frac{1}{N}\sum_{i=1}^{N}\mathbf{1} \{r_i^{N}=r,X_i^N(t)=n\}.
\end{equation}

We further define that $Y_{t}^{N}(n)=\sum_{r}Y_{t}^{N}(r,n)$. By observing the empirical measure process, we notice that $Y_{t}^{N}=(Y_{t}^{N}(0),\cdots,Y_{t}^{N}(K)) \in [0,1]^{K+1}$.  Thus, for our empirical measure process, we only need to solve $O(K^2)$ differential equations for understanding the mean and variance dynamics of the bike sharing system, where $K<<N$. More importantly, the empirical measure will also allow us to obtain salient performance measures such as $Y_{t}^{N}(0)$ (the proportion of stations with no bikes), $Y_{t}^{N}(K)$ (the proportion of stations that are full of bikes), $M - \sum^{K}_{j=0} j \cdot Y_{t}^{N}(j)N$ (the number of bikes in circulation), among others.

Conditioning on $Y_{t}^{N}(r,n)=y(r,n)$ and given our assumptions that $\mu=1$ and $P_{i}=1/N$, the transition rates of $y$ are specified as follows:\\
When a customer arrives to a station with $n$ bikes and relative utilization $r$ to retrieve a bike, the
proportion of stations having $n$ bikes goes down by $1/N$, the proportion of stations having $n-1$ bikes goes up by $1/N$, and the transition rate $Q^N$ is 
\begin{eqnarray}
Q^{N}\left(y,y+\frac{1}{N}(\mathbf{1}_{(r,n-1)}-\mathbf{1}_{(r,n)}) \right) &=& 
y(r,n)\lambda_r N \mathbf{1}_{n>0}\nonumber \\ 
&=& y(r,n)\frac{\mu P_i}{R}N \mathbf{1}_{n>0}\nonumber \\
&=& y(r,n)\frac{1}{NR} N\mathbf{1}_{n>0} \nonumber \\ &=&
\frac{y(r,n)}{rR^N_{\max}} \mathbf{1}_{n>0}.
\end{eqnarray}
When a customer returns a bike to a station with $n$ bikes and relative utilization $r$, the proportion of stations having $n$ bikes goes down by $1/N$, the proportion of stations having $n+1$ bikes goes up by $1/N$, and the transition rate $Q^N$ is 
\begin{eqnarray}
Q^{N} \left(y,y+\frac{1}{N}(\mathbf{1}_{(r,n+1)}-\mathbf{1}_{(r,n)}) \right) &=&
 y(r,n) \cdot \mu \cdot \left(M-\sum_{n'}\sum_{r'}n'y(r',n')N\right)\mathbf{1}_{n<K}\nonumber  \\  &=&
 y(r,n)N\left(\frac{M}{N}-\sum_{n'}\sum_{r'}n'y(r',n')\right)\mathbf{1}_{n<K}.
\end{eqnarray}
Similarly, we have the functional forward equations for $Y_{t}^{N}(r,n)$.

\begin{proposition}\label{functional_forward_Y}
 For any integrable function $f: [0,1]^{K+1}\rightarrow \mathbb{R}$, $Y_{t}^N(r)=(Y_{t}^{N}(r,0),\cdots,Y_{t}^{N}(r,K))$ satisfies the following functional forward equation,
 \begin{eqnarray}
 & &\updot{\mathbb{E}}(f(Y_{t}^{N}(r)) | Y_{0}^{N}(r) = y_{0}(r)]\nonumber \\
 &=&  \sum_{n=0}^{K}\mathbb{E}\left[\left(f\left(Y_t^N(r)+\frac{1}{N}(\mathbf{1}_{r,n-1}-\mathbf{1}_{r,n})\right)-f(Y_t^N(r))\right)\frac{Y^{N}_{t}(r,n)}{rR_{\max}}\mathbf{1}_{n>0}\right]\nonumber \\
 &+&\sum_{n=0}^{K}\mathbb{E}\left[\left(f\left(Y_t^N(r)+\frac{1}{N}(\mathbf{1}_{r,n+1}-\mathbf{1}_{r,n})\right)-f(Y_t^N(r))\right)Y^{N}_{t}(r,n)N\left(\frac{M}{N}\right.\right.\nonumber\\
 &-&\left.\left.\sum_{n'}\sum_{r'}n'Y_{t}^{N}(r',n')\right)\mathbf{1}_{n<K}\right]
 \end{eqnarray}
 \end{proposition}
 
  \begin{proof}
  The proof is similar to the proof of Proposition~\ref{functional_forward}.
  \end{proof}

 \begin{corollary}
  The time derivative of the mean of $Y_{t}^{N}(r,n)$ is given by 
 \begin{eqnarray}
\updot{\mathbb{E}}[Y_{t}^{N}(r,n)]  \nonumber  &=& \mathbb{E}\left[ \frac{1}{rNR^N_{\max}} \left(Y^{N}_{t}(r,n+1)\mathbf{1}_{n<K}-Y^{N}_{t}(r,n)\mathbf{1}_{n>0}\right) \right] \\
& & + \mathbb{E}\left[ \left(\frac{M}{N}-\sum_{n'}\sum_{r'}n'Y_{t}^{N}(r',n')\right)  \left(Y^{N}_{t}(r,n-1)\mathbf{1}_{n>0}-Y^{N}_{t}(r,n)\mathbf{1}_{n<K}\right) \right].\nonumber\\
 \end{eqnarray}
 for $n=0,\cdots,K$. 
 Denote $\Sigma_{i,j}^{N}(r,t)=\mathrm{Cov}[Y_{t}^{N}(r,i),Y_{t}^{N}(r,j)]$. When $i=j$, 
the time derivative of the variance term $\mathrm{Var}[Y_{t}^{N}(r,i)]$ is given by 
   \begin{eqnarray}
\updot{\mathrm{Var}}[Y_{t}^{N}(r,i)]  \nonumber &=& \updot{\mathbb{E}}[Y_{t}^{N}(r,i)^2]-2\updot{\mathbb{E}}[Y_{t}^{N}(r,i)]\mathbb{E}[Y_{t}^{N}(r,i)]\nonumber \\
 &=&  \frac{2}{rNR^N_{\max}}\left(\Sigma_{i,i+1}^{N}(r,t)\mathbf{1}_{i<K}-\Sigma_{i,i}^{N}(r,t)\mathbf{1}_{i>0}\right)+\frac{2M}{N}\left(\Sigma^{N}_{i,i-1}(r,t)\mathbf{1}_{i>0} -\Sigma^{N}_{i,i}(r,t)\mathbf{1}_{i<K}\right)\nonumber \\
 & &- 2 \sum_{n'}\sum_{r'}n'\left(\mathrm{Cov}\left[Y_{t}^{N}(r,i),Y_{t}^{N}(r',n')Y^{N}_{t}(r,i-1)\right]\mathbf{1}_{i>0}\right.\nonumber \\
 & &\left.-\mathrm{Cov}\left[Y_{t}^{N}(r,i),Y_{t}^{N}(r',n')Y^{N}_{t}(r,i) \right]\mathbf{1}_{i<K}\right)\nonumber \\
  & &+ \frac{1}{rN^2R^N_{\max}} \left(\mathbb{E}\left[Y^{N}_{t}(r,i+1) \right]\mathbf{1}_{i<K}+\mathbb{E}\left[Y^{N}_{t}(r,i)\right]\mathbf{1}_{i>0}\right)\nonumber  \\
  & &+\frac{M}{N^2}\left(\mathbb{E}\left[Y^{N}_{t}(r,i-1)\right]\mathbf{1}_{i>0}+\mathbb{E}\left[Y^{N}_{t}(r,i)\right]\mathbf{1}_{i<K}\right)\nonumber \nonumber \\
& & -\frac{1}{N} \sum_{n'}\sum_{r'}n'\left(\mathbb{E}\left[Y_{t}^{N}(r',n')Y^{N}_{t}(r,i-1)\right]\mathbf{1}_{i>0}+\mathbb{E}\left[Y_{t}^{N}(r',n')Y^{N}_{t}(r,i) \right]\mathbf{1}_{i<K}\right).\nonumber\\
 \end{eqnarray}
  When $|i-j|>1$, the time derivative of the covariance term $\mathrm{Cov}[Y_{t}^{N}(r,i), Y_{t}^{N}(r,j)]$ is given by
   \begin{eqnarray}
& & \updot{\mathrm{Cov}}[Y_{t}^{N}(r,i), Y_{t}^{N}(r,j)]  \nonumber \\
&=&\updot{\mathbb{E}}[Y_{t}^{N}(r,i)Y_{t}^{N}(r,j)]-\updot{\mathbb{E}}[Y_{t}^{N}(r,i)]\mathbb{E}[Y_{t}^{N}(r,j)] -\updot{\mathbb{E}}[Y_{t}^{N}(r,j)]\mathbb{E}[Y_{t}^{N}(r,i)]\nonumber \\
&=& \frac{1}{rNR^N_{\max}}\left[\Sigma_{i+1,j}^{N}(r,t)\mathbf{1}_{i<K}+\Sigma_{i,j+1}^{N}(r,t)\mathbf{1}_{j<K}-\Sigma_{i,j}^{N}(r,t)\left(\mathbf{1}_{j>0}+\mathbf{1}_{i>0}\right)\right]\nonumber \\
& &+\frac{M}{N}\left[\Sigma^{N}_{i-1,j}(r,t)\mathbf{1}_{i>0}+\Sigma^{N}_{i,j-1}(r,t)\mathbf{1}_{j>0} -\Sigma^{N}_{i,j}(r,t)\left(\mathbf{1}_{i<K}+\mathbf{1}_{j<K}\right)\right]\nonumber \\
 & &-  \sum_{n'}\sum_{r'}n'\left(\mathrm{Cov}\left[Y_{t}^{N}(r,j),Y_{t}^{N}(r',n')Y^{N}_{t}(r,i-1)\right]\mathbf{1}_{i>0}-\mathrm{Cov}\left[Y_{t}^{N}(r,j),Y_{t}^{N}(r',n')Y^{N}_{t}(r,i) \right]\mathbf{1}_{i<K}\right)\nonumber \\
  & &-  \sum_{n'}\sum_{r'}n'\left(\mathrm{Cov}\left[Y_{t}^{N}(r,i),Y_{t}^{N}(r',n')Y^{N}_{t}(r,j-1)\right]\mathbf{1}_{j>0}-\mathrm{Cov}\left[Y_{t}^{N}(r,i),Y_{t}^{N}(r',n')Y^{N}_{t}(r,j) \right]\mathbf{1}_{j<K}\right).\nonumber \\
 \end{eqnarray}
and when $j=i+1$, the time derivative of the covariance term $\mathrm{Cov}[Y_{t}^{N}(r,i), Y_{t}^{N}(r,i+1)]$ is given by
\begin{eqnarray}
& & \updot{\mathrm{Cov}}[Y_{t}^{N}(r,i), Y_{t}^{N}(r,i+1)]  \nonumber \\
&=&\updot{\mathbb{E}}[Y_{t}^{N}(r,i)Y_{t}^{N}(r,i+1)]-\updot{\mathbb{E}}[Y_{t}^{N}(r,i)]\mathbb{E}[Y_{t}^{N}(r,i+1)] -\updot{\mathbb{E}}[Y_{t}^{N}(r,i+1)]\mathbb{E}[Y_{t}^{N}(r,i)]\nonumber \\
&=&\frac{1}{rNR^N_{\max}}\left[\Sigma_{i+1,i+1}^{N}(r,t)+\Sigma_{i,i+2}^{N}(r,t)\mathbf{1}_{i<K-1}-\Sigma_{i,i+1}^{N}(r,t)\left(1+\mathbf{1}_{i>0}\right)\right]\nonumber \\
& &+\frac{M}{N}\left[\Sigma^{N}_{i-1,i+1}(r,t)\mathbf{1}_{i>0}+\Sigma^{N}_{i,i}(r,t) -\Sigma^{N}_{i,i+1}(r,t)(1+\mathbf{1}_{i<K-1})\right]\nonumber \\
 & &-  \sum_{n'}\sum_{r'}n'\left(\mathrm{Cov}\left[Y_{t}^{N}(r,i+1),Y_{t}^{N}(r',n')Y^{N}_{t}(r,i-1)\right]\mathbf{1}_{i>0}\right.\nonumber\\
 & &\left.-\mathrm{Cov}\left[Y_{t}^{N}(r,i+1),Y_{t}^{N}(r',n')Y^{N}_{t}(r,i) \right]\right)\nonumber \\
  & &-  \sum_{n'}\sum_{r'}n'\left(\mathrm{Cov}\left[Y_{t}^{N}(r,i),Y_{t}^{N}(r',n')Y^{N}_{t}(r,i)\right]-\mathrm{Cov}\left[Y_{t}^{N}(r,i),Y_{t}^{N}(r',n')Y^{N}_{t}(r,i+1) \right]\mathbf{1}_{i<K-1}\right)\nonumber \\
& &-\frac{\mathbb{E}\left[Y_{t}^{N}(r,i+1)\right]}{rN^2R_{\text{max}}}  -\frac{M}{N^2}\mathbb{E}\left[Y_{t}^{N}(r,i)\right]+\frac{1}{N}\sum_{n'}\sum_{r'}n'\mathbb{E}\left[Y_{t}^{N}(r',n')Y_{t}^{N}(r,i)\right] .\nonumber\\
\end{eqnarray}

 \end{corollary}
 \begin{proof}
The time derivatives of $\mathbb{E}[Y_t^N(r,i)]$ and $\mathrm{Var}[Y_t^N(r,i)]$ come directly from applying Proposition \ref{functional_forward_Y} with $f(Y_{t}(r))=Y_{t}(r,i), Y_{t}^2(r,i)$ respectively. For the covariance term, we let $f(Y_t^N(r))=f(Y^{N}_{t}(r,i),Y^{N}_{t}(r,j))$ to prove the result.
 \end{proof}

Although we have equations for the moments of the empirical process and individual stations, it is still difficult to analyze them and gain insights from them directly.  One reason is that even though we have reduced the dimensionality of the analysis significantly, we have not removed the non-closure property of the differential equations.  Thus, we need to develop a new approach that will allow us to get around this complication.  The method that we choose to use is asymptotic analysis and will be described in more details in the sequel.  However, before we get to the asymptotic analysis we state some technicalities about weak convergence.

 \subsection{Preliminaries of Weak Convergence}

Following \citet{ko2016strong}, we assume that all random variables in this paper are defined on a common probability space $(\Omega, \mathcal{F}, \mathbb{P})$.  Moreover, for all positive integers $k$, we let $\mathcal{D}([0 , \infty), \mathbb{R}^k)$ be the space of right continuous functions with left limits (RCLL) in $\mathbb{R}^k$ that have a time domain in $[0, \infty)$.  As is usual, we endow the space $\mathcal{D}([0 , \infty), \mathbb{R}^k)$ with the usual Skorokhod $J_1$ topology, and let $M^k$ be defined as the Borel $\sigma$-algebra associated with the $J_1$ topology. We also assume that all stochastic processes are measurable functions from our common probability space $(\Omega, \mathcal{F}, \mathbb{P})$ into $(\mathcal{D}([0 , \infty), \mathbb{R}^k), M^k)$.   Thus, if $\{\zeta\}^\infty_{n=1}$ is a sequence of stochastic processes, then the notation $\zeta^n \rightarrow \zeta$ implies that the probability measures that are induced by the $\zeta^n$'s on the space $(\mathcal{D}([0 , \infty), \mathbb{R}^k), M^k)$ converge weakly to the probability measure on the space $(\mathcal{D}([0 , \infty), \mathbb{R}^k), M^k)$ induced by $\zeta$.  For any $x \in (\mathcal{D}([0 , \infty), \mathbb{R}^k), M^k)$  and any $T > 0$, we define 
\begin{equation}
||x||_T \equiv \sup_{0 \leq t \leq T}  \ \max_{i = 1,2,...,k} |x_i(t)|
\end{equation}
and note that $\zeta^n$ converges almost surely to a continuous limit process $\zeta$ in the $J_1$ topology if and only if 
\begin{equation}
||\zeta^n - \zeta||_T \to 0 \quad a.s.
\end{equation}
for every $T > 0$.

\section{Mean Field Limit} \label{Fluid_Limit}

In this section, we prove the mean field limit for our bike sharing model.  A mean field limit describes the large station dynamics of the bike sharing network over time.  Deriving the mean field limit allows us to obtain new insights on average system dynamics, when the demand for bikes and the number of stations are very large. Thus, we avoid the need to study an $N$-dimensional CTMC and compute its steady state distribution in this high dimensional setting. 

First, we state the important assumptions that will be used through out the paper, to ensure the existence of a mean field limit of our model.

\begin{assumption}\label{assumption}
 There exists a probability measure $I(r,k)$  on $]0,1] \times	 \mathbb{N}$ with finite support and $\Lambda > 0$ such that, as $N$ tends to infinity, we have
\begin{itemize}
\item[i)] $\frac{1}{N}\sum_{i=1}^{N}\mathbf{1} _{(r_i^{N},K_i^N)}\Rightarrow I(r,k)$,
\item[ii)]$NR_{\max}^{N}\rightarrow \Lambda^{-1}$,
\item[iii)]$\frac{M}{N}\rightarrow \gamma$.
\end{itemize}
\end{assumption}

Now we state the main theorem in this section that proves the convergence of empirical process to its mean field limit.
\begin{theorem}[Functional Law of Large Numbers]\label{fluid_limit}
Let $|\cdot |$ denote the Euclidean norm in $\mathbb{R}^{K+1}$. Under Assumption~\ref{assumption}, suppose that $Y_{0}^{N}\xrightarrow{p} y_0$,  then we have for any $\epsilon>0$ and $t_0>0$,
$$\lim_{N\rightarrow \infty}P\left(\sup_{t\leq t_0}|Y_t^N-y_t|>\epsilon\right)=0.$$
Here $y_t=(y_t(0),\cdots,y_t(K))$, where $y_{t}(k)=\int_{0}^{1}dy_{t}(r,k)$ for $k=0,\cdots,K$. And $y_t$ is the unique solution to the following differential equation starting at $y_0$
\begin{equation}\label{diff_eqn}
\shortdot{y}_{t}=b(y_{t})
\end{equation}

where $b:[0,1]^{K+1}\rightarrow \mathbb{R}^{K+1}$ is a vector field satisfies
\begin{eqnarray}\label{eqn:b}
b(y_{t})&=&\iint\limits_{]0,1]\times [0,...,K]}\left[ \frac{\Lambda}{r}(\mathbf{1}_{(r,n-1)}-\mathbf{1}_{(r,n)})\mathbf{1}_{n>0}+\left(\gamma-\sum_{n}\int_{0}^{1} ndy_{t}(r,n) \right)(\mathbf{1}_{(r,n+1)}-\mathbf{1}_{(r,n)})\mathbf{1}_{n<K}\right] dy_{t}(r,n),\nonumber \\
\end{eqnarray}
or componentwise
$$b(y_{t})(0)=\underbrace{-\int_{0}^{1}\left(\gamma-\sum_{n}\int_{0}^{1} ndy_{t}(r,n) \right)dy_{t}(r,0)}_{\text{return a bike to a no-bike station}}+\underbrace{\int_{0}^{1}\frac{\Lambda}{r}dy_{t}(r,1)}_{\text{retrieve a bike from a 1-bike station}},$$

\begin{equation*}
\begin{split}
b(y_{t})(k)=& \underbrace{\int_{0}^{1}\frac{\Lambda}{r}dy_{t}(r,k+1)}_{\text{retrieve a bike from a $k+1$-bike station}}-\underbrace{\int_{0}^{1}\left(\frac{\Lambda}{r}+\gamma-\sum_{n}\int_{0}^{1} ndy_{t}(r,n) \right)dy_{t}(r,k)}_{\text{retrieve and return a bike to a $k$-bike station}}\\
&+\underbrace{\int_{0}^{1}\left(\gamma-\sum_{n}\int_{0}^{1} ndy_{t}(r,n) \right)dy_{t}(r,k-1)}_{\text{return a bike to a $k-1$-bike station}},
\end{split}
\end{equation*}

for $ k=1,...,K-1$, and 
$$b(y_{t})(K)=\underbrace{-\int_{0}^{1}\frac{\Lambda}{r} dy_{t}(r,K)}_{\text{retrieve a bike from a $K$-bike station}}+\underbrace{\int_{0}^{1}\left(\gamma-\sum_{n}\int_{0}^{1} ndy_{t}(r,n) \right)dy_{t}(r,K-1)}_{\text{return a bike to a $K-1$-bike station}}.$$
\end{theorem}

\begin{proof}
A similar theorem is given in the paper of \citet{fricker2012mean}, however, a proof is not given in their work.  Thus, to make our paper self contained, we provide a full proof of the mean field limit for the convenience of the reader as it is essential for our future results.  Our proof exploits Doob's inequality for martingales and Gronwall's lemma. Moreover, we use Proposition~\ref{bound}, Proposition~\ref{Lipschitz}, and Proposition~\ref{drift} in the proof, and they are stated after the proof of Theorem~\ref{fluid_limit}.

Since $Y_{t}^{N}$ is a semi-martingale, we have the following decomposition of $Y_{t}^{N}$ ,
\begin{equation}\label{semiY}
Y_{t}^{N}=\underbrace{Y_0^{N}}_{\text{initial condition}}+\underbrace{M_t^{N}}_{\text{martingale}}+\int_{0}^{t}\underbrace{\beta(Y_s^{N})}_{\text{drift term}}ds
\end{equation}
where
$Y_{0}^{N}$ is the initial condition and $M_{t}^{N}$ is a family of martingales.  Moreover, $\int_{0}^{t}\beta(Y_s^{N})ds$ is the integral of the drift term where the drift term is given by $\beta: [0,1]^{K+1}\rightarrow \mathbb{R}^{K+1}$ or
\begin{equation*}
\begin{split}
\beta(y)&=\sum_{x\neq y}(x-y)Q(y,x)\\
&= \sum_{n,r}\left[\frac{1}{rNR_{\max}}(\mathbf{1}_{(r,n-1)}-\mathbf{1}_{(r,n)})\mathbf{1}_{n>0}+\left(\frac{M}{N}-\sum_{n'}\sum_{r'}n'y(r',n')\right)(\mathbf{1}_{(r,n+1)}-\mathbf{1}_{(r,n)})\mathbf{1}_{n<K} \right]y(r,n).
\end{split}
\end{equation*}

We want to compare the empirical measure $Y_{t}^{N}$ with the mean field limit $y_{t}$ defined by
\begin{equation}
y_t=y_0+\int_{0}^{t}b(y_s)ds.
\end{equation}

The remaining of the proof of this theorem can be found in the Appendix.

\end{proof}

\begin{proposition}[Bounding martingales]\label{bound}
For any stopping time $T$ such that $\mathbb{E}(T)<\infty$, we have
\begin{equation}
\mathbb{E}\left(\sup_{t\leq T}|M_{t}^{N}|^2\right)\leq 4\mathbb{E}\int_{0}^{T}\alpha(Y_{t}^{N})dt.
\end{equation}  
\begin{proof}
The proof is found in the Appendix.  
\end{proof}
\end{proposition}

\begin{proposition}[Asymptotic Drift is Lipschitz]\label{Lipschitz}
The drift function $b(y)$ given in Equation (\ref{eqn:b}) is a Lipschitz function with respect to the Euclidean norm in $\mathbb{R}^{K+1}$. 
\begin{proof}
The proof is found in the Appendix.  
\end{proof}
\end{proposition}

\begin{proposition}[Drift is Asymptotically Close to a Lipschitz Drift]\label{drift}
Under Assumption~\ref{assumption}, we have for any $\epsilon>0$ and $s\geq 0$,
$$\lim_{N\rightarrow \infty}\mathbb{P}(|\beta(Y_s^{N})-b(Y_s^{N})|>\epsilon)= 0.$$
\begin{proof}
The proof is found in the Appendix.  
\end{proof}
\end{proposition}

We have proved mean field limit for our bike sharing model.  Our analysis yields that as the number of stations goes towards infinity, we can solve a set of ordinary differential equations to obtain important performance measure information. The performance measures that we can approximate are the mean proportion of empty or saturated stations, and the average number of bikes in circulation.  Moreover, we can analyze how factors such as fleet size and capacity change the value of the performance measures.  

However, just knowing the mean field limit is not enough.  One reason is that we would like to know more about the stochastic variability of the system, i.e. the fluctuations around the mean field limit.  The mean field limit cannot explain the stochastic fluctuations of the BSS and therefore, we need to analyze the BSS in a different way. Thus, in the subsequent section we develop a functional central limit theorem for our bike sharing model, and explain why it is important for understanding stochastic fluctuations of bike sharing networks.


\section{Diffusion Limit}\label{Diffusion_Limit}
In this section, we derive the diffusion limit of our stochastic empirical process bike sharing model.  Diffusion limits are critical for obtaining a deep understanding of the sample path behavior of stochastic processes.  One reason is that diffusion limits describe the fluctuations around the mean field limit and can help understand the variance or the asymptotic distribution of the stochastic process being analyzed.  We define our diffusion scaled bike sharing model by subtracting the mean field limit from the empirical measure process and rescaling it by $\sqrt{N}$.  Thus, we obtain the following expression for the diffusion scaled bike sharing empirical process
\begin{equation}
D_{t}^{N}=\sqrt{N}(Y_{t}^{N}-y_{t}) .
\end{equation}

Unlike many other ride-sharing systems such as Lyft or Uber, bike sharing programs cannot use pricing as a mechanism for redistributing bikes to satisfy demand in real-time.  For this reason, it is essential to understand the dynamics and behavior of $D_{t}^{N}$.  $D_{t}^{N}$ can be useful for describing the probability that the proportion of stations with $i$ bikes exceeds a threshold i.e. $\mathbb{P}(Y^N_t > x)$ for some $x \in [0,1]^{K+1}$.  It also describes this probability in a situation where there is no control or rebalancing of bikes in the system.  This knowledge of the uncontrolled system is especially important for newly-started bike sharing systems who are still in the process of gathering information about the system demand. The diffusion limit helps managers of BSS to understand the system dynamics and stability, which in turn helps them make short term and long term managerial decisions. It is also helpful in the case when the operators of the bike sharing system have no money for rebalancing the system to meet  real-time demand.

Using the semi-martingale decomposition of $Y^N_t$ given in Equation (\ref{semiY}), we can write a similar decomposition for $D_{t}^{N}$ as follows:
\begin{equation}
\begin{split}
D_{t}^{N}&=\sqrt{N}(Y_{0}^{N}-y_0)+\sqrt{N}M_{t}^{N}+\int_{0}^{t}\sqrt{N}[\beta(Y_{s}^{N})-b(y_s)]ds\\
&=D_{0}^{N}+\sqrt{N}M_{t}^{N}+\int_{0}^{t}\sqrt{N}[\beta(Y_{s}^{N})-b(Y_{s}^{N})]ds+\int_{0}^{t}\sqrt{N}[b(Y_{s}^{N})-b(y_s)]ds.
\end{split}
\end{equation}

Define 
\begin{equation}\label{sde}
D_{t}=D_{0}+\int_{0}^{t}b'(y_s)D_s ds+M_t
\end{equation}
where $b'(y)=\left(\frac{\partial b(y)(i)}{\partial y(j)}\right)_{ij}\in \mathbb{R}^{(K+1)\times (K+1)}$ and $M_{t}=(M_{t}(0),\cdots,M_{t}(K))\in \mathbb{R}^{K+1}$ is a real continuous centered Gaussian martingale, with Doob-Meyer brackets given by
\begin{eqnarray}
\boldlangle M(k) \boldrangle_{t}&=&\int_{0}^{t}(b_{+}(y_{s})(k)+b_{-}(y_{s})(k))ds,\nonumber\\
\boldlangle M(k),M(k+1) \boldrangle_{t}&=&-\int_{0}^{t}\left[\int_{0}^{1}\frac{\Lambda}{r}dy_{s}(r,k+1)\right. \nonumber\\
& &+ \left.\int_{0}^{1}\left(\gamma-\sum_{n}\int_{0}^{1}ndy_{s}(r,n)\right) dy_{s}(r,k)\right]ds \quad \text{for } k<K,\nonumber \\
\boldlangle M(k),M(j) \boldrangle_{t}&=&0\quad \text{for } |k-j|>1.
\end{eqnarray}
Here $b_{+}(y)=\max(b(y),0)$ and $b_{-}(y)=-\min(b(y),0)$ denote the positive and the negative parts of function $b(y)$ respectively.

Now we state the functional central limit theorem for the empirical measure process as follows,
\begin{theorem}[Functional Central Limit Theorem]\label{difftheorem}
Consider $D_{t}^{N}$ in $\mathbb{D}(\mathbb{R}_{+},\mathbb{R}^{K+1})$ with the Skorokhod $J_{1}$ topology, and suppose that 
\begin{itemize}
\item[1)]
$\limsup_{N\rightarrow \infty}\sqrt{N}\left(\min_{i}\lambda_{i}^{N}-\Lambda\right)<\infty,$
\item[2)]
$\limsup_{N\rightarrow \infty}\sqrt{N}\left(\frac{M}{N}-\gamma\right)< \infty.$
\end{itemize}
Then if $D_{0}^{N}$ converges in distribution to $D_{0}$, then $D_{t}^{N}$ converges to the unique OU process solving $D_{t}=D_{0}+\int_{0}^{t}b'(y_s)D_s ds+M_t$ in distribution.

\end{theorem}

To prove Theorem~\ref{difftheorem}, we take the following 4 steps,

\begin{itemize}
\item[1).]$\sqrt{N}M_{t}^{N}$ is a family of martingales independent of $D_{0}^{N}$ with Doob-Meyer brackets given by
\begin{eqnarray}
\boldlangle \sqrt{N}M^{N}(k)\boldrangle_{t}&=&\int_{0}^{t}(\beta_{+}(Y_{s}^{N})(k)+\beta_{-}(Y_{s}^{N})(k))ds,\nonumber\\
\boldlangle \sqrt{N}M^{N}(k),\sqrt{N}M^{N}(k+1)\boldrangle_{t}&=&-\int_{0}^{t}\sum_{r}\left[\frac{1}{rNR_{\max}}Y_{s}^{N}(r,k+1)\right.\nonumber\\
& & \left.+\left(\frac{M}{N}-\sum_{n'}\sum_{r'}n'Y_{s}^{N}(r',n')\right)Y_{s}^{N}(r,k)\right] ds \quad \text{for } k<K,\nonumber\\
\boldlangle \sqrt{N}M^{N}(k),\sqrt{N}M^{N}(j)\boldrangle_{t}&=&0\quad \text{for } |k-j|>1.
\end{eqnarray}

\item[2).]For any $T\geq 0$, $$\limsup_{N\rightarrow \infty}\mathbb{E}(|D_{0}^{N}|^2)<\infty \Rightarrow \limsup_{N\rightarrow \infty}\mathbb{E}(\sup_{0\leq t\leq T}|D_{t}^{N}|^2)<\infty. $$
\item[3).] If $(D_{0}^{N})_{N=1}^{\infty}$ is tight then $(D^{N})_{N=1}^{\infty}$ is tight and its limit points are continuous.
\item[4).]	If $D_{0}^{N}$ converges to $D_{0}$ in distribution, then $D_{t}^{N}$ converges to the unique OU process solving $D_{t}=D_{0}+\int_{0}^{t}b'(y_s)D_s ds+M_t$ in distribution.

\end{itemize}

\begin{lemma}\label{martingale-brackets}
$\sqrt{N}M_{t}^{N}$ is a family of martingales independent of $D_{0}^{N}$ with Doob-Meyer brackets given by
\begin{eqnarray}
\boldlangle \sqrt{N}M^{N}(k)\boldrangle_{t}&=&\int_{0}^{t}(\beta_{+}(Y_{s}^{N})(k)+\beta_{-}(Y_{s}^{N})(k))ds,\nonumber\\
\boldlangle \sqrt{N}M^{N}(k),\sqrt{N}M^{N}(k+1)\boldrangle_{t}&=&-\int_{0}^{t}\sum_{r}\left[\frac{1}{rNR_{\max}}Y_{s}^{N}(r,k+1)\right.\nonumber\\
& & \left.+\left(\frac{M}{N}-\sum_{n'}\sum_{r'}n'Y_{s}^{N}(r',n')\right)Y_{s}^{N}(r,k)\right] ds \quad \text{for } k<K,\nonumber\\
\boldlangle \sqrt{N}M^{N}(k),\sqrt{N}M^{N}(j)\boldrangle_{t}&=&0\quad \text{for } |k-j|>1.
\end{eqnarray}
\end{lemma}
\begin{proof}
The proof is found in the Appendix. 
\end{proof}

\begin{proposition}\label{driftbound}
For any $s\geq 0$,
\begin{equation}
\limsup_{N\rightarrow \infty}\sqrt{N}\left|\beta(Y_{s}^{N})-b(Y_{s}^{N})\right|<\infty.
\end{equation}
\begin{proof}
The proof is found in the Appendix. 
\end{proof}
\end{proposition}

\begin{lemma}[Finite Horizon Bound]\label{L2bound}
For any $T\geq 0$, if $$\limsup_{N\rightarrow \infty}\mathbb{E}\left(|D_{0}^{N}|^2 \right) < \infty ,$$ then we have $$\limsup_{N\rightarrow \infty}\mathbb{E}\left(\sup_{0\leq t\leq T}|D_{t}^{N}|^2 \right) < \infty .$$
\end{lemma}

\begin{proof}
The proof is found in the Appendix. 
\end{proof}

\begin{lemma}\label{tightness}
If $(D_{0}^{N})_{N=1}^{\infty}$ is tight then $(D^{N})_{N=1}^{\infty}$ is tight and its limit points are continuous.
\end{lemma}
\begin{proof}
The proof is found in the Appendix. 

\end{proof}

\begin{proposition}\label{b'}

$b(y)$ is continuously differentiable with the derivatives 
$\frac{\partial b(y(r,i))}{\partial y(r,j)}$ as follows,
\begin{eqnarray}
\frac{\partial b(y(r,0))}{\partial y(r,j)}&=&j\cdot y(r,0)+\frac{\Lambda}{r}\mathbf{1}_{\{j=1\}}-\left(\gamma-\sum_{n=0}^{K}n\left(\sum_{r'}y(r',n)\right)\right)\mathbf{1}_{\{j=0\}},\nonumber
\\
\frac{\partial b(y(r,k))}{\partial y(r,j)}&=&j\cdot \left( y(k)-y(k-1)\right)+\frac{\Lambda}{r}\left(\mathbf{1}_{\{j=k+1\}}-\mathbf{1}_{\{j=k\}}\right)\nonumber \\
& &+\left(\gamma-\sum_{n=0}^{K}n\left(\sum_{r'}y(r',n)\right)\right)\left(\mathbf{1}_{\{j=k-1\}}-\mathbf{1}_{\{j=k\}}\right)\quad \text{for } 0<k<K,\nonumber
\\
\frac{\partial b(y(r,K))}{\partial y(r,j)}&=&-j\cdot y(K-1)-\frac{\Lambda}{r}\mathbf{1}_{\{j=K\}}+\left(\gamma-\sum_{n=0}^{K}n\left(\sum_{r'}y(r',n)\right)\right)\mathbf{1}_{\{j=K-1\}}.\nonumber\\
\end{eqnarray}
\end{proposition}

\begin{proof}
The above equations can be obtained by directly taking derivatives to Equation (\ref{eqn:b}). We can see that the derivatives of $b(y)$ are linear in $y$. Thus we can conclude that $b(y)$ is continuously differentiable with respect to $y$.
\end{proof}

\begin{proof}[Proof of Theorem \ref{difftheorem}]
By Theorem 4.1 in Chapter 7 of \citet{Ethier2009},  it suffices to prove that the following condition holds
$$\sup_{t\leq T}\left|\int_{0}^{t}\left\{\sqrt{N}[b(Y_s^N)-b(y_s)]-b'(y_s)D_{s}^{N}\right\}ds\right|\xrightarrow{p} 0.$$
By Proposition~\ref{b'}, we know that $b(y_t)$ is continuously differentiable with respect to $y_t$. By the mean value theorem, for every $0\leq s\leq t$ there exists a vector $Z_{s}^{N}$ in between $Y_{s}^{N}$ and $y_s$ such that
$$b(Y_{s}^{N})-b(y_s)=b'(Z_{s}^{N})(Y_{s}^{N}-y_s).$$
Therefore
$$\int_{0}^{t}\left\{\sqrt{N}[b(Y_s^N)-b(y_s)]-b'(y_s)D_{s}^{N}\right\}ds=\int_{0}^{t}[b'(Z_{s}^{N})-b'(y_s)]D_{s}^{N}ds.$$
We know that $$\lim_{N\rightarrow \infty}\sup_{t\leq T}|b'(Z_{s}^{N})-b'(y_s)|=0\quad \text{in probability}$$
by the mean field limit convergence and the uniform continuity of $b'$.
By applying Chebyshev inequality we have   that $D_{s}^{N}$ is bounded in probability, then by Lemma 5.6 in \citet{ko2016strong},
$$\sup_{t\leq T}\left|\int_{0}^{t}\left\{\sqrt{N}[b(Y_s^N)-b(y_s)]-b'(y_s)D_{s}^{N}\right\}ds\right|\xrightarrow{p} 0.$$
\end{proof}

\begin{theorem}[Solution of the OU Process]\label{dt_solution}
The SDE (\ref{sde}) has a unique solution 
\begin{equation}\label{dt}
D_{t}=e^{\int_{0}^{t}b'(y_s)ds}D_0+\int_{0}^{t}e^{\int_{s}^{t}b'(y_u)du}dM_s.
\end{equation}

Define $\mathcal{A}(t)=b'(y_t)$, $\mathcal{B}(t)=\left(\frac{d}{dt}\boldlangle M(i),M(j)\boldrangle_{t}  \right)_{ij}$, then the expectation $E(D_t)$ is
\begin{equation}
\mathbb{E}[D_t]=e^{\int_{0}^{t}\mathcal{A}(s)ds} \mathbb{E}[D_0],
\end{equation}
 and the covariance matrix $\Sigma (t)=\mathrm{Cov}[D_{t},D_{t}]$ is 
 \begin{equation}
 \Sigma (t)=e^{\int_{0}^{t}\mathcal{A}(s)ds}\Sigma(0)e^{\int_{0}^{t}\mathcal{A^\top}(s)ds}+\int_{0}^{t}e^{\int_{s}^{t}\mathcal{A}(u)du}\mathcal{B}(s)e^{\int_{s}^{t}\mathcal{A^\top}(u)du}ds.
 \end{equation}
Moreover, differentiation with respect to $t$ yields
\begin{equation}
\frac{d\mathbb{E}[D_t]}{dt}=\mathcal{A}(t)\mathbb{E}[D_t],
\end{equation}

\begin{equation}\label{ode_var}
\frac{d\Sigma(t)}{dt}=\Sigma(t)\mathcal{A}(t)^\top+\mathcal{A}(t)\Sigma(t)+\mathcal{B}(t).
\end{equation}

\begin{proof}
The proof is found in the Appendix. 
\end{proof}
\end{theorem}

\section{Extensions} \label{Ext}
In Section \ref{Sec_Bike_Model}, we assumed WLOG that the routing probabilities and capacities are uniform throughout all stations. If one views our statement as assumptions, then it seems like our model is limiting, however, we emphasize that these are not assumptions and that our analysis extends to the broader case of non-uniform routing probabilities and capacities.  Thus, in the non-uniform setting, we are still able to prove the same fluid and diffusion limits and more importantly \emph{reduce the dimensionality of the bike sharing network.}  

\subsection{Extensions to non-uniform routing probabilities  and capacities}\label{Ext-routing}

The first possible extension of our current model is to incorporate non-uniform	routing probabilities, and take into account the origin-destination pairs. That is, after a customer picks up a bike from station $i$, the probability that he will drop off at station $j$ is equal to $P_{ij}$. However, to preserve the Markovian property of the queueing process, we also need to track the number of bikes in use which originated from station $i$, denoted as $U^N_t(i)$, in addition to the empirical measure process $Y^N_t$. As the scale of the system $N$ goes to infinity, the dimensionality of the new queueing process $(Y^N_t, U^N_t)$ will also go to infinity, therefore we are back to the same infinite-dimensionality problem with $X^N(t)$, which is illustrated more in details in Section \ref{intractability}, and lose the benefit of finite-dimensionality we get by studying the empirical measure $Y^N_t$.

However, we can still extend our model to non-uniform routing probabilities and capacities without considering the origin-destination pairs. In particular, we assume the probability that a bike being dropped off at station $i$ is equals to $P_{i}$, and the capacity at station $i$ is equal to $K_i$. 

We define relative routing probability at station $i$ as 
\begin{equation}
p_i=\frac{P_i}{\max_{i}P_i}.
\end{equation}

We consider the empirical measure process $Y^N_{t}(r,n,p,k)$ that counts the proportion of stations that have $n$ bikes, relative utilization $r$, routing probability $p$, and capacity $k$,
\begin{equation}
Y^N_{t}(r,n,p,k)=\frac{1}{N}\sum_{i}^{N}\mathbf{1}\{X_i^{N}(t)=n, r_i^{N}=r, p_i^{N}=p,K_i^{N}=k\}.
\end{equation}

Conditioning on $Y_{t}^{N}(r,n,p,k)=y(r,n,p,k)$, the transition rates of $y$ are specified as follows:\\
When a customer arrives to a station with $n$ bikes,  relative utilization $r$, relative routing probability $p$, and capacity $k$, to retrieve a bike, the
proportion of stations having $n$ bikes goes down by $1/N$, the proportion of stations having $n-1$ bikes goes up by $1/N$, and the transition rate $Q^N$ is 
\begin{eqnarray}
Q^{N}\left(y,y+\frac{1}{N}(\mathbf{1}_{(r,n-1,p,k)}-\mathbf{1}_{(r,n,p,k)}) \right) &=& 
y(r,n,p,k)\lambda_r N \mathbf{1}_{n>0}\nonumber \\ 
&=& y(r,n,p,k)\frac{\mu P}{R}N \mathbf{1}_{n>0}\nonumber \\
&=& y(r,n,p,k)\frac{pP^{N}_{\max}}{rR^N_{\max}}N \mathbf{1}_{n>0}.
\end{eqnarray}
When a customer returns a bike to a station with $n$ bikes ,  relative utilization $r$, relative routing probability $p$, and capacity $k$, to retrieve a bike,, the proportion of stations having $n$ bikes goes down by $1/N$, the proportion of stations having $n+1$ bikes goes up by $1/N$, and the transition rate $Q^N$ is 
\begin{eqnarray}
Q^{N} \left(y,y+\frac{1}{N}(\mathbf{1}_{(r,n+1,p,k)}-\mathbf{1}_{(r,n,p,k)}) \right) &=&
 y(r,n,p,k) N\cdot \mu \cdot P\left(M-\sum_{n'}\sum_{r',p',k'}n'y(r',n',p',k')N\right)\mathbf{1}_{n<k}\nonumber  \\  &=&
 y(r,n,p,k)N^2pP^{N}_{\max}\left(\frac{M}{N}-\sum_{n'}\sum_{r',p',k'}n'y(r',n',p',k')\right)\mathbf{1}_{n<k}.\nonumber\\
\end{eqnarray}

We have the following functional forward equations for $Y_{t}^{N}(r,n,p,k)$.

\begin{proposition}\label{functional_forward_Yex_routing}
 For any integrable function $f: [0,1]^{k+1}\rightarrow \mathbb{R}$, and \\$Y_{t}^N(r,p,k)=(Y_{t}^{N}(r,0,p,k),\cdots,Y_{t}^{N}(r,k,p,k))$ satisfies the following functional forward equation,
 \begin{eqnarray}
 & &\updot{\mathbb{E}}(f(Y_{t}^{N}(r,p,k)) | Y_{0}^{N}(r,p,k) = y_{0}(r,p,k)]\nonumber \\
 &=&  \sum_{n=0}^{k}\mathbb{E}\left[\left(f\left(Y_t^N(r,p,k)+\frac{1}{N}(\mathbf{1}_{r,n-1,p,k}-\mathbf{1}_{r,n,p,k})\right)-f(Y_t^N(r,p,k))\right)Y^{N}_{t}(r,n,p,k)\frac{NpP^{N}_{\max}}{rR^{N}_{\max}}\mathbf{1}_{n>0}\right]\nonumber \\
 &+&\sum_{n=0}^{k}\mathbb{E}\left[\left(f\left(Y_t^N(r,p,k)+\frac{1}{N}(\mathbf{1}_{r,n+1,p,k}-\mathbf{1}_{r,n,p,k})\right)-f(Y_t^N(r,p,k))\right)Y^{N}_{t}(r,n,p,k)N^2pP^{N}_{\max}\left(\frac{M}{N}\right.\right.\nonumber\\
 &-&\left.\left.\sum_{n'}\sum_{r',p',k'}n'Y_{t}^{N}(r',n',p',k')\right)\mathbf{1}_{n<k}\right]\nonumber\\
 \end{eqnarray}
 \end{proposition}

\subsubsection{Mean Field Limit}
We first state the following assumptions that we use throughout this section,
\begin{assumption}\label{assumption_ex}
 There exists a probability measure $I(r,p,k)$  on $[0,1]\times [0,1] \times	 \mathbb{N}$ and $\Lambda > 0$ such that, as $N$ tends to infinity, we have
\begin{itemize}
\item[i)] $\frac{1}{N}\sum_{i=1}^{N}\mathbf{1} _{(r_i^{N},p_i^{N},K_i^N)}\Rightarrow I(r,p,k)$,
\item[ii)]$P^{N}_{\max}/R_{\max}^{N}\rightarrow \Lambda$, $NP^{N}_{\max}\rightarrow \mathcal{P}$, $\frac{M}{N}\rightarrow \gamma$,
\item[iii)]The set $\mathcal{K}=\bigcup_{N=1}^{\infty}\{K_{i}^{N},i=1,...,N\}$ is finite.
\end{itemize}
\end{assumption}

Now we state the main theorem in this section that proves the convergence of empirical process to its mean field limit.
\begin{theorem}\label{fluid_limit_ex}
Let $|\cdot |$ denote the Euclidean norm in $\mathbb{R}^{K_{\max}+1}$. Under Assumption~\ref{assumption_ex}, suppose that $Y_{0}^{N}\xrightarrow{p} y_0$,  then we have for any $\epsilon>0$ and $t_0>0$,
$$\lim_{N\rightarrow \infty}\mathbb{P}\left(\sup_{t\leq t_0}|Y_t^N-y_t|>\epsilon\right)=0.$$
Here $y_t=(y_t(0),\cdots,y_t(K_{\max}))$, where $y_{t}(j)=\sum_{k\in \mathcal{K}}\int_{0}^{1}\int _{0}^{1}dy_{t}(r,j,p,k)$ for $j=0,\cdots,K_{\max}$. And $y_t$ is the unique solution to the following differential equation starting at $y_0$
\begin{equation}\label{fluid_eqn_ex}
\shortdot{y}_{t}=b(y_{t})
\end{equation}

where $b:[0,1]^{K+1}\rightarrow \mathbb{R}^{K_{\max}+1}$ is a vector field satisfies
\begin{eqnarray}\label{eqn:b_ex}
b(y_{t})&=&\sum_{k\in \mathcal{K}}\sum_{n=0}^{k}\iint\limits_{]0,1]\times [0,1]}\left[ \frac{p\Lambda}{r}(\mathbf{1}_{(r,n-1,p,k)}-\mathbf{1}_{(r,n,p,k)})\mathbf{1}_{n>0}+\right.\nonumber\\
& &\left.p\mathcal{P}\left(\gamma-\sum_{n}\sum_{k\geq n}\iint_{]0,1]\times[0,1]} ndy_{t}(r,n,p,k) \right)(\mathbf{1}_{(r,n+1,p,k)}-\mathbf{1}_{(r,n,p,k)})\mathbf{1}_{n<k}\right] dy_{t}(r,n,p,k),\nonumber \\
\end{eqnarray}
or componentwise

\begin{eqnarray}
b(y_{t})(0)&=&\underbrace{-\sum_{k\in \mathcal{K}}\iint\limits_{]0,1]\times [0,1]}p\mathcal{P}\left(\gamma-\sum_{n}\sum_{k\in \mathcal{K}}\iint\limits_{]0,1]\times [0,1]} ndy_{t}(r,n,p,k) \right)dy_{t}(r,0,p,k)}_{\text{return a bike to a no-bike station}}\nonumber\\
& &+\underbrace{\sum_{k\in \mathcal{K}}\iint\limits_{]0,1]\times [0,1]}\frac{p\Lambda}{r}dy_{t}(r,1,p,k)}_{\text{retrieve a bike from a 1-bike station}},
\end{eqnarray}

\begin{equation}
\begin{split}
b(y_{t})(j)=& \underbrace{\sum_{k\in \mathcal{K}}\iint\limits_{]0,1]\times [0,1]}\frac{p\Lambda}{r}dy_{t}(r,j+1,p,k)}_{\text{retrieve a bike from a $j+1$-bike station}}\\
& -\underbrace{\sum_{k\in \mathcal{K}}\iint\limits_{]0,1]\times [0,1]}\left(\frac{p\Lambda}{r}+p\mathcal{P}\left(\gamma-\sum_{n}\sum_{k\in \mathcal{K}}\iint\limits_{]0,1]\times [0,1]} ndy_{t}(r,n,p,k) \right)\right)dy_{t}(r,j,p,k)}_{\text{retrieve and return a bike to a $j$-bike station}}\\
&+\underbrace{\sum_{k\in \mathcal{K}}\iint\limits_{]0,1]\times [0,1]}p\mathcal{P}\left(\gamma-\sum_{n}\sum_{k\in \mathcal{K}}\iint\limits_{]0,1]\times [0,1]} ndy_{t}(r,n,p,k) \right)dy_{t}(r,j-1,p,k)}_{\text{return a bike to a $j-1$-bike station}},
\end{split}
\end{equation}

for $ j=1,...,K_{\max}-1$, and 
\begin{eqnarray}
b(y_{t})(K_{\max})&=&\underbrace{-\sum_{k\in \mathcal{K}}\iint\limits_{]0,1]\times [0,1]}\frac{p\Lambda}{r} dy_{t}(r,K_{\max},p,k)}_{\text{retrieve a bike from a $K_{\max}$-bike station}}\nonumber\\
& &+\underbrace{\sum_{k\in \mathcal{K}}\iint\limits_{]0,1]\times [0,1]}p\mathcal{P}\left(\gamma-\sum_{n}\sum_{k\in \mathcal{K}}\iint\limits_{]0,1]\times [0,1]} ndy_{t}(r,n,p,k) \right)dy_{t}(r,K_{\max}-1,p,k)}_{\text{return a bike to a $K_{\max}-1$-bike station}}.\nonumber\\
\end{eqnarray}
\end{theorem}
\begin{proof}
The only things we need to prove in this extension case are that $b(y)$ is Lipschitz and that $\beta(Y_{t}^{N})\xrightarrow{p} b(Y_{t}^{N})$ for any $t\geq 0$. The rest of the proof stays the same as in Section \ref{Fluid_Limit}.
\end{proof}

\begin{proposition}[Asymptotic Drift is Lipschitz]\label{Lipschitz_ex}
The drift function $b(y)$ given in Equation (\ref{eqn:b_ex}) is a Lipschitz function with respect to the Euclidean norm in $\mathbb{R}^{K_{\max}+1}$. 
\begin{proof}
The proof is found in the Appendix.  
\end{proof}
\end{proposition}

\begin{proposition}[Drift is Asymptotically Close to a Lipschitz Drift]\label{driftbound_ex}
Under Assumption~\ref{assumption_ex}, we have for any $\epsilon>0$ and $s\geq 0$,
$$\lim_{N\rightarrow \infty}\mathbb{P}(|\beta(Y_s^{N})-b(Y_s^{N})|>\epsilon)= 0.$$
\begin{proof}
The proof is found in the Appendix.  
\end{proof}
\end{proposition}
\subsubsection{Diffusion Limit}
Now we state the functional central limit theorem for the empirical measure process in the extension case as follows,
\begin{theorem}\label{difftheorem_ex}
Consider $D_{t}^{N}$ in $\mathbb{D}(\mathbb{R}_{+},\mathbb{R}^{K_{\max}+1})$ with the Skorokhod $J_{1}$ topology, and suppose that 
\begin{itemize}
\item[1)]
$\limsup_{N\rightarrow \infty}\sqrt{N}\left(\frac{P_{\max}^{N}}{R_{\max}^{N}}-\Lambda\right)<\infty,$
\item[2)]
$\limsup_{N\rightarrow \infty}\sqrt{N}\left(\frac{M}{N}-\gamma\right)< \infty,$
\item[3)] $\limsup_{N\rightarrow \infty}\sqrt{N}\left(NP_{\max}^{N}-\mathcal{P}\right)< \infty.$
\end{itemize}
Then if $D_{0}^{N}$ converges in distribution to $D_{0}$, then $D_{t}^{N}$ converges to the unique OU process solving $D_{t}=D_{0}+\int_{0}^{t}b'(y_s)D_s ds+M_t$ in distribution.
Here $b'(y)=\left(\frac{\partial b(y)(i)}{\partial y(j)}\right)_{ij}\in \mathbb{R}^{(K_{\max}+1)\times (K_{\max}+1)}$ and $M_{t}=(M_{t}(0),\cdots,M_{t}(K_{\max}))\in \mathbb{R}^{K_{\max}+1}$ is a real continuous centered Gaussian martingale, with Doob-Meyer brackets given by
\begin{eqnarray}
\boldlangle M(k) \boldrangle_{t}&=&\int_{0}^{t}(b_{+}(y_{s})(k)+b_{-}(y_{s})(k))ds,\nonumber\\
\boldlangle M(k),M(k+1) \boldrangle_{t}&=&-\int_{0}^{t}\left[\sum_{K\in \mathcal{K}}\iint_{]0,1]\times [0,1]}\frac{p\Lambda}{r}dy_{s}(r,k+1,p,K)\right. \nonumber\\
& &+ \left.\sum_{K\in \mathcal{K}}\iint_{]0,1]\times [0,1]}p\mathcal{P}\left(\gamma-\sum_{n}\sum_{K\in \mathcal{K}}\iint_{]0,1]\times [0,1]}ndy_{s}(r,n,p,K)\right) dy_{s}(r,k,p,K)\right]ds \nonumber\\
& &\quad \text{for } k<K_{\max},\nonumber \\
\boldlangle M(k),M(j) \boldrangle_{t}&=&0\quad \text{for } |k-j|>1.
\end{eqnarray}
Here $b_{+}(y)=\max(b(y),0)$ and $b_{-}(y)=-\min(b(y),0)$ denote the positive and the negative parts of function $b(y)$ respectively.
\end{theorem}
\begin{proof}
The only things we need to prove in this extension case are listed as propositions below. The rest of the proof stays the same as in Section \ref{Diffusion_Limit}.
\end{proof}

\begin{lemma}\label{martingale_ex}
$\sqrt{N}M_{t}^{N}$ is a family of martingales independent of $D_{0}^{N}$ with Doob-Meyer brackets given by
\begin{eqnarray}
\boldlangle M(k) \boldrangle_{t}&=&\int_{0}^{t}(b_{+}(y_{s})(k)+b_{-}(y_{s})(k))ds,\nonumber\\
\boldlangle M(k),M(k+1) \boldrangle_{t}&=&-\int_{0}^{t}\left[\sum_{K\in \mathcal{K}}\iint_{]0,1]\times [0,1]}\frac{p\Lambda}{r}dy_{s}(r,k+1,p,K)\right. \nonumber\\
& &+ \left.\sum_{K\in \mathcal{K}}\iint_{]0,1]\times [0,1]}p\mathcal{P}\left(\gamma-\sum_{n}\sum_{K\in \mathcal{K}}\iint_{]0,1]\times [0,1]}ndy_{s}(r,n,p,K)\right) dy_{s}(r,k,p,K)\right]ds \nonumber\\
& &\quad \text{for } k<K_{\max},\nonumber \\
\boldlangle M(k),M(j) \boldrangle_{t}&=&0\quad \text{for } |k-j|>1.
\end{eqnarray}
\end{lemma}

\begin{proposition}\label{bbound_ex}
 For any $s\geq 0$, 
\begin{equation}
\limsup_{N\rightarrow \infty}\sqrt{N}|\beta(Y_{s}^{N}-b(Y_s^{N})|< \infty
\end{equation}
\begin{proof}
The proof is found in the Appendix.  
\end{proof}
\end{proposition}

\begin{proposition}\label{adjacent_ex}
For $k<K_{\max}$
\begin{equation}
\lim_{N\rightarrow \infty}\mathbb{P}\left(\sup_{t\leq T}\left|\boldlangle \sqrt{N}M^N(k),\sqrt{N}M^N(k+1)\boldrangle_{t}- \boldlangle M(k),M(k+1)\boldrangle_{t}\right|>\epsilon\right)=0
\end{equation}
\begin{proof}
The proof is found in the Appendix.  
\end{proof}
\end{proposition}

\begin{proposition}\label{b'_ex}

$b(y)$ is continously differentiable with the derivatives 
$\frac{\partial b(y(r,i,p,K))}{\partial y(r,j,p,K)}$ as follows,
\begin{eqnarray}
\frac{\partial b(y(r,0,p,K))}{\partial y(r,j,p,K)}&=&p\mathcal{P}j\cdot y(r,0,p,K)+\frac{p\Lambda}{r}\mathbf{1}_{\{j=1\}}-p\mathcal{P}\left(\gamma-\sum_{n=0}^{K_{\max}}n\left(\sum_{r',p',K'}y(r',n,p',K')\right)\right)\mathbf{1}_{\{j=0\}},\nonumber
\\
\frac{\partial b(y(r,k,p,K))}{\partial y(r,j,p,K)}&=&p\mathcal{P} j\cdot \left( y(r,k,p,K)-y(r,k-1,p,K)\right)+\frac{p\Lambda}{r}\left(\mathbf{1}_{\{j=k+1\}}-\mathbf{1}_{\{j=k\}}\right)\nonumber \\
& &+p\mathcal{P}\left(\gamma-\sum_{n=0}^{K_{\max}}n\left(\sum_{r',p',K'}y(r',n,p',K')\right)\right)\left(\mathbf{1}_{\{j=k-1\}}-\mathbf{1}_{\{j=k\}}\right)\quad \text{for } 0<k<K_{\max},\nonumber
\\
\frac{\partial b(y(r,K,p,K))}{\partial y(r,j,p,K)}&=&-p\mathcal{P} j\cdot y(K_{\max}-1)-\frac{p\Lambda}{r}\mathbf{1}_{\{j=K_{\max}\}}\nonumber\\
& &+p\mathcal{P}\left(\gamma-\sum_{n=0}^{K_{\max}}n\left(\sum_{r',p',K'}y(r',n,p',K')\right)\right)\mathbf{1}_{\{j=K_{\max}-1\}}.\nonumber\\
\end{eqnarray}
\end{proposition}

\begin{proof}
The above equations can be obtained by directly taking derivatives to Equation (\ref{eqn:b_ex}). We can see that the derivatives of $b(y)$ are linear in $y$. Thus we can conclude that $b(y)$ is continuously differentiable with respect to $y$.
\end{proof}

\subsection{Extensions to adding bike repositioning}\label{Ext-repositioning}
We can also extend our model to incorporate repositioning of bikes, which is adopted by most bike-sharing companies to rebalance the network over time. To avoid cumbersome notations, we restrict ourselves to the uniform routing probabilities and capacities scenario for the model discussed in this section.

We may assume that "rebalancers" will arrive to the system according to a independent Poisson process with rate $\lambda_{R}Y^N_t(K)N$, where $\lambda_R>0$ is a constant and $Y^N_t(K)$ is the proportion of full stations at time $t$. They pick a full station uniformly at random among all full stations, and remove a certain number $C$ of the bikes from that station. 

In terms of repositioning these bikes, we assume that $C$ number of bikes are being put back to stations with bikes ranging from $0$ to $K-C$. The probability of repositioning to a station with $j$ bikes is equal to $\frac{Y^N_t(j)g_j}{\sum_{i=0}^{K-C}Y^N_t(i)g_i}$, where $\{g_i\}_i$ are positive constants and $g_0>>g_1>>g_2\cdots>>g_{K-C}>0$. This is a smoothed version of choosing the station with minimum number of bikes when repositioniong, and it preserves the Lipschitz property of the drift function, which is essential for deriving the fluid limit and the diffusion limit. We assume additionally that $M<N(K-C)$ to avoid the case where the denominator $\sum_{i=0}^{K-C}Y^N_t(i)g_i$ can be 0. This can be easily shown through a proof by contradiction. For simplicity of the model, we assume repositioning time are exponentially distributed with rate $\mu_R$. Finally, we use $R^N_t$ to denote the number of bikes currently in repositioning by the rebalancers at time $t$.

Conditioning on $(Y_{t}^{N}(r,n),R^N_t)=(y(r,n),R)$, , we have the following transition rate for the four types of events:

When a customer arrives to a station with $n$ bikes,  relative utilization $r$, to retrieve a bikes, the proportion of stations having $n$ bikes goes down by $1/N$, the proportion of stations having $n-1$ bikes goes up by $1/N$, and the transition rate $Q^N$ is 
\begin{eqnarray}
Q^{N}\left((y,R),\left(y+\frac{1}{N}(\mathbf{1}_{(r,n-1)}-\mathbf{1}_{(r,n)}),R\right) \right) &=& 
\frac{y(r,n)}{rR_{\max}^N}\mathbf{1}_{n>0}.
\end{eqnarray}

When a customer returns a bike to a station with $n$ bikes and relative utilization $r$, the proportion of stations having $n$ bikes goes down by $1/N$, the proportion of stations having $n+1$ bikes goes up by $1/N$, and the transition rate $Q^N$ is 
\begin{eqnarray}
& &Q^{N}\left((y,R),\left(y+\frac{1}{N}(\mathbf{1}_{(r,n+1)}-\mathbf{1}_{(r,n)}),R\right) \right)= y(r,n) \cdot \mu \cdot \left(M-\sum_{n'}\sum_{r'}n'y(r',n')N-R\right)\mathbf{1}_{n<K}.\nonumber\\
\end{eqnarray}

When a rebalancer arrives to a station with $K$ bikes,  relative utilization $r$, to retrieve $C$ bikes, the proportion of stations having $K$ bikes goes down by $1/N$, the proportion of stations having $K-C$ bikes goes up by $1/N$, and the transition rate $Q^N$ is 
\begin{eqnarray}
Q^{N}\left((y,R),\left(y+\frac{1}{N}(\mathbf{1}_{(r,K-C)}-\mathbf{1}_{(r,K)}),R+C\right) \right) &=& 
y(r,K)\lambda_R N.
\end{eqnarray}

When a rebalancer arrives to a station with $j$ bikes and relative utilization $r$, where $0\leq j\leq K-C$, to put back $C$ bikes, the proportion of stations having $j$ bikes goes down by $1/N$, the proportion of stations having $j+C$ bikes goes up by $1/N$, and the transition rate $Q^N$ is 
\begin{eqnarray}
& &Q^{N}\left((y,R),\left(y+\frac{1}{N}(\mathbf{1}_{(r,j+C)}-\mathbf{1}_{(r,j)}),R-C\right) \right)=\mu_R \frac{y(r,j)g_j}{\sum_{i=0}^{K-C}y(r,i)g_i}\cdot \frac{R}{C}.
\end{eqnarray}

We have the following functional forward equations for $(Y_{t}^{N}(r,n),R)$.

\begin{proposition}\label{functional_forward_Yex_repositioning}
 For any integrable function $f: [0,1]^{K+1}\times [0,\infty)\rightarrow \mathbb{R}$, \\$Y_{t}^N(r)=(Y_{t}^{N}(r,0),\cdots,Y_{t}^{N}(r,K))$ satisfies the following functional forward equation,
 \begin{eqnarray}
 & &\updot{\mathbb{E}}(f(Y_{t}^{N}(r),R) | Y_{0}^{N}(r) = y_{0}(r), R_0^N=R_0]\nonumber \\
 &=&  \sum_{n=0}^{K}\mathbb{E}\left[\left(f\left(Y_t^N(r)+\frac{1}{N}(\mathbf{1}_{r,n-1}-\mathbf{1}_{r,n}),R_t^N\right)-f(Y_t^N(r),R_t^N)\right)\frac{Y^{N}_{t}(r,n)}{rR^{N}_{\max}}\mathbf{1}_{n>0}\right]\nonumber \\
 &+&\sum_{n=0}^{K}\mathbb{E}\left[\left(f\left(Y_t^N(r)+\frac{1}{N}(\mathbf{1}_{r,n+1}-\mathbf{1}_{r,n}),R_t^N\right)-f(Y_t^N(r)),R_t^N\right)Y^{N}_{t}(r,n)N\left(\frac{M}{N}\right.\right.\nonumber\\
 &-&\left.\left.\sum_{n'}\sum_{r'}n'Y_{t}^{N}(r',n')-\frac{R^N_t}{N}\right)\mathbf{1}_{n<K}\right]\nonumber\\
 &+&  \mathbb{E}\left[\left(f\left(Y_t^N(r)+\frac{1}{N}(\mathbf{1}_{r,K-C}-\mathbf{1}_{r,K}),R^N_t+C\right)-f(Y_t^N(r),R^N_t)\right)Y^{N}_{t}(r,K)\lambda_R N\right]\nonumber \\
 &+& \sum_{n=0}^{K-C}\mathbb{E}\left[\left(f\left(Y_t^N(r)+\frac{1}{N}(\mathbf{1}_{r,n+C}-\mathbf{1}_{r,n}),R^N_t-C\right)-f(Y_t^N(r),R^N_t)\right)\mu_R \frac{Y^N_t(r,n)g_n}{\sum_{i=0}^{K-C}Y^N_t(r,i)g_i}\cdot\frac{R^N_t}{C}\right].\nonumber\\
 \end{eqnarray}
 \end{proposition}

\subsubsection{Mean Field Limit}
Now we show the mean field limit result for the empirical process with rebalancing.
\begin{theorem}\label{fluid_limit}
Let $|\cdot|$ denote the Euclidean norm in $\mathbb{R}^{K+2}$. Under Assumption~\ref{assumption}, suppose that $(Y_{0}^{N},R_0^N/N)\xrightarrow{p} (y_0,\bar{r}_0)$ ,  then we have for any $\epsilon>0$ and $t_0>0$,
$$\lim_{N\rightarrow \infty}P\left(\sup_{t\leq t_0}|(Y_t^N,R^N_t/N)-(y_t,\bar{r}_t)|>\epsilon\right)=0.$$
Here $y_t=(y_t(0),\cdots,y_t(K))$, where $y_{t}(k)=\int_{0}^{1}dy_{t}(r,k)$ for $k=0,\cdots,K$. And $(y_t,\bar{r}_t)$ is the unique solution to the following differential equation starting at $(y_0,\bar{r}_0)$
\begin{eqnarray}\label{diff_eqn}
\shortdot{y}_{t}&=&b_1(y_{t},\bar{r}_t),\\
\shortdot{\bar{r}}_{t}&=&b_2(y_{t},\bar{r}_t),
\end{eqnarray}
where $b_1:[0,1]^{K+1}\rightarrow \mathbb{R}^{K+1}$ and $b_2:[0,\infty)\rightarrow \mathbb{R}$ are vector fields satisfy
\begin{eqnarray}\label{eqn:b}
& & b_1(y_{t},\bar{r}_t)\nonumber\\
&=&\iint\limits_{]0,1]\times [0,...,K]}\left[ \frac{\Lambda}{r}(\mathbf{1}_{(r,n-1)}-\mathbf{1}_{(r,n)})\mathbf{1}_{n>0}+\left(\gamma-\sum_{n}\int_{0}^{1} ndy_{t}(r,n)-\bar{r}_t \right)(\mathbf{1}_{(r,n+1)}-\mathbf{1}_{(r,n)})\mathbf{1}_{n<K}\right] dy_{t}(r,n)\nonumber \\
&+& \int\limits_{]0,1]}\lambda_R \mathbf{1}_{(r,K-C}-\mathbf{1}_{(r,K)})dy_{t}(r,K)+\iint\limits_{]0,1]\times [0,...,K-C]}\mu_R \frac{g_n}{\sum_{i=0}^{K-C}y_t(r,i)g_i}\cdot \frac{\bar{r}_t}{C} (\mathbf{1}_{(r,n+C)}-\mathbf{1}_{(r,n)}) dy_{t}(r,n),\nonumber \\
\end{eqnarray}
or componentwise
\begin{eqnarray}
b_1(y_{t},\bar{r}_t)(0)&=&\underbrace{-\int_{0}^{1}\left(\gamma-\sum_{n}\int_{0}^{1} ndy_{t}(r,n) -\bar{r}_t \right)dy_{t}(r,0)}_{\text{return a bike to a no-bike station}}+\underbrace{\int_{0}^{1}\frac{\Lambda}{r}dy_{t}(r,1)}_{\text{retrieve a bike from a 1-bike station}}\nonumber\\
& &\underbrace{- \int_{0}^{1}\mu_R \frac{g_0}{\sum_{i=0}^{K-C}y_t(r,i)g_i}\cdot \frac{\bar{r}_t}{C}dy_t(r,0)}_{\text{rebalancer put back C bikes to a 0-bike station}},
\end{eqnarray}
\begin{equation*}
\begin{split}
b_1(y_{t},\bar{r}_t)(k)=& \underbrace{\int_{0}^{1}\frac{\Lambda}{r}dy_{t}(r,k+1)}_{\text{retrieve a bike from a $k+1$-bike station}}-\underbrace{\int_{0}^{1}\left(\frac{\Lambda}{r}+\gamma-\sum_{n}\int_{0}^{1} ndy_{t}(r,n) -\bar{r}_t \right)dy_{t}(r,k)}_{\text{retrieve and return a bike to a $k$-bike station}}\\
&+\underbrace{\int_{0}^{1}\left(\gamma-\sum_{n}\int_{0}^{1} ndy_{t}(r,n) -\bar{r}_t \right)dy_{t}(r,k-1)}_{\text{return a bike to a $k-1$-bike station}}-\underbrace{\int_{0}^{1}\mu_R \frac{g_k}{\sum_{i=0}^{K-C}y_t(r,i)g_i}\cdot \frac{\bar{r}_t}{C}dy_t(r,k)}_{\text{rebalancer put back C bikes to a k-bike station}},
\end{split}
\end{equation*}
for $ k=1,...,K-C-1$, and 
\begin{equation*}
\begin{split}
b_1(y_{t},\bar{r}_t)(k)=& \underbrace{\int_{0}^{1}\frac{\Lambda}{r}dy_{t}(r,k+1)}_{\text{retrieve a bike from a $k+1$-bike station}}-\underbrace{\int_{0}^{1}\left(\frac{\Lambda}{r}+\gamma-\sum_{n}\int_{0}^{1} ndy_{t}(r,n) -\bar{r}_t \right)dy_{t}(r,k)}_{\text{retrieve and return a bike to a $k$-bike station}}\\
&+\underbrace{\int_{0}^{1}\left(\gamma-\sum_{n}\int_{0}^{1} ndy_{t}(r,n) -\bar{r}_t \right)dy_{t}(r,k-1)}_{\text{return a bike to a $k-1$-bike station}}+\underbrace{\int_{0}^{1}\mu_R \frac{g_{k-C}}{\sum_{i=0}^{K-C}y_t(r,i)g_i}\cdot \frac{\bar{r}_t}{C}dy_t(r,k-C)}_{\text{rebalancer put back C bikes to a k-C-bike station}}\\
& +\underbrace{\int_0^1 C\lambda_R dy_t(r,K)\mathbf{1}_{k=K-C}}_{\text{rebalancers retrieve C bikes from a $K$-bike station}},
\end{split}
\end{equation*}
for $ k=K-C,...,K-1$, and 
\begin{eqnarray}
b_1(y_{t},\bar{r}_t)(K)&=&\underbrace{-\int_{0}^{1}\frac{\Lambda}{r} dy_{t}(r,K)}_{\text{retrieve a bike from a $K$-bike station}}+\underbrace{\int_{0}^{1}\left(\gamma-\sum_{n}\int_{0}^{1} ndy_{t}(r,n) -\bar{r}_t \right)dy_{t}(r,K-1)}_{\text{return a bike to a $K-1$-bike station}}\nonumber\\
& &-\underbrace{\int_0^1 C\lambda_R dy_t(r,K)\mathbf{1}_{k=K-C}}_{\text{rebalancers retrieve C bikes from a $K$-bike station}},
\end{eqnarray}
and
\begin{eqnarray}
b_2(y_t,\bar{r}_t)=\underbrace{\int_0^1 C\lambda_R dy_t(r,K)}_{\text{rebalancers retrieve C bikes from a $K$-bike station}}-\sum_{n=0}^{K-C}\underbrace{\int_0^1\mu_R \frac{g_n}{\sum_{i=0}^{K-C}y_t(r,i)g_i}\bar{r}_t dy_t(r,n)}_{\text{rebalancers put back C bikes to a $n$-bike station}}.\nonumber\\
\end{eqnarray}
\end{theorem}
\begin{proof}
This result can be shown by similar techniques used in Section \ref{Fluid_Limit}.
\end{proof}
\section{Numerical Examples and Simulation}\label{simulation}
In this section,  we confirm our theoretical results of our bike sharing model with a stochastic simulation.   We perform simulations with both stationary  and non-stationary arrival rates, which are discussed substantially in the following subsections.

Figures~\ref{Fig_Avg_Trips} and~\ref{Fig_Avg_Duration} provide the user patterns of BSS from the historical data of Citi Bike. In Figure~\ref{Fig_Avg_Trips}, we can see that the arrival  rate of users to CitiBike stations varies not only on the time of the day, but also on the day of the week. So in the subsequent numerical examples, we consider both stationary and non-stationary arrival rates to provide insights for the behavior of such systems under different demand and usage conditions.  However, unlike the arrival rate, in Figure~\ref{Fig_Avg_Duration}, we observe that the mean travel duration of Citi Bike users does not vary significantly as a function of the time of day or the day of the week.  Therefore, in the subsequent numerical examples we assume a constant mean travel time $1/\mu$.

\begin{figure}[htbp]
\centering
\includegraphics[width=0.9\textwidth]{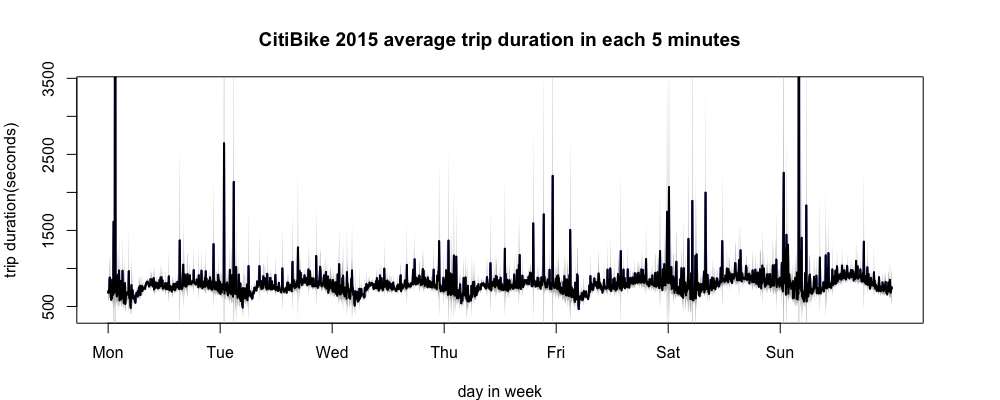}
\caption{Citi Bike average trip duration in each 5 minutes (Jan 1st-Dec 31st, 2015)}\label{Fig_Avg_Duration}
\end{figure}

\subsection{Simulation Experiments with a Non-stationary Arrival Rate}
In this section, we provide the results of the simulation studies of our stochastic bike sharing model with a non-stationary arrival process.  We perform our simulation with the following parameters: $N=100$, $\lambda(t)=1+0.5\sin(t/2)$, $\mu=1$. The number of sample paths we use in our simulation is 50. Other parameters are specified in the illustration of each figure below. For the time-varying arrival rate, we use a periodic function $\lambda(t)=1+0.5\sin(t/2)$ to mimic the patterns of morning rush and the afternoon-evening rush that we observe from the Citi Bike data (See Figure~\ref{Fig_Avg_Trips}).
\paragraph{Figure \ref{Fig_ykt}:}

 In Figure~\ref{Fig_ykt}, we simulate the different components of $Y_{t}^{N}$ and their subsequent mean field limits $y_{t}$ when $K=10$. Each station has 5 bikes at the beginning of the simulation. The solid lines represent the simulation results of different components of $Y_{t}^{N}$, and the dashed lines of the same color represent the corresponding mean field limit $y_{t}$. We also add the purple dashed line to represent the time-varying arrival rate. We observe that the mean field limit provides an accurate approximation to the empirical measure on each component.  We also notice that the empirical measure is indeed affected by the non-stationary arrival rate and is also non-stationary. 

\paragraph{Figures~\ref{Fig_y0t} - \ref{Fig_ykt(3)}:}

Figures~\ref{Fig_y0t} - \ref{Fig_y3t} show the simulation results of $Y^{N}_{t}$ with a 95\% confidence interval, vs. the mean field limit $y(t)$ with  two  standard deviations of the unscaled diffusion limit, when $K=3$. The initial distribution of bikes is given as $Y_{0}^{N}=(0,0.5,0.5,0)$. In Figure~\ref{Fig_y0t}, the black curve shows $Y_{t}^{N}(0)$, the simulated proportion of stations with no bikes over time. To produce the simulation, we take an average of 50 independent sample paths. The red dashed line shows $y_{t}(0)$, the mean field limit of the proportion of stations with no bikes over time, which is obtained from solving the system of ODEs in Theorem~\ref{fluid_limit}. The green curves show $2\sigma(Y_{t}^{N}(0))$ or two standard deviations of $Y_{t}^{N}(0)$, from the 50 sample paths in simulation. Lastly, the blue dashed lines show $2\sigma(D_{t})/N$ or two standard deviations of the unscaled diffusion limit, which is obtained from solving the system of ODEs from Theorem~\ref{dt_solution}  Equation (\ref{ode_var}). We also add a purple dashed line to represent the time-varying arrival rate, with black dashed vertical lines to indicate the peaks and valleys of the arrival rate and the corresponding mean field limit.  Again, we observe that the mean field limit and the diffusion limit provide an accurate approximation to the simulation results.  This accuracy serves to validate the correctness of the mean field and diffusion limits we proved earlier. Moreover, with a small time lag, the simulation results of $Y_{t}^{N}(0)$ are positively associated with movements of the arrival rate. Figures~\ref{Fig_y1t}, ~\ref{Fig_y2t}, and~\ref{Fig_y3t} show the dynamics of the proportion of stations with 1, 2, and 3 bikes respectively. Unlike the case of zero bikes, we observe that the dynamics are negatively associated with movements of the arrival rate.  To observe these dynamics in one graph, Figure~\ref{Fig_ykt(3)} combines Figures~\ref{Fig_y0t} -~\ref{Fig_y3t} in one graph, but only keeps the mean field limit curves.   This allows the reader to visualize the dynamics of the mean field empirical measure under a time-varying arrival rate.

\paragraph{Figure \ref{Fig_Vart}  :} Figure~\ref{Fig_Vart} shows the simulation results of the variance of $D_{t}^{N}$ vs. the numerical solution for the system of ODEs regarding the covariance matrix of $D_{t}$ (see Equation~(\ref{ode_var})) when $K=3$. The approximation of the variance of diffusion limit to the variance of the actual diffusion process is quite accurate. We also observe that as arrival rate increases, the variance of the proportion of stations with no bikes increases, while the proportion of stations with 1, 2, or 3 bikes decreases, and vice versa.

\paragraph{Figure~\ref{Fig_circ_NS}  :}  Figure~\ref{Fig_circ_NS} shows the the average number of bikes in circulation over time, which is denoted as 
\begin{equation}
C_t^{N}\triangleq M-\sum_{j=0}^{K}j\cdot Y_{t}^{N}(j)N.
\end{equation} 
We use $N\left(\gamma-\sum_{j=0}^{K}j\cdot y_{t}(j)\right)$ to approximate $\mathbb{E}[C_{t}^{N}]$, the expectation of the number of bikes in circulation. We use 
\begin{eqnarray}\label{circ_var}
& &\text{Var}\left[M-\sum_{j=0}^{K}j\cdot Y_{t}^{N}(j)N \right] \nonumber \\
&=&\text{Var}\left[\sum_{j=0}^{K}j\cdot Y_{t}^{N}(j)N \right]\nonumber\\
&=& N^2\left[\sum_{j=0}^{K}j^2 \cdot\text{Var}\left[Y^{N}_{t}(j)\right]+\sum_{i=0}^{K}\sum_{j=i+1}^{K}2ij\cdot \text{Cov}\left[Y^{N}_{t}(i),Y^{N}_{t}(j)\right]\right]\nonumber\\
&=& N^2\left[\sum_{j=0}^{K}j^2 \cdot\text{Var}\left[y_{t}(j)+\frac{1}{\sqrt{N}}D_{t}^{N}(j)\right]\right.\nonumber\\
& &\left.+\sum_{i=0}^{K}\sum_{j=i+1}^{K}2ij\cdot \text{Cov}\left[y_{t}(i)+\frac{1}{\sqrt{N}}D_{t}^{N}(i),y_{t}(j)+\frac{1}{\sqrt{N}}D_{t}^{N}(j)\right]\right]\nonumber\\
&\approx & N\left[\sum_{j=0}^{K}j^2 \cdot\text{Var}[D_{t}(j)]+\sum_{i=0}^{K}\sum_{j=i+1}^{K}2ij\cdot \text{Cov}[D_{t}(i),D_{t}(j)]\right]\nonumber\\
\end{eqnarray}
to approximate $\mathrm{Var}[C^{N}_{t}]$, the variance of the number of bikes in circulation. The black curve shows the simulated result of the average number of bikes in circulation.  The red dashed line is the mean field limit of the number of bikes in circulation over time, from solving the system of ODEs obtained by Theorem~\ref{fluid_limit}. The green curves are $2\sigma(C_{t}^{N})$, or two standard deviations from the mean of the number of bikes in circulation we obtain by simulating 50 independent sample paths of the stochastic bike sharing model. Lastly, the blue dashed lines show  our approximation of  $2\sigma(C_{t}^{N})$  using the diffusion limit (see Equation (\ref{circ_var})), which can be computed from solving the system of ODEs from Theorem~\ref{dt_solution} Equation (\ref{ode_var}). We also add a purple dashed line to represent the time-varing arrival rate, with black dashed vertical lines to indicate the peaks and valleys of the arrival rate and the corresponding mean field limit. Again, we can see that the mean field limit and the diffusion limit provide a high quality approximation for the actual dynamics of the average number of bikes in circulation.  Moreover, we observe two important things.  First, the average number of bikes in circulation is also time varying and fluctuates between 40 bikes and 80 bikes.  Second, the average number of bikes in circulation lags slightly behind the arrival rate, which is a common phenomenon in non-stationary queues.  

\subsubsection{Additional Commentary on the Lag Effect}

In the case where the arrival process is non-stationary, we observe a lag effect on the change of empirical measure in response to the change of the arrival rate.   Here we explain what the dynamics of the lag effect are in this non-stationary case.  We observe in Figure \ref{Fig_y0t} and Figure \ref{Fig_y3t} that after time 5, when the effect of the transient behavior is reduced,  the proportion of stations with no bikes increases as arrival rate increases. However, the proportion of stations with 3 bikes goes down as arrival rate increases. The intuition is that when more people starting picking up bikes at stations, we end up with more empty stations and less full stations.  Similarly, as the arrival rate decreases, the proportion of stations with no bikes decreases while the proportion of stations with 3 bikes increases. The intuition is that when less people starting picking up bikes at stations, we end up with less empty stations and more full stations.  

We also observe a lag between the peak(valley) of arrival and the peak(valley) of the proportion of stations with no bikes (Figure \ref{Fig_y0t}). There is also a lag between the peak(valley) of arrival and the valley(peak) of the proportion of stations with 3 bikes (Figure \ref{Fig_y3t}). This is because it takes time for the empirical process to respond to the change in the arrival rate.  In Figure \ref{Fig_lag}, we plot the size of the lag for different values of the service rate $\mu$.  We see that the lag effect has the most impact on the proportion of stations with one bike.  It also has an effect on the proportion of stations with no bikes, 2 bikes and 3 bikes, however, the time lag is much smaller than that of one bike.  We also find that the lag effect does not affect maximums and minimums the same way.  We observe in all of the plots, except for proportion of stations with one bike, that the lag effect is more pronounced for minimums than maximums.  This is especially true in the case of the proportion of stations having two bikes.

\begin{figure}[H]
\centering
\vspace{-.25in}
\includegraphics[width=0.65\textwidth]{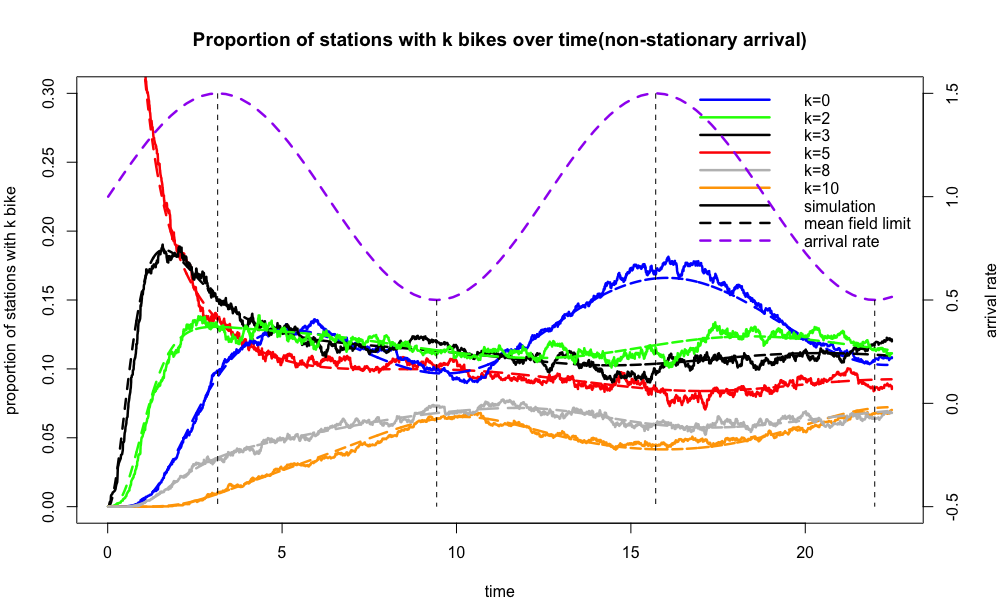}
\captionsetup{justification=centering}
\caption{Proportion of stations with k bikes over time when $K=10$\\ The solid lines represent the simulation results $Y_{t}^{N}$. The dashed lines of the same color represent the corresponding mean field limit $y_t$.}\label{Fig_ykt}
\end{figure}

\begin{figure}[H]
\centering
\vspace{-.25in}
\includegraphics[width=0.65\textwidth]{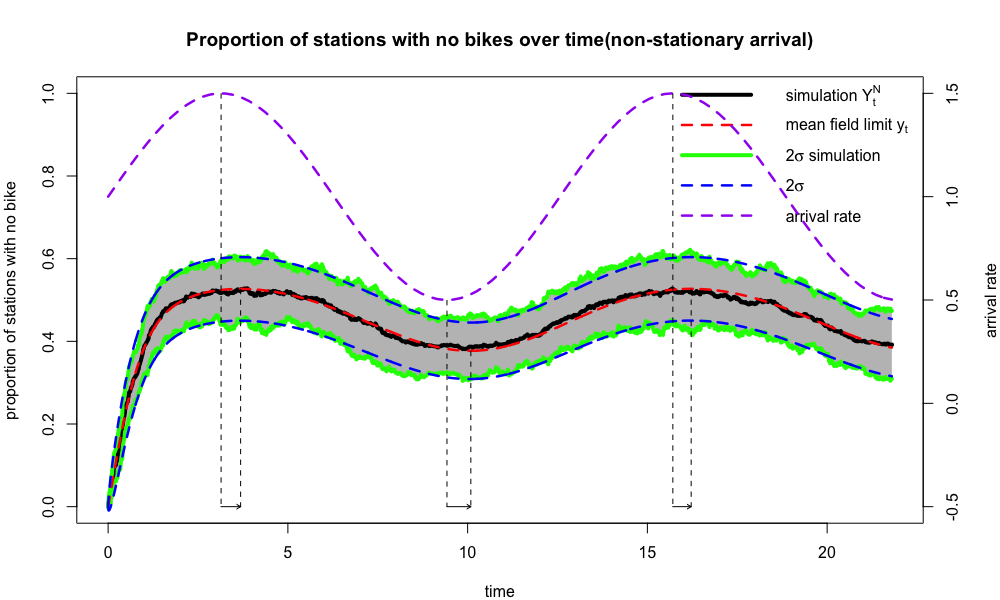}
\caption{Simulation results of $Y_{t}^{N}(0)$ vs. $y_{t}(0)$ when $K=3$}\label{Fig_y0t}
\end{figure}

\begin{figure}[H]
\centering
\vspace{-.25in}
\includegraphics[width=0.65\textwidth]{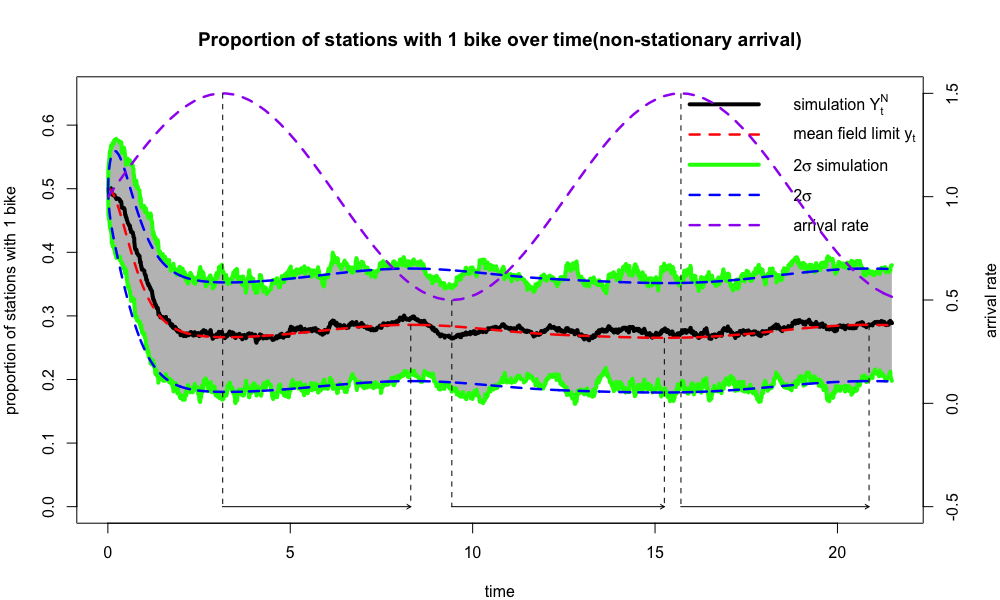}
\caption{Simulation results of $Y_{t}^{N}(1)$ vs. $y_{t}(1)$ when $K=3$}\label{Fig_y1t}
\end{figure}

\begin{figure}[H]
\centering
\vspace{-.25in}
\includegraphics[width=0.65\textwidth]{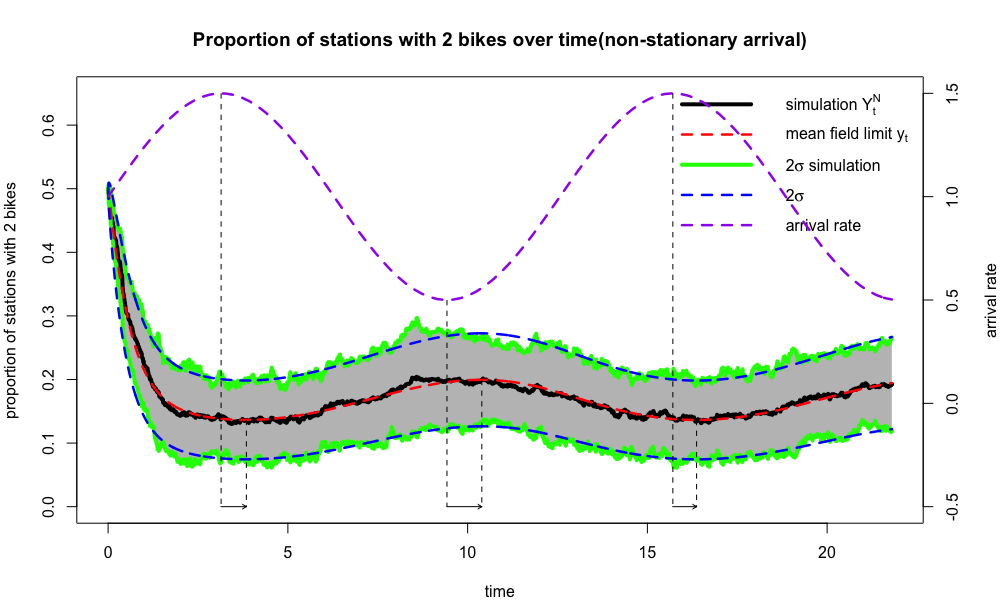}
\caption{Simulation results of $Y_{t}^{N}(2)$ vs. $y_{t}(2)$ when $K=3$}\label{Fig_y2t}
\end{figure}

\begin{figure}[H]
\centering
\vspace{-.25in}
\includegraphics[width=0.65\textwidth]{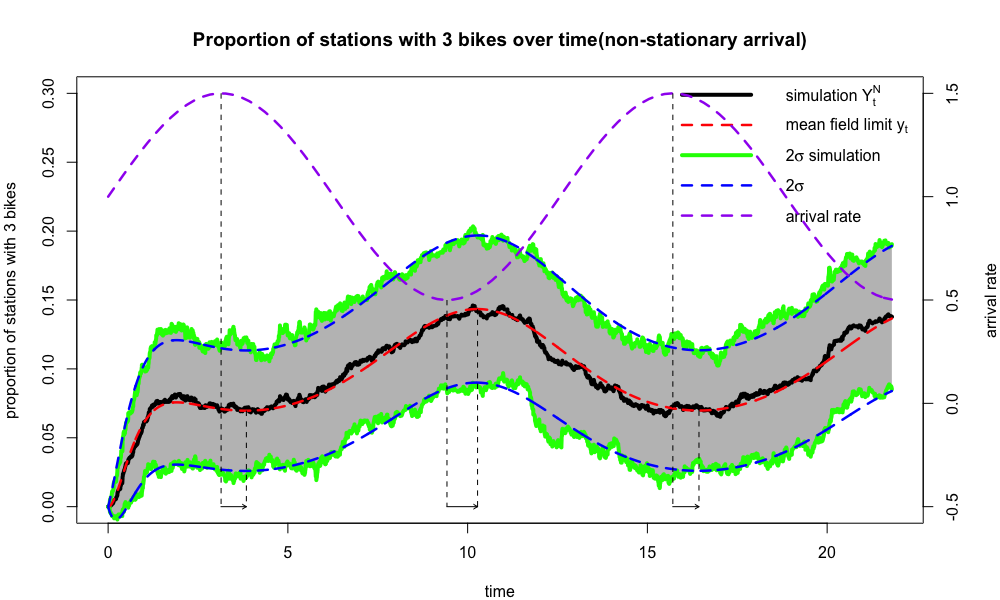}
\caption{Simulation results of $Y_{t}^{N}(3)$ vs. $y_{t}(3)$ when $K=3$}\label{Fig_y3t}
\end{figure}

\begin{figure}[H]
\centering
\vspace{-.25in}
\includegraphics[width=0.65\textwidth]{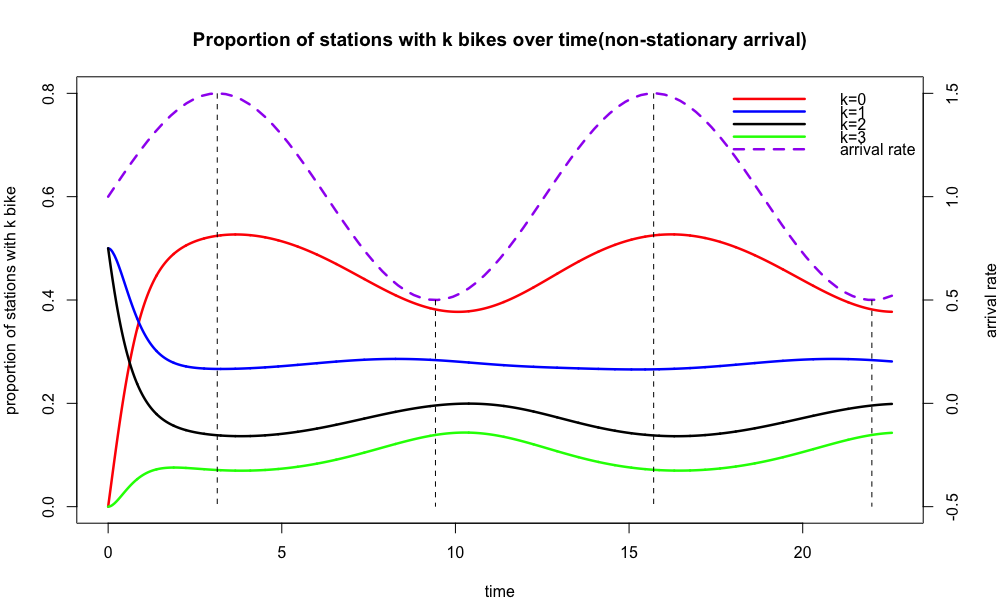}
\caption{Mean field limits of the proportion of stations with $k$ bikes over time when $K=3$}\label{Fig_ykt(3)}
\end{figure}

\begin{figure}[H]
\centering
\vspace{-.25in}
\includegraphics[width=0.65\textwidth]{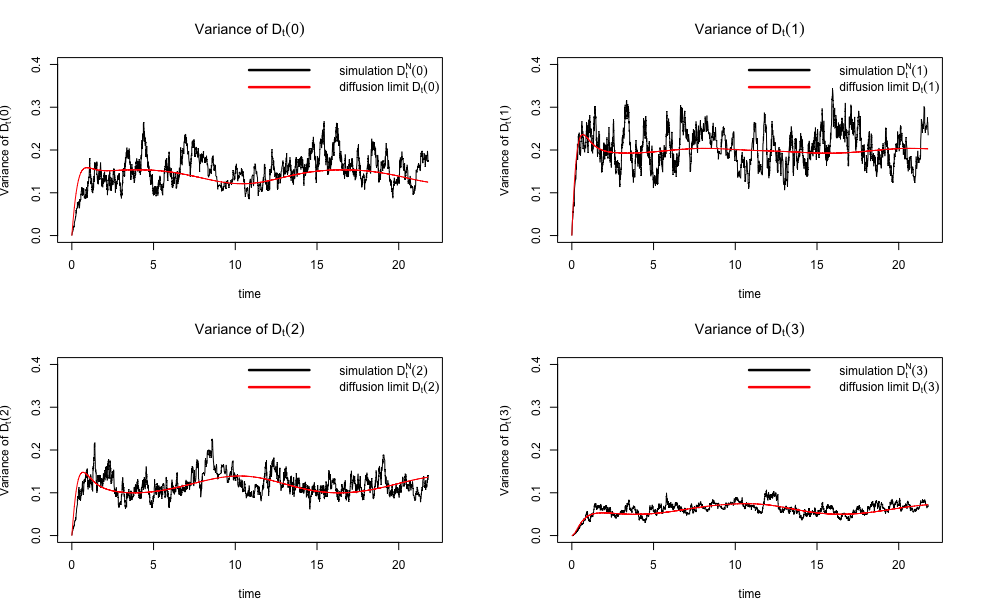}
\caption{Simulation results of the $\text{Var}[D_{t}^{N}]$ vs. $\text{Var}[D_{t}]$ with $K=3$}\label{Fig_Vart}
\end{figure}

\begin{figure}[H]
\centering
\vspace{-.25in}
\includegraphics[width=0.65\textwidth]{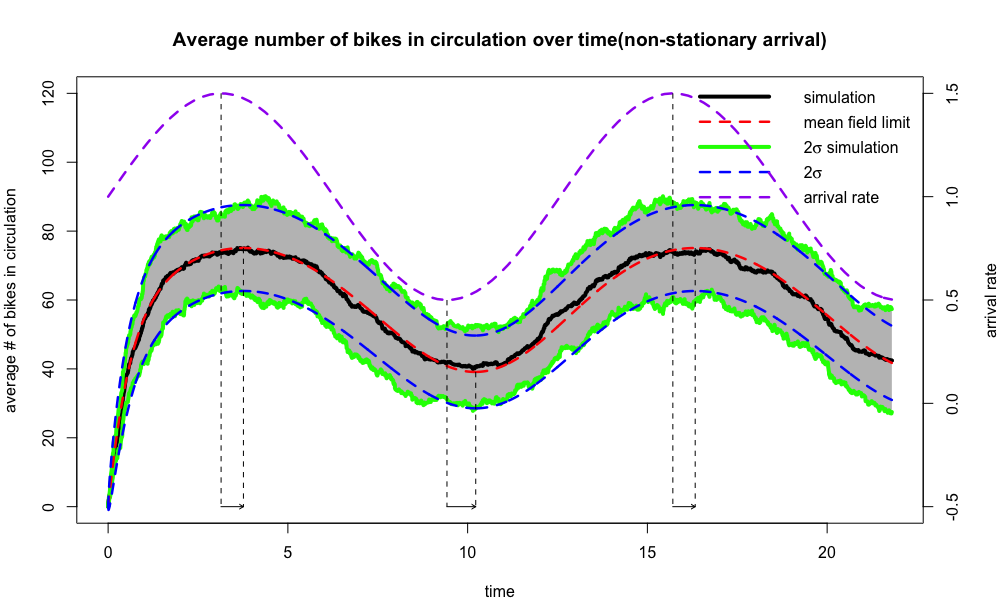}
\captionsetup{justification=centering}
\caption{Simulation results vs. mean field limit of the average number of bikes in circulation with $K=3$ and $M=150$}\label{Fig_circ_NS}
\end{figure}

\begin{figure}[H]
\centering
\vspace{-.25in}
\includegraphics[width=0.65\textwidth]{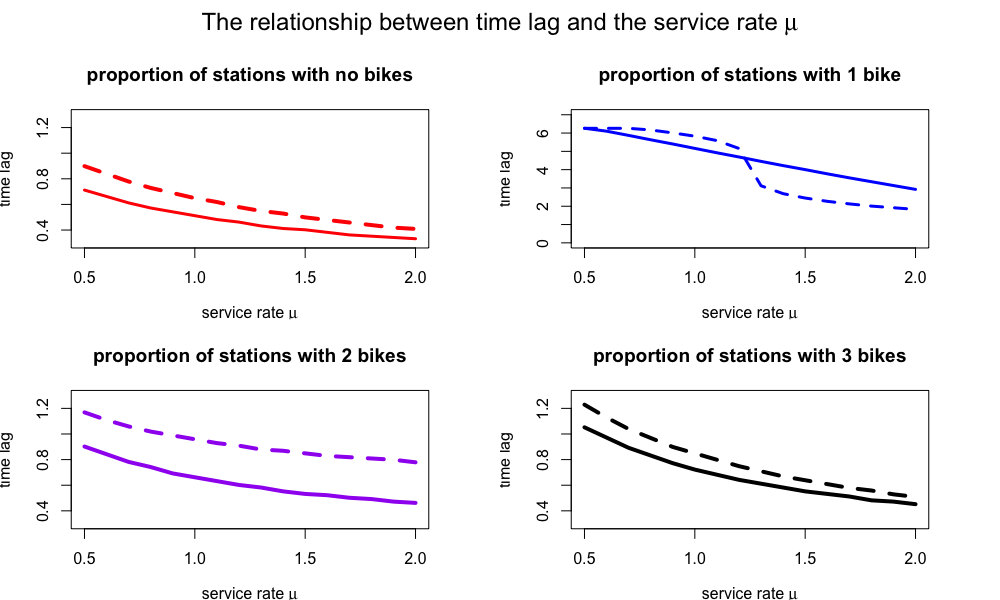}
\captionsetup{justification=centering}
\caption{The relationship between time lag and the service rate $\mu$. The (solid, dashed) lines represent the time lag when arrival rate reaches (maximum, minumum).}\label{Fig_lag}
\end{figure}



\subsection{Simulation Experiments with Heterogeneous Arrival Rates and Capacities}
 In this section, we provide the results of the simulation studies of our stochastic bike sharing model with heterogeneous arrival rates and capacities. We also extend the scale of the system in our example to make it closer to reality. Some of the common model parameters that we use in all of our simulation experiments are given as follows:  $N=800$, $\mu=1$, and for each station, $K=20,30,40,$ or 50 with equal probabilities. The number of sample paths we use is 30. Other parameters are specified in the illustration of each figure below.
 
\paragraph{Figure~\ref{Fig_capacity_distribution} :}  Figure \ref{Fig_capacity_distribution} shows the distribution of station capacities of Citi Bike (December 2017). We observe that most station capacities range from  20 to 50. Thus we set the parameter $K=20,30,40,$ or 50 with equal probabilities. 

\paragraph{Figure~\ref{Fig_citi_average_em} :} Figure \ref{Fig_citi_average_em} shows the average empirical distribution of CitiBike in a day (September 25th, 2017), with one standard deviation. This is exactly the empirical measure $y$ we try to analyze in our bike sharing model. Studying the empirical measure and the its fluctuation over time is of great importance to understanding the system dynamics, providing prediction and guidance for rebalancing and strategic design.

\paragraph{Figure~\ref{Fig_simulation_citi} :} In Figure \ref{Fig_simulation_citi}, we simulate the real Citi Bike network (New York City area) and show the empirical measure results from the simulation. The network setup is exactly the same as CitiBike over the month of September, 2017, where we have $N=695$ number of stations with capacities being the same as Citi Bike. We use Fourier regression to provide a fit for arrival functions for each station, and use them as heterogeneous time-varying arrival rates in our bike sharing model. We used a constant travel time rate $\mu=3.75$, which comes from taking average of the trip duration data. We run the simulation for a 24-hour period and plot the average empirical measure from the simulation. 

\paragraph{Figure~\ref{Fig_difference} :} Figure \ref{Fig_difference} shows the difference in average empirical measure between simulation and real CitiBike data. We can see that the simulated empirical measure have at most 0.02 difference from the real CitiBike data. We can also observe that stations with 10-28 bikes have a positive difference while stations with less than 10 bikes have a negative difference. This is because rebalancing happens in the real CitiBike network, where bikes will be moved from stations with more bikes to stations with less bikes during peak hours to make the system more balanced, something we are not able to capture in our model. This causes the result from our model to have a larger number of  stations with more bikes, and a smaller number of stations with less stations, compared to reality. Besides, we also didn't consider the effect of information on CitiBike customers in our model, such as smartphone app that tells people the number of bikes and docks at stations in real time. In reality, a lot of people rely on the app to find stations for picking up and dropping off, which drives people to stations that have more bikes or docks. This would cause the same effect in making simulation results slightly different from reality. Overall, we can conclude from the simulation results that the bike sharing model proposed in our paper is doing well in capturing the real-life situation at CitiBike, which shows that our mean field limit and diffusion limit results will be of great value to the understanding the bike sharing systems in real life.

\paragraph{Figure~\ref{Fig_y_stationary} :}This figure shows the empirical measure from simulation and its mean field limit in the heterogeneous arrivals and capacities case where arrivals are stationary. The arrival rates are set up to be $\lambda=0.25,0.5,0.75,$ or 1 with equal probabilities. In this figure we showed the the empirical measure and its corresponding mean field limit with 95\% confidence interval at a given time point. The blue bars show the empirical measure from simulation, which is an average of 30 sample paths. The red bars show the corresponding mean field limit we get from solving the ODEs in (\ref{fluid_eqn_ex}). The error bars show the 95\% confidence interval from the simulation results. We can see that the mean field limit is fitting well with empirical measure, which shows that our model is able to capture heterogeneous arrivals and capacities very well, and it can be easily adapted to more complex models, such as those with non-uniform routing probabilities, as shown in Section \ref{Ext}.

\paragraph{Figure~\ref{Fig_yk_stationary} :} This figure shows the evolution of different components of the empirical measure over time in the heteogeneous arrivals and capacities setting with the same parameters setup as in Figure \ref{Fig_yk_stationary}.  The solid lines represents proportion of stations with $k$ bikes where $k=5, 10, 15, 20, 25$. The dashed lines represents their corresponding mean field limits. Again the mean field limits are fitting well with the empirical measure.

\paragraph{Figure~\ref{Fig_y_nonstationary} :}This figure shows the empirical measure from simulation and its mean field limit in the heterogeneous arrivals and capacities case where arrivals are non-stationary. The arrival rates are set up to be $\lambda(t)=4\lambda_0(1+\sin(2t))$ where $\lambda_0=0.25,0.5,0.75,1$ with equal probabilities. In this figure we showed the the empirical measure and its corresponding mean field limit with 95\% confidence interval at a given time point. The blue bars show the empirical measure from simulation, which is an average of 30 sample paths. The red bars show the corresponding mean field limit we get from solving the ODEs in (\ref{fluid_eqn_ex}). The error bars show the 95\% confidence interval from the simulation results. We can see that the mean field limit is fitting well with empirical measure, which shows that our model is able to capture heterogeneous arrivals that are also non-stationary very well, and it can be easily adapted to more complex models, such as those with non-uniform routing probabilities, as shown in Section \ref{Ext}.

\paragraph{Figure~\ref{Fig_yk_nonstationary} :} This figure shows the evolution of different components of the empirical measure over time in the heterogeneous arrivals and capacities setting with the same parameters setup as in Figure \ref{Fig_yk_nonstationary}.  The solid lines represents proportion of stations with $k$ bikes where $k=0, 10, 20, 30, 40, 50$. The dashed lines represents their corresponding mean field limits. We see that the empirical measure is also non-stationary and that the mean field limits are fitting well with the empirical measure.

\paragraph{Computational cost}:
Another major benefit of our model compared to just using simulation to study the bike sharing system is that our model is extremely computationally inexpensive, given that it only involves numerically solving ODEs. The computational time for running the simulation example in Figure \ref{Fig_y_nonstationary}, with 10 sample paths for 8 unit times is 13.5 hours, while getting the mean field limit for the same example only takes less than 30 seconds.
 
\begin{figure}[H]
\centering
\vspace{-.25in}
\includegraphics[width=0.5\textwidth]{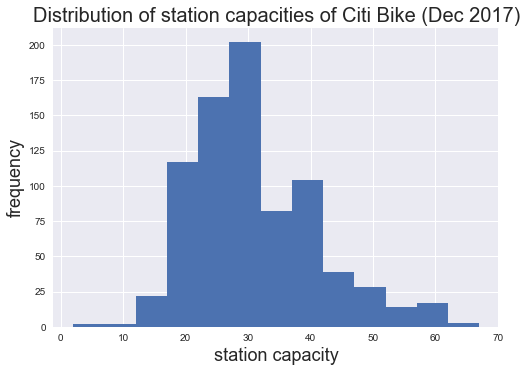}
\caption{Histogram of capacities at Citi Bike}\label{Fig_capacity_distribution}
\end{figure}

\begin{figure}[H]
\centering
\vspace{-.25in}
\includegraphics[width=0.65\textwidth]{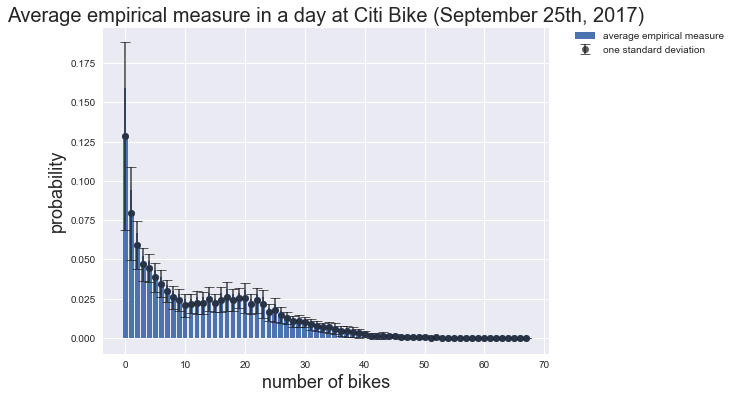}
\caption{Average empirical measure in a week at Citi Bike}\label{Fig_citi_average_em}
\end{figure}

\begin{figure}[H]
\centering
\vspace{-.25in}
\includegraphics[width=0.65\textwidth]{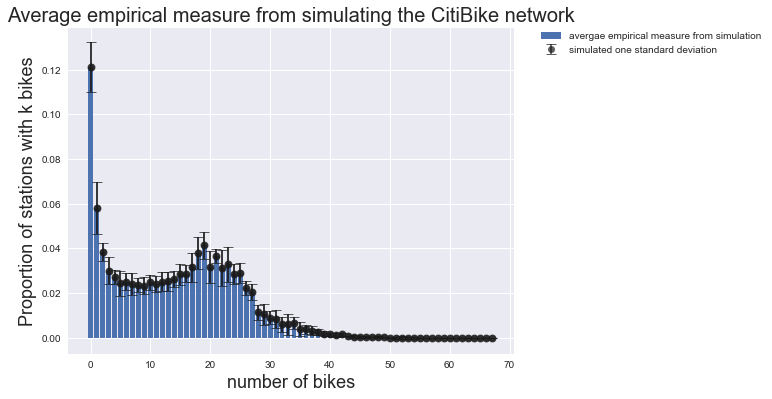}
\caption{Average empirical measure from simulating the Citi Bike network}\label{Fig_simulation_citi}
\end{figure}

\begin{figure}[H]
\centering
\vspace{-.25in}
\includegraphics[width=0.65\textwidth]{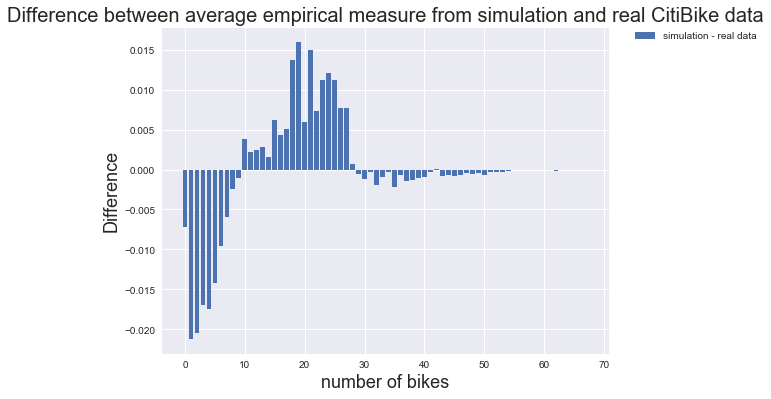}
\caption{Difference between average empirical measure from simulation and real Citi Bike data}\label{Fig_difference}
\end{figure}

\begin{figure}[H]
\centering
\vspace{-.25in}
\includegraphics[width=0.65\textwidth]{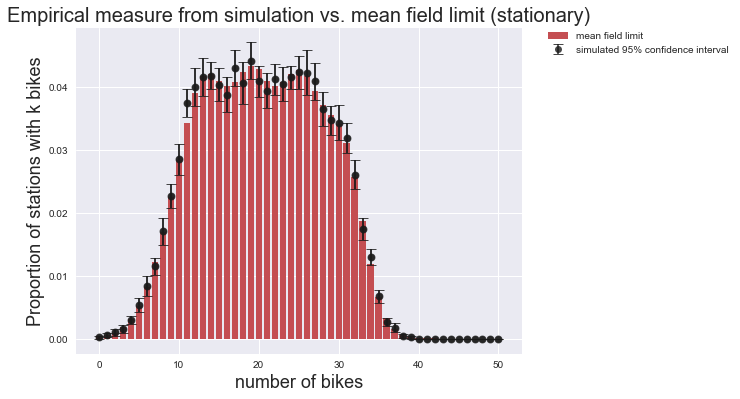}
\caption{Empirical measure from simulation vs. mean field limit}\label{Fig_y_stationary}
\end{figure}

\begin{figure}[H]
\centering
\vspace{-.25in}
\includegraphics[width=0.65\textwidth]{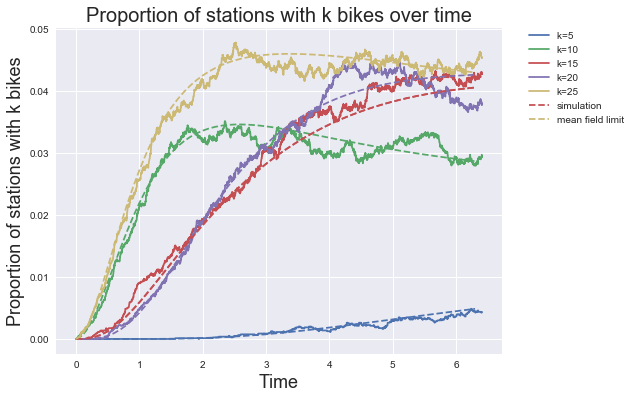}
\caption{Empirical measure from simulation vs. mean field limit}\label{Fig_yk_stationary}
\end{figure}

\begin{figure}[H]
\centering
\vspace{-.25in}
\includegraphics[width=0.65\textwidth]{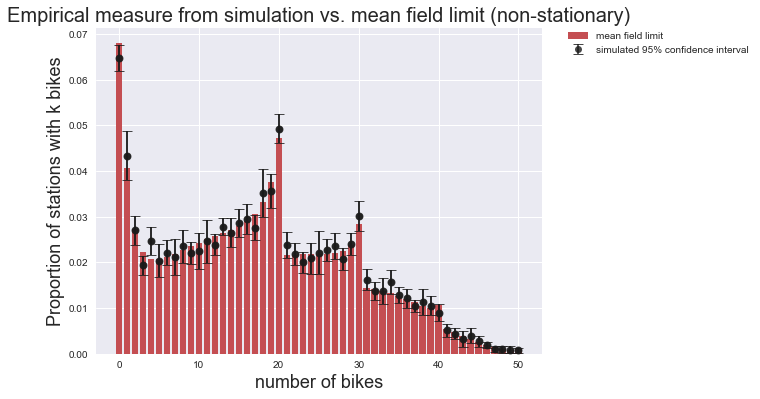}
\caption{Empirical measure from simulation vs. mean field limit (non-stationary)}\label{Fig_y_nonstationary}
\end{figure}

\begin{figure}[H]
\centering
\vspace{-.25in}
\includegraphics[width=0.65\textwidth]{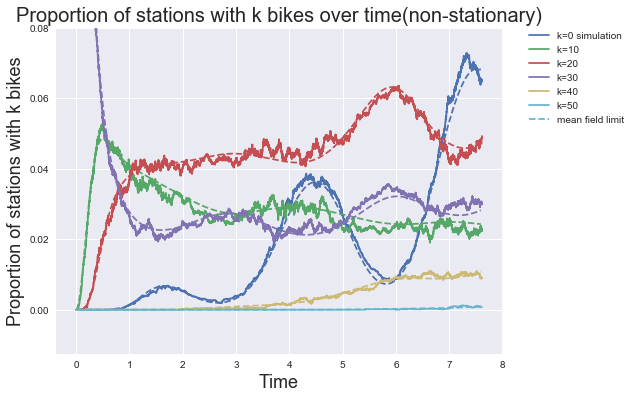}
\caption{Empirical measure from simulation vs. mean field limit (non-stationary)}\label{Fig_yk_nonstationary}
\end{figure}

\section{Applied and Practical Value of Our Work}\label{value}
The results presented in this paper have great value to the operations of bike sharing systems. First of all, not only is the empirical measure itself an important performance measure to the bike sharing systems, but it also allows us to obtain salient performance measures such as $y_{t}(0)$ (the proportion of stations with no bikes), $y_{t}(K)$ (the proportion of stations that are full), $M - \sum^{K}_{j=0} j \cdot y_{t}(j)N$ (the number of bikes in circulation), among others.  The empirical measure approach is a significant reduction in computational complexity when compared to the full stochastic model.  

Moreover, it is the first time that a diffusion limit of the empirical measure of a inhomogeneous bike sharing systems is derived, which gives great insights to the fluctuation of the systems performance over time from a queueing and risk management perspective. It is a much more computationally efficient way to study the system behaviors in strategic planning stage, where given the design of the system and parameters such as arrival rate, travel time distribution and fleet size, you can easily get the proportion of problematic stations (empty and full stations) over time, which is a key measure that we want to minimize in a bike sharing system. More importantly, the diffusion limit of the empirical measure provides a refinement to the mean field limit, which gives us a better understanding to system fluctuations over time. By using the diffusion limit, we are able to build confidence intervals for the proportion of problematic stations. This helps us design a BSS with low blocking experiences, not just in expectation, but with high probability.  This is especially important for managers of BSSs who want to control the dynamics of bike stations and reduce the volatility of station fluctuations.

Another benefit of our work is that deriving the mean field limit and diffusion limit provides a way to formulate optimization problems associated with BSS. For example, the mean field and diffusion limit for the proportion of problematic stations can be used as objective functions that we try to minimize given system parameters. The expectation and variance of the number of bikes in circulation can computed through the mean field limit and diffusion limit of the empirical process, and therefore can be used to determine optimal fleet size.  The current literature only uses mean field limits and our diffusion limits can be used to determine optimal fleet size under a more complex stochastic setting instead of a deterministic setting.  

Our analysis also benefits the rebalancing of bike sharing systems. It can be used to provide short-term prediction of the empirical measure of the system, which helps the operators of BSS identify when key measures such as proportion of problematic stations, or number of bikes in circulation, will go beyond a threshold and act beforehand. Different from traditional data analysis methods that predict patterns of BSS solely using history data, our method analyzes the system behaviors from a more fundamental way, one that does not heavily rely on data.  Most importantly, as peak hours only last a few hours, decisions for rebalancing need to be made fast. In this case our method is much more computationally efficient and effective, than just doing simulations, which tends to be very slow and intractable for large scale systems like CitiBike.  

Another major benefit of using the empirical measure approach is that if CitiBike chooses to add a station, then in the empirical measure approach the dimensionality will only increase if the number of docks at the new station is larger than all of the rest.  However, in the individual station model, it will automatically increase by one.  Even though systems like CitiBike are large, they continue to add stations and increase the complexity of simulating the system.  The empirical measure approach that we advocate in this work does not get worse when the management chooses to add stations.

From a broader perspective, the framework of mean field and diffusion limits we established in this paper provide an  effective and efficient way to analyze different problems associated with BSS, such as, designing reasonable architecture of a BSS, finding a better path scheduling, improving inventory management, redistributing the bikes among stations or clusters, price optimization, application of intelligent information technologies and so forth.  Our work serves as the initial step to exploring these important problems facing BSS.

\section{Conclusion}\label{conc}

In this paper, we construct a bike sharing queueing model that incorporates the finite capacity of stations. Since our model is intractable to analyze directly, especially for a large number of stations, we propose to analyze the limiting dynamics of an empirical process that describes the proportion of stations that have a certain number of bikes.  We prove a mean field limit and a functional central limit theorem for our stochastic bike sharing empirical process, showing that the mean field limit and the variance of the empirical process can be described by a system of $\frac{1}{2}(K+4)(K+1)$ differential equations where $K$ is the maximum station capacity.  We compare the mean field limit and the functional central limit theorem with simulation and show that the differential equations approximate the mean and variance of the empirical process extremely well. 

There are many directions for future work.  The first direction would be to generalize the arrival and service distribution to follow general distributions.  As Figure \ref{Fig_histogram} shows, the trip durations are not exponential and are closer to a lognormal distribution.  An extension to general distributions would aid in showing how the non-exponential distributions affect the dynamics of the empirical process.  Recent work by \citet{ ko2016strong, ko2017diffusion, pender2016approximations} provides a Poisson process representation of phase type distributions and Markovian arrival processes.  This work might be useful in deriving new limit theorems for the queueing process with non-renewal arrival and service processes.  

In the non-stationary context, it is not only important to understand the dynamics of the mean field limit, but also it is important to know various properties of the mean field limit. For example, it would be informative to know the size of the amplitude and the frequency of the mean field limit when the arrival rate is periodic. One way to analyze the amplitude and the frequency is to exploit methods from non-linear dynamics like Lindstedt's method and the two-variable expansion method in \citet{pender2017queues, pender2018analysis, nirenberg2018impact, novitzky2019nonlinear}.

Lastly, it is also interesting to consider a spatial model of arrivals to the bike sharing network.  In this case, we would consider customers arriving to the system via a spatial Poisson process and customers would choose among the nearest stations to retrieve a bike.  This spatial process can model the real choices that riders make and would model the real spatial dynamics of bike sharing networks.  We intend to pursue these extensions in future work.   

\begin{figure}[H]
\centering
\vspace{-.25in}
\includegraphics[width=0.65\textwidth]{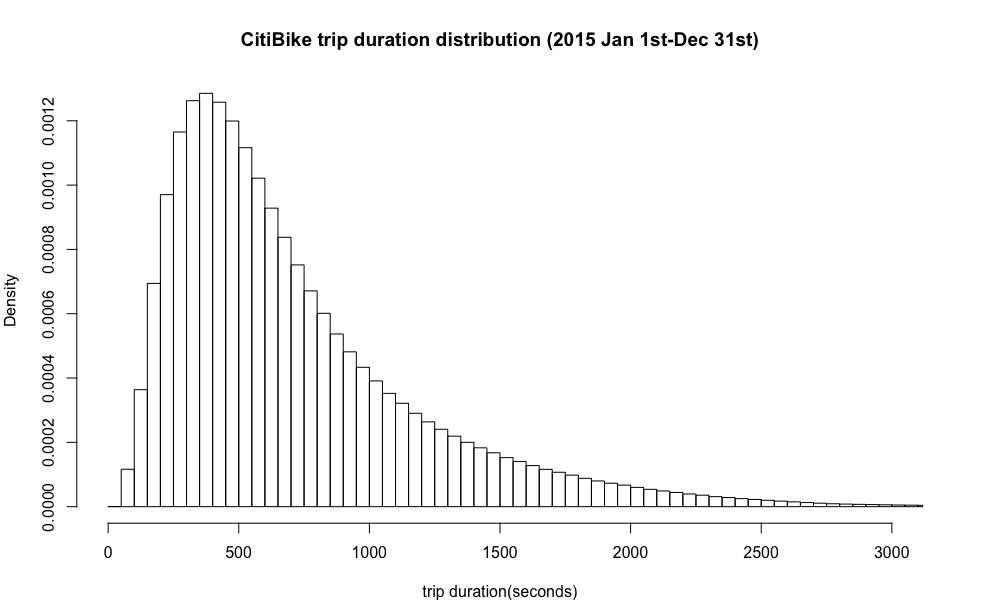}
\captionsetup{justification=centering}
\caption{Histogram of Citi Bike trip duration (Jan 1st-Dec 31st, 2015) \\ Mean = 713s/11.9 mins,  Median = 576s/9.6 mins, Stdev = 492s/8.2 mins.}\label{Fig_histogram}
\end{figure}

\section*{Acknowledgments}
Jamol Pender gratefully acknowledges the support of National Science Foundation (NSF) for Jamol Pender's Career Award CMMI \# 1751975.


\appendix
\section{Appendix} \label{App}

\begin{proof}[Proof of Proposition \ref{functional_forward}]
 The time derivative of the expectation $\mathbb{E}[f(\mathbf{X}(t))]$ can be derived by the following discretization method. Taking the expectation on $f(\mathbf{X}(t+\Delta))$ conditioned on $\mathbf{X}(0)=\mathbf{x}$ for some small $\Delta>0$, we have
 \begin{eqnarray}
 & &\mathbb{E}\left[f(\mathbf{X}(t+\Delta)|\mathbf{X}(0)=\mathbf{x}\right] \nonumber \\
 &=&\sum_{i=1}^{N}f(\mathbf{X}(t))\left[1-\lambda_{i}\Delta\mathbf{1}_{\{X_{i}(t)>0\}}-\mu P_{i} \left(M-\sum_{k=1}^{N}X_{k}(t)\right)\Delta\mathbf{1}_{\{X_{i}(t)<K_{i}\}}\right] \nonumber \\
& & +\sum_{i=1}^{N} \left[f(\mathbf{X}(t)-\mathbf{1}_{i})\lambda_{i}\Delta\mathbf{1}_{\{X_{i}(t)>0\}}+ f(\mathbf{X}(t)+\mathbf{1}_{i})\mu P_{i}\left(M-\sum_{k=1}^{N}X_{k}(t)\right)\Delta\mathbf{1}_{\{X_{i}(t)<K_{i}\}} \right]\nonumber  \\
& &+o(\Delta).
 \end{eqnarray}
 Then 
 \begin{eqnarray}\label{forward_deriv}
 & &\mathbb{E}\left[f(\mathbf{X}(t+\Delta)|\mathbf{X}(0)=\mathbf{x}\right]- f(\mathbf{X}(t))\nonumber \\
&=& \sum_{i=1}^{N} \left[f(\mathbf{X}(t)-\mathbf{1}_{i})-f(\mathbf{X}(t)\right]\lambda_{i}\Delta\mathbf{1}_{\{X_{i}(t)>0\}}\nonumber \\
& & +\sum_{i=1}^{N}\left[f(\mathbf{X}(t)+\mathbf{1}_{j})-f(\mathbf{X}(t)\right]\mu P_{i}\left(M-\sum_{k=1}^{N}X_{k}(t)\right)\Delta\mathbf{1}_{\{X_{i}(t)<K_{i}\}}\nonumber   \\
& &+o(\Delta).
 \end{eqnarray}
 By dividing by $\Delta$ and taking the expectation on both sides of Equation(\ref{forward_deriv}) we get
 \begin{eqnarray}
& & \frac{\mathbb{E}\left[f(\mathbf{X}(t+\Delta))\right]-\mathbb{E}\left[f(\mathbf{X}(t))\right]}{\Delta}\nonumber \\
&=& \sum_{i=1}^{N} \mathbb{E}\left[\left(f(\mathbf{X}(t)-\mathbf{1}_{i})-f(\mathbf{X}(t)\right)\lambda_{i}\mathbf{1}_{\{X_{i}(t)>0\}}\right]\nonumber \\
& & +\sum_{i=1}^{N} \mathbb{E}\left[\left(f(\mathbf{X}(t)+\mathbf{1}_{i})-f(\mathbf{X}(t)\right)\mu P_{i}\left(M-\sum_{k=1}^{N}X_{k}(t)\right)\mathbf{1}_{\{X_{i}(t)<K_{i}\}} \right]\nonumber  \\
& &+o(\Delta)/\Delta .
 \end{eqnarray}
 Taking $\Delta \rightarrow 0$ yields
 \begin{eqnarray}
& &\frac{d}{dt}\mathbb{E}\left[f(\mathbf{X}(t))\right]\nonumber \\
&=& \sum_{i=1}^{N} \mathbb{E}\left[\left(f(\mathbf{X}(t)-\mathbf{1}_{i})-f(\mathbf{X}(t)\right)\lambda_{i}\mathbf{1}_{\{X_{i}(t)>0\}}\right]\nonumber  \\
& & +\sum_{i=1}^{N} \mathbb{E}\left[\left(f(\mathbf{X}(t)+\mathbf{1}_{i})-f(\mathbf{X}(t)\right)\mu P_{i}\left(M-\sum_{k=1}^{N}X_{k}(t)\right)\mathbf{1}_{\{X_{i}(t)<K_{i}\}} \right].
 \end{eqnarray}
 Let $f(\mathbf{X}(t))=X_i(t)$ for $i=1,\cdots,N$, we have the following functional forward equations to each component of $\mathbf{X}(t)$,
 \begin{eqnarray}
 & &\lefteqn{ \updot{\mathbb{E}}[f(X_i(t)) | X_i(0) = x_i]  \equiv  \frac{d}{dt} \mathbb{E}[f(X_i(t)) | X_i(0) = x_i]  } \nonumber \\ &=&  \mathbb{E}\left[ \left( f(X_i(t) +1) - f(X_i(t)) \right) \cdot \left(\mu P_{i}\left(M-\sum_{k=1}^{N}X_{k}(t)\right)\mathbf{1}\{X_{i}(t)<K_{i}\}\right) \right]  \nonumber\\& &+  \mathbb{E}\left[ \left( f(X_i(t) -1) - f(X_i(t)) \right) \cdot \left( \lambda_{i}\mathbf{1}\{ X_{i}(t) > 0 \}  \right) \right] .
 \end{eqnarray} 
\end{proof}

\begin{proof}[Proof of Theorem \ref{fluid_limit}]
Let $|\cdot|$ denote the Euclidean norm in $\mathbb{R}^{K+1}$, then
\begin{eqnarray}
\left|Y_{t}^{N}-y_{t}\right|&=&\left|Y_{0}^{N}+M_{t}^{N}+\int_{0}^{t}\beta
(Y_{s}^{N})ds-y_0-\int_{0}^{t}b(y_{s})ds\right|\nonumber \\
& =& \left|Y_{0}^{N}-y_0+M_{s}^{N}+\int_{0}^{t}\left(\beta
(Y_{s}^{N})-b(Y_{s}^{N})\right)ds+\int_{0}^{t}(b(Y_{s}^{N})-b(y_{s}))ds\right|.\nonumber \\
\end{eqnarray}
Now define the random function $f^{N}(t)=\sup_{s\leq t}\left|Y_s^{N}-y_s\right|$, we have
$$f^{N}(t)\leq |Y_0^{N}-y_0|+\sup_{s\leq t}|M_s^{N}|+\int_0^t|\beta(Y_s^{N})-b(Y_s^{N})|ds+\int_{0}^{t}|b(Y_s^{N})-b(y_s)|ds.$$
By Proposition~\ref{Lipschitz}, $b(y)$ is Lipschitz with respect to Euclidean norm. Let $L$ be the Lipschitz constant of $b(y)$, then
\begin{eqnarray}
f^{N}(t)&\leq& |Y_0^{N}-y_0|+\sup_{s\leq t}|M_s^{N}|+\int_0^t|\beta(Y_s^{N})-b(Y_s^{N})|ds+\int_{0}^{t}|b(Y_s^{N})-b(y_s)|ds \nonumber \\
&\leq& |Y_0^{N}-y_0|+\sup_{s\leq t}|M_s^{N}|+\int_0^t|\beta(Y_s^{N})-b(Y_s^{N})|ds+L\int_{0}^{t}|Y_s^{N}-y_s|ds \nonumber \\
&\leq& |Y_0^{N}-y_0|+\sup_{s\leq t}|M_s^{N}|+\int_0^t|\beta(Y_s^{N})-b(Y_s^{N})|ds+L\int_{0}^{t}f^{N}(s)ds.
\end{eqnarray}

By Gronwall's lemma (See \citet{ames1997inequalities}), 
\begin{equation}
f^{N}(t) \leq \left(|Y_0^{N}-y_0|+\sup_{s\leq t}|M_s^{N}|+\int_0^t|\beta(Y_s^{N})-b(Y_s^{N})|ds\right)e^{Lt}.
\end{equation}

Now  to bound $f^{N}(t)$ term by term, we define the function $\alpha: [0,1]^{K+1}\rightarrow \mathbb{R}^{K+1}$ as
\begin{eqnarray}
\alpha(y)&=&\sum_{x\neq y}|x-y|^2Q(y,x)\nonumber\\
&=&\frac{1}{N}\sum_{n}\sum_{r}\left[\frac{1}{rNR_{\max}}(\mathbf{1}_{(r,n-1)}+\mathbf{1}_{(r,n)})\mathbf{1}_{n>0}  \right.\nonumber \\& &+
\left. \left(\frac{M}{N}-\sum_{n'}\sum_{r'}n'y(r',n')\right)(\mathbf{1}_{(r,n+1)}+\mathbf{1}_{(r,n)})\mathbf{1}_{n<K} \right] \cdot y(r,n) 
\end{eqnarray}
and consider  the following four sets 
\begin{eqnarray}
\Omega_0 &=& \{|Y_0^{N}-y_0|\leq \delta \}, \\
\Omega_1 &=& \left\{\int_0^{t_0}|\beta(Y_s^{N})-b(Y_s^{N})|ds\leq \delta \right\}, \\
\Omega_2 &=& \left\{\int_0^{t_0}\alpha(Y_t^{N})dt \leq A(N)t_0 \right\}, \\
\Omega_3 &=& \left\{\sup_{t\leq t_0}|M_t^{N}|\leq \delta \right\} ,
\end{eqnarray}
where $\delta=\epsilon e^{-Lt_0}/3$. Here the set $\Omega_{1}$ is to bound the initial condition, the set $\Omega_{2}$ is to bound the drift term $\beta$ and the limit of drift term $b$, and  the sets $\Omega_{2},\Omega_{3}$ are to bound the martingale $M_{t}^{N}$.

Therefore on the event $\Omega_0\cap \Omega_1\cap \Omega_3$,
\begin{equation}\label{e}
f^{N}(t_0)\leq 3\delta e^{Lt_0}=\epsilon.
\end{equation}

Since $\lim_{N\rightarrow \infty}\frac{M}{N}=\gamma$ and $\lim_{N\rightarrow \infty}NR_{\max}=\frac{1}{\Lambda}$, we can
choose  $N$ large enough such that
$$\frac{M}{N}\leq 2\gamma,\quad NR_{\max}\geq \frac{1}{2\Lambda}.$$
And by the proof of Proposition \ref{Lipschitz}, there exists $C>0$ such that $\lim_{N\rightarrow \infty}r_{i}^{N}\geq \Lambda/C$. See the proof of Proposition \ref{Lipschitz} in the Appendix for details. 
Thus
\begin{eqnarray}
\alpha(y) &\leq& \frac{1}{N}\sum_{n}\sum_{r} \left(\frac{2\Lambda}{r}\cdot 2+2\gamma \cdot 2 \right)  \cdot y(r,n)\nonumber \\ &\leq&
 \frac{1}{N}\left( \frac{4\Lambda}{\Lambda/C}+4\gamma \right)\sum_{n}\sum_{r}  y(r,n)\nonumber \\ &\leq&
  \frac{4}{N}\left(C+\gamma \right).
  \end{eqnarray}
Consider the stopping time $$T=t_0\wedge \inf \left\{t\geq 0:\int _{0}^{t}\alpha(Y_s^{N})ds>A(N)t_0\right\}.$$
By Proposition~\ref{bound},
$$\mathbb{E}\left(\sup_{t\leq T}|M_{t}^{N}|^2\right)\leq 4\mathbb{E}\int_{0}^{T}\alpha(Y_{t}^{N})dt\leq 4A(N)t_0. $$
On $\Omega_2$, we have $T =t_0$, so 
$\Omega_2 \cap \Omega_3^{c}\subset \{\sup_{t\leq T}|M^{N}_t|>\delta\}$. By
Chebyshev's inequality, 
\begin{equation}
\mathbb{P}(\Omega_2 \cap \Omega_3^{c})\leq \mathbb{P}\left(\sup_{t\leq T}|M^{N}_t|>\delta\right)\leq \frac{\mathbb{E}\left(\sup_{t\leq T}|M_{t}^{N}|^2\right)}{\delta^2}\leq 4A(N)t_0/\delta^2.
\end{equation}
Thus, by Equation (\ref{e}), we have the following result,
\begin{equation}
\begin{split}
\mathbb{P}\left(\sup_{t\leq t_0}|Y_t^{N}-y_t|>\epsilon\right)&\leq \mathbb{P}(\Omega_0^c\cup \Omega_1^c\cup \Omega_3^c)\\
&\leq \mathbb{P}(\Omega_2 \cap \Omega_3^{c})+\mathbb{P}(\Omega_0^{c}\cup \Omega_1^{c}\cup \Omega_2^{c})\\
&\leq 4A(N)t_0/\delta^2+\mathbb{P}(\Omega_0^{c}\cup \Omega_1^{c}\cup \Omega_2^{c})\\
&=36A(N)t_0 e^{2Lt_0}/\epsilon^2+\mathbb{P}(\Omega_0^{c}\cup \Omega_1^{c}\cup \Omega_2^{c}).
\end{split}
\end{equation}

Let $A(N)=\frac{4(C+\gamma)}{N}$, then $\Omega_{2}^{c}=\emptyset$.
And since $Y_0^N\xrightarrow{p} y_0$,   $\lim_{N\rightarrow \infty}\mathbb{P}(\Omega_{2}^{c})=0$. Therefore we have $$\lim_{N\rightarrow \infty}\mathbb{P}\left(\sup_{t\leq t_0}|Y_t^{N}-y_t|>\epsilon\right)=\lim_{N\rightarrow\infty}\mathbb{P}(\Omega_{1}^{c}) .$$
By Proposition \ref{drift}, $\lim_{N\rightarrow\infty}\mathbb{P}(\Omega_{1}^{c})=0$.
Thus, we proved the final result
$$\lim_{N\rightarrow \infty}\mathbb{P}\left(\sup_{t\leq t_0}|Y_t^{N}-y_t|>\epsilon\right)=0.$$
\end{proof}

\begin{customthm}{\ref{bound}}[Bounding Martingale]
For any stopping time $T$ such that $\mathbb{E}(T)<\infty$, we have
\begin{equation}
\mathbb{E}\left(\sup_{t\leq T}|M_{t}^{N}|^2\right)\leq 4\mathbb{E}\int_{0}^{T}\alpha(Y_{t}^{N})dt.
\end{equation}  
\end{customthm}

\begin{proof}[Proof of Proposition~\ref{bound}]
Let $\tilde{\mu}$ be the jump measure of $Y_{t}^{N}$, and $\nu$ be its compensator, defined on $(0,\infty)\times [0,1]$  by 
\begin{eqnarray}
\tilde{\mu}=\sum_{t:Y_{t}^{N}\neq Y_{t-}^{N}}\delta(t,Y_{t}^{N}), & & \nu(dt,B)=Q(Y_{t-}^{N},B)dt \quad \forall B\in \mathcal{B}([0,1]).
\end{eqnarray}
Let $\tilde{Y}_{m,t}^{N}$ be the jump chain of $Y_{t}^{N}$, $J_{m}$ be the jump time, then we have for any $t\in [0,\infty)$, $J_{n}\leq t<J_{n+1}$ for some $n\geq 0$.
The martingale $M_{t}^{N}$ can be written as
\begin{eqnarray}
M_{t}^{N}&=&Y^{N}_{t}-Y_{0}^{N}-\int_{0}^{t}\beta(Y_{s}^{N})ds \nonumber  \\
&=&\sum_{m=0}^{n-1}(\tilde{Y}^{N}_{m+1}-\tilde{Y}^{N}_{m})-\int_{0}^{t}\int_{0}^{1}(y-Y_{s-}^{N})Q(Y_{s-}^{N},dy)ds\nonumber  \\
&=&\int_{0}^{t}\int_{0}^{1}(y-Y_{s-}^{N})\tilde{\mu}(ds,dy)-\int_{0}^{t}\int_{0}^{1}(y-Y_{s-}^{N})\nu(ds,dy)\nonumber  \\
&=&\int_{0}^{t}\int_{0}^{1}(y-Y_{s-}^{N})(\tilde{\mu}-\nu)(ds,dy).
\end{eqnarray}
Note the following identity
\begin{eqnarray}
\left(M_{t}^{N}\right)^2&=&2\int_{0}^{t}\int_{0}^{1}M_{s-}(y-Y_{s-}^{N})(\tilde{\mu}-\nu)(ds,dy)+\int_{0}^{t}\int_{0}^{1}(y-Y_{s-}^{N})^{2}\tilde{\mu}(ds,dy).
\end{eqnarray}
This can be established by verifying that the jumps of the left and right hand sides agree, and that their derivatives agree between jump times. Then we can write
\begin{equation}\label{martingale}
\left(M_{t}^{N}\right)^2=N_{t}^{N}+\int_{0}^{t}\alpha(Y_{t}^{N})ds
\end{equation}
where
\begin{equation}
N_{t}^{N}=\int_{0}^{t}\int_{0}^{1}H(s,y)(\tilde{\mu}-\nu)(ds,dy),
\end{equation}
and
\begin{equation}
H(s,y)=2M_{s-}(y-Y_{s-}^{N})+(y-Y_{s-}^{N})^{2}.
\end{equation}
Consider the previsible process
\begin{equation}
H_{2}(t,y)=H(t,y)\mathbf{1}_{\{t\leq T\wedge T_{n}\}}
\end{equation}
where $T_{n}=\inf\{t\geq 0: \beta(Y_{t}^{N})>n\}\wedge n$.

Then
\begin{equation}
N_{T\wedge T_{n}}=\int_{0}^{\infty}\int_{0}^{1}H_{2}(t,y)(\tilde{\mu}-\nu)(dt,dy),
\end{equation}
and
\begin{eqnarray}
\mathbb{E}\int_{0}^{\infty}\int_{0}^{1}|H_{2}(s,y)|\nu(ds,dy)&=&\mathbb{E}\int_{0}^{T\wedge T_{n}}\int_{0}^{1}|2M_{s-}(y-Y_{s-}^{N})+(y-Y_{s-}^{N})^{2}|\nu(ds,dy) \nonumber \\
&= & \mathbb{E}\int_{0}^{T\wedge T_{n}}\left(2|M_{t}^{N}|\beta(Y_{t}^{N})+\alpha(Y_{t}^{N})\right)dt \nonumber \\
& \leq & \mathbb{E}\int_{0}^{T\wedge T_{n}}2(2+n^2)ndt+\mathbb{E}\int_{0}^{T}\alpha(Y_{t}^{N})dt \nonumber \\
&\leq &2n^4+4n^2+\frac{4(C+\gamma)}{N}\mathbb{E}(T)<\infty.
\end{eqnarray}
By Theorem 8.4 in \citet{darling2008differential}, we have that $N^{T\wedge T_{n}}$ is a martingale. Replace $t$ by $T\wedge T_{n}$ in Equation~(\ref{martingale}) and take expectation to obtain
\begin{equation}
\mathbb{E}(|M_{T\wedge T_{n}}|^2)=\mathbb{E}(N_{T\wedge T_{n}})+\mathbb{E}\int_{0}^{T\wedge T_{n}}\alpha(Y_{t}^{N})dt.
\end{equation}
Since $N^{T\wedge T_{n}}$ is a martingale, 
\begin{equation}
\mathbb{E}(N_{T\wedge T_{n}})=\mathbb{E}(N_0)=0.
\end{equation}
Apply Doob's $L^2$-inequality to the martingale $M^{T\wedge T_{n}}$ to obtain
\begin{eqnarray}
\mathbb{E}\left(\sup_{t\leq T\wedge T_{n} }|M_{t}|^2\right)&\leq& 4\mathbb{E}(|M_{T\wedge T_{n}}|^2)\nonumber \\
&=& 4\mathbb{E}\int_{0}^{T\wedge T_{n}}\alpha(Y_{t}^{N})dt.
\end{eqnarray}
\end{proof}

\begin{customthm}{\ref{Lipschitz}}[Asymptotic Drift is Lipschitz]
The drift function $b(y)$ given in Equation (\ref{eqn:b}) is a Lipschitz function with respect to the Euclidean norm in $\mathbb{R}^{K+1}$. 
\end{customthm}
\begin{proof}[Proof of Proposition~\ref{Lipschitz}]

Assume that $\max_{i}(\lambda_{i}^{N})\leq C
<\infty$ for all $N$, 
then $$r_{i}^{N}=\frac{R_{i}^{N}}{R_{\max}}=\frac{\frac{1}{N\lambda_{i}^{N}}}{\max_{i}\left(\frac{1}{N\lambda_{i}^{N}}\right)}=\frac{\min_{i}\lambda_{i}^{N}}{\lambda_{i}^{N}}.$$
By Assumption~\ref{assumption}, $$NR_{\max}=\frac{1}{\min_{i}\lambda_{i}^{N}}\rightarrow \frac{1}{\Lambda}>0.$$
Therefore $$\lim_{N\rightarrow \infty}r_{i}^{N}\geq \lim_{N\rightarrow \infty}\frac{\Lambda}{\max_{i}\lambda_{i}^{N}}\geq \Lambda/C.$$
Thus the integral that defines function $b$ should start from $\Lambda/C$ instead of 0, i.e.
$$b(y)=\iint\limits_{[\Lambda/C,1]\times [0,...,K]}\left[ \frac{\Lambda}{r}(\mathbf{1}_{(r,n-1)}-\mathbf{1}_{(r,n)})\mathbf{1}_{n>0}+\left(\gamma-\sum_{n}\int ndy(r,n) \right)(\mathbf{1}_{(r,n+1)}-\mathbf{1}_{(r,n)})\mathbf{1}_{n<K}\right] dy(r,n).$$
Now consider $y,\tilde{y}\in [0,1]^{K+1}$,
$$\left|b(y)-b(\tilde{y})\right|\leq 2\left(\frac{\Lambda}{\Lambda/C}+\gamma \right)|y-\tilde{y}|=2(C+\gamma)|y-\tilde{y}|\triangleq L|y-\tilde{y}|$$
where $|\cdot|$ denotes the Euclidean norm.
\end{proof}

\begin{customthm}{\ref{Lipschitz_ex}}[Asymptotic Drift is Lipschitz]
The drift function $b(y)$ given in Equation (\ref{eqn:b_ex}) is a Lipschitz function with respect to the Euclidean norm in $\mathbb{R}^{K_{\max}+1}$. 
\end{customthm}
\begin{proof}[Proof of Proposition~\ref{Lipschitz}]

Assume that $\max_{i}(\lambda_{i}^{N})\leq C
<\infty$ for all $N$, 
then $$\frac{p_i^{N}}{r_{i}^{N}}=\frac{P_{i}^{N}/P_{\max}^{N}}{R_i^N/R^{N}_{\max}}=\lambda_i^{N}\frac{R^{N}_{\max}}{P_{\max}^{N}}$$
By Assumption~\ref{assumption_ex}, $$R^{N}_{\max}/P^{N}_{\max}\rightarrow \frac{1}{\Lambda}>0.$$
Therefore $$\limsup_{N\rightarrow \infty}\frac{p_{i}^{N}}{r_{i}^{N}}\leq C/\Lambda.$$
Thus the function to be integrated in $b$ is bounded,
\begin{eqnarray}
b(y)&=&\sum_{k\in \mathcal{K}}\iint\limits_{[0,1]\times [0,1]}\left[ \frac{p\Lambda}{r}(\mathbf{1}_{(r,n-1,p,k)}-\mathbf{1}_{(r,n,p,k)})\mathbf{1}_{n>0}\right.\nonumber\\
& &+\left.p\mathcal{P}\left(\gamma-\sum_{n}\sum_{k\in \mathcal{K}}\iint\limits_{[0,1]\times [0,1]} ndy(r,n,p,k) \right)(\mathbf{1}_{(r,n+1,p,k)}-\mathbf{1}_{(r,n,p,k)})\mathbf{1}_{n<k}\right] dy(r,n,p,k).\nonumber
\end{eqnarray}

Now consider $y,\tilde{y}\in [0,1]^{K_{\max}+1}$,
$$\left|b(y)-b(\tilde{y})\right|\leq 2\left(\frac{C}{\Lambda}\cdot \Lambda+\mathcal{P}\gamma \right)|y-\tilde{y}|=2(C+\mathcal{P}\gamma)|y-\tilde{y}|\triangleq L|y-\tilde{y}|$$
where $|\cdot|$ denotes the Euclidean norm.
\end{proof}

\begin{customthm}{\ref{drift}}[Drift is Asymptotically Close to a Lipschitz Drift]
Under Assumption~\ref{assumption}, we have for any $\epsilon>0$ and $s\geq 0$,
$$\lim_{N\rightarrow \infty}\mathbb{P}(|\beta(Y_s^{N})-b(Y_s^{N})|>\epsilon)= 0.$$
\end{customthm}
\begin{proof}[Proof of Proposition~\ref{drift}]
 \begin{equation}
 \begin{split}
 &|\beta(Y_s)-b(Y_s)|^2\\
 \leq &2^2\sum_{n}\left|\sum_{r}\frac{1}{rNR_{\max}}Y_s(r,n)-\int_{0}^{1}\frac{\Lambda}{r}dY_s(r,n)\right|^2+2^2\sum_{n}\left|\sum_{r}\frac{M}{N}Y_s(r,n)-\int_{0}^{1}\gamma dY_s(r,n)\right|^2\\
 &+2^2\sum_{n}\left|\sum_{r}\left(\sum_{n'}\sum_{r'}n'Y_s(r',n')\right)Y_s(r,n)-\int_{0}^{1}\left(\sum_{n'}\int_{0}^{1}n'dY_s(r',n')\right)dY_s(r,n)\right|^2.
 \end{split}
 \end{equation} 
 It suffices to show each term goes to zero as $N\rightarrow \infty$.\\
Since $Y_s^{N}$ is a discrete random variable, we have for each $n$,
$$\int_{0}^{1}f(r)dY_{s}(r,n)=\sum_{r}f(r)Y_{s}(r,n)$$  holds for any function $f$.\\
Then
\begin{equation}\label{bound1}
\begin{split}
&\left|\sum_{r}\frac{1}{rNR_{\max}}Y_s(r,n)-\int_{0}^{1}\frac{\Lambda}{r}dY_s(r,n)\right|\\
\leq &\left|\sum_{r}\left(\frac{1}{rNR_{\max}}-\frac{\Lambda}{r}\right) Y_s(r,n)\right|+\left|\sum_{r}\frac{\Lambda}{r}Y_s(r,n)-\int_{0}^{1}\frac{\Lambda}{r}dY_s(r,n)\right|\\
= & \left|\frac{1}{NR_{\max}}-\Lambda\right|\sum_{r}\frac{1}{r}Y_s(r,n)\\
\leq &\left|\frac{1}{NR_{\max}}-\Lambda\right| \frac{C}{\Lambda}\sum_{r}Y_{s}(r,n)\\
\leq &\left|\frac{1}{NR_{\max}}-\Lambda\right|\frac{C}{\Lambda}\rightarrow 0.\\
\end{split}	
\end{equation}
Similarly,
		\begin{equation}\label{bound2}
\begin{split}
&\left|\sum_{r}\frac{M}{N}Y_s(r,n)-\int_{0}^{1}\gamma dY_s(r,n)\right|\\
\leq &\left|\sum_{r}\left(\frac{M}{N}-s\right) Y_s(r,n)\right|+\left|\sum_{r}\gamma Y_s(r,n)-\int_{0}^{1}\gamma dY_s(r,n)\right|\\
\leq &\left|\frac{M}{N}-\gamma \right|\rightarrow 0.\\
\end{split}	
\end{equation}	
The last term is zero since $Y_s^{N}$ is discrete and 
$$\int_{0}^{1}f(r)dY_{s}(r,n)=\sum_{r}f(r)Y_{s}(r,n)$$ for any function $f$.
\end{proof}

\begin{customthm}{\ref{driftbound_ex}}[Drift is Asymptotically Close to a Lipschitz Drift]
Under Assumption~\ref{assumption_ex}, we have for any $\epsilon>0$ and $s\geq 0$,
$$\lim_{N\rightarrow \infty}\mathbb{P}(|\beta(Y_s^{N})-b(Y_s^{N})|>\epsilon)= 0.$$
\end{customthm}
\begin{proof}[Proof of Proposition~\ref{driftbound_ex}]
 \begin{eqnarray}
 & & |\beta(Y_s)-b(Y_s)|^2\nonumber\\
 &\leq & 2^2\sum_{n}\left|\sum_{r,p,k}\frac{pP^{N}_{\max}}{rR^{N}_{\max}}Y_s(r,n,p,k)-\sum_{k\in \mathcal{K}}\iint_{[0,1]\times [0,1]}\frac{p\Lambda}{r}dY_s(r,n,p,k)\right|^2\nonumber\\
 & &+2^2\sum_{n}\left|\sum_{r,p,k}pNP^{N}_{\max}\frac{M}{N}Y_s(r,n,p,k)-\sum_{k\in \mathcal{K}}\iint_{[0,1]\times [0,1]}p\mathcal{P}\gamma dY_s(r,n,p,k)\right|^2\nonumber\\
 & &+2^2\sum_{n}\left|\sum_{r,p,k}\left(\sum_{n'}\sum_{r',p',k'}n'Y_s(r',n',p',k')\right)Y_s(r,n,p,k)\right.\nonumber\\
 & &-\left.\sum_{k\in \mathcal{K}}\iint_{[0,1]\times [0,1]}\left(\sum_{n'}\sum_{k\in \mathcal{K}}\iint_{[0,1]\times [0,1]}n'dY_s(r',n',p',k')\right)dY_s(r,n,p,k)\right|^2.
 \end{eqnarray}
It suffices to show each term goes to zero as $N\rightarrow \infty$.\\
Since $Y_s^{N}$ is a discrete random variable, we have for each $n$,
$$\sum_{k\in \mathcal{K}}\iint_{[0,1]\times [0,1]}f(r,p,k)dY_{s}(r,n,p,k)=\sum_{k\in \mathcal{K}}\iint_{[0,1]\times [0,1]}f(r,p,k)Y_{s}(r,n,p,k)$$  holds for any function $f$.\\
Then
\begin{equation}\label{bound1_ex}
\begin{split}
&\left|\sum_{r,p,k}\frac{pP^{N}_{\max}}{rR^{N}_{\max}}Y_s(r,n,p,k)-\sum_{k\in \mathcal{K}}\iint_{[0,1]\times [0,1]}\frac{\Lambda}{r}dY_s(r,n,p,k)\right|\\
\leq &\left|\sum_{r,p,k}\left(\frac{pP^{N}_{\max}}{rR^{N}_{\max}}-\frac{p\Lambda}{r}\right) Y_s(r,n,p,k)\right|+\left|\sum_{r,p,k}\frac{p\Lambda}{r}Y_s(r,n,p,k)-\sum_{k\in \mathcal{K}}\iint_{[0,1]\times [0,1]}\frac{p\Lambda}{r}dY_s(r,n,p,k)\right|\\
= & \left|\frac{P^{N}_{\max}}{R^{N}_{\max}}-\Lambda\right|\sum_{r,p,k}\frac{p}{r}Y_s(r,n,p,k)\\
\leq &\left|\frac{P^{N}_{\max}}{R^{N}_{\max}}-\Lambda\right| \frac{C}{\Lambda}\sum_{r,p,k}Y_{s}(r,n,p,k)\\
\leq &\left|\frac{P^{N}_{\max}}{R^{N}_{\max}}-\Lambda\right|\frac{C}{\Lambda}\rightarrow 0.\\
\end{split}	
\end{equation}
Similarly,
\begin{equation}\label{bound2_ex}
\begin{split}
&\left|\sum_{r,p,k}pNP^{N}_{\max}\frac{M}{N}Y_s(r,n,p,k)-\sum_{k\in \mathcal{K}}\iint_{[0,1]\times [0,1]}p\mathcal{P}\gamma dY_s(r,n,p,k)\right|\\
\leq &\left|\sum_{r,p,k}\left(NP^{N}_{\max}\frac{M}{N}-\mathcal{P}\gamma\right) pY_s(r,n,p,k)\right|+\left|\sum_{r,p,k}p\mathcal{P}\gamma Y_s(r,n)-\sum_{k\in \mathcal{K}}\iint_{[0,1]\times [0,1]}p\mathcal{P}\gamma dY_s(r,n,p,k)\right|\\
\leq &\left|NP^{N}_{\max}\frac{M}{N}-\mathcal{P}\gamma \right|\rightarrow 0.\\
\end{split}	
\end{equation}	
The last term is zero since $Y_s^{N}$ is discrete and 
$$\sum_{k\in \mathcal{K}}\iint_{[0,1]\times [0,1]}f(r,p,k)dY_{s}(r,n,p,k)=\sum_{k\in \mathcal{K}}\iint_{[0,1]\times [0,1]}f(r,p,k)Y_{s}(r,n,p,k)$$ for any function $f$.
\end{proof}

\begin{proof}[Proof of Lemma~\ref{martingale-brackets}]
By Dynkin's formula,
\begin{equation}
\begin{split}
\boldlangle \sqrt{N}M^{N}(k)\boldrangle_{t}=&\int_{0}^{t}N\sum_{x\neq Y_{s}^{N}}|x(k)-Y_{s}^{N}(k)|^2 Q(Y_{s}^{N},x)ds\\
=&N\int_{0}^{t}\alpha(Y_{s}^{N})(k)ds\\
=&\int_{0}^{t}\sum_{r} \left[\frac{1}{rNR_{\max}}\left(Y_{s}^{N}(r,k+1)\mathbf{1}_{k<K}+Y_{s}^{N}(r,k)\mathbf{1}_{k>0}\right)\right.\\
&+\left.\left(\frac{M}{N}-\sum_{n'}\sum_{r'}n'Y_{s}^{N}(r',n')\right)\left(Y_{s}^{N}(r,k)\mathbf{1}_{k<K}+Y_{s}^{N}(r,k-1)\mathbf{1}_{k>0}\right)\right]ds\\
=&\int_{0}^{t}(\beta_{+}(Y_{s}^{N})(k)+\beta_{-}(Y_{s}^{N})(k))ds.
\end{split}
\end{equation}

To compute  $\boldlangle \sqrt{N}M^{N}(k),\sqrt{N}M^{N}(k+1)\boldrangle_{t}$ for $k<K$, since
\begin{equation}
\begin{split}
&\boldlangle M^{N}(k)+M^{N}(k+1)\boldrangle_{t}\\
=&\int_{0}^{t}\sum_{x\neq Y_{s}^{N}}\left|x(k)+x(k+1)-Y_{s}^{N}(k)-Y_{s}^{N}(k+1)\right|^2 Q(Y_{s}^{N},x)ds\\
=&\frac{1}{N}\int_{0}^{t}\sum_{r} \left[\frac{1}{rNR_{\max}}\left(Y_{s}^{N}(r,k+2)\mathbf{1}_{k<K-1}+Y_{s}^{N}(r,k)\mathbf{1}_{k>0}\right)\right.\\
&+\left.\left(\frac{M}{N}-\sum_{n'}\sum_{r'}n'Y_{s}^{N}(r',n')\right)\left(Y_{s}^{N}(r,k+1)\mathbf{1}_{k<K-1}+Y_{s}^{N}(r,k-1)\mathbf{1}_{k>0}\right)\right]ds.
\end{split}
\end{equation}
We have that
\begin{equation}
\begin{split}
&\boldlangle \sqrt{N}M^{N}(k),\sqrt{N}M^{N}(k+1)\boldrangle_{t}\\
=&\frac{N}{2}\left[\boldlangle M^{N}(k)+M^{N}(k+1)\boldrangle_{t}-\boldlangle M^{N}(k)\boldrangle_{t}-\boldlangle M^{N}(k+1)\boldrangle_{t}\right]\\
=&\frac{1}{2}\int_{0}^{t}\sum_{r} \left[\frac{1}{rNR_{\max}}\left(Y_{s}^{N}(r,k+2)\mathbf{1}_{k<K-1}+Y_{s}^{N}(r,k)\mathbf{1}_{k>0}\right)\right.\\
&+\left.\left(\frac{M}{N}-\sum_{n'}\sum_{r'}n'Y_{s}^{N}(r',n')\right)\left(Y_{s}^{N}(r,k+1)\mathbf{1}_{k<K-1}+Y_{s}^{N}(r,k-1)\mathbf{1}_{k>0}\right)\right]ds\\
&-\frac{1}{2}\int_{0}^{t}(\beta_{+}(Y_{s}^{N})(k)+\beta_{+}(Y_{s}^{N})(k+1)+\beta_{-}(Y_{s}^{N})(k)+\beta_{-}(Y_{s}^{N})(k+1))ds\\
=&-\int_{0}^{t}\sum_{r}\left[\frac{1}{rNR_{\max}}Y_{s}^{N}(r,k+1)+\left(\frac{M}{N}-\sum_{n'}\sum_{r'}n'Y_{s}^{N}(r',n')\right)Y_{s}^{N}(r,k)\right] ds.
\end{split}
\end{equation}
When $|k-j|>1$, $M^{N}(k)$ and $M^{N}(j)$ are independent, thus 
\begin{equation}
\boldlangle \sqrt{N}M^{N}(k),\sqrt{N}M^{N}(j)\boldrangle_{t}=0.
\end{equation}
\end{proof}

\begin{customthm}{\ref{driftbound}}
For any $s\geq 0$,
\begin{equation}
\limsup_{N\rightarrow \infty}\sqrt{N}\left|\beta(Y_{s}^{N})-b(Y_{s}^{N})\right|<\infty.
\end{equation}
\end{customthm}
\begin{proof}[Proof of Proposition~\ref{driftbound}]
 \begin{equation}
 \begin{split}
 &|\beta(Y_s)-b(Y_s)|\\
 \leq &2\sum_{n}\left|\sum_{r}\frac{1}{rNR_{\max}}Y_s(r,n)-\int_{0}^{1}\frac{\Lambda}{r}dY_s(r,n)\right|+2\sum_{n}\left|\sum_{r}\frac{M}{N}Y_s(r,n)-\int_{0}^{1}\gamma dY_s(r,n)\right|\\
 &+2\sum_{n}\left|\sum_{r}\left(\sum_{n'}\sum_{r'}n'Y_s(r',n')\right)Y_s(r,n)-\int_{0}^{1}\left(\sum_{n'}\int_{0}^{1}n'dY_s(r',n')\right)dY_s(r,n)\right|.
 \end{split}
 \end{equation} 
 By Equation (\ref{bound1}) and (\ref{bound2}), 
 \begin{eqnarray}
 |\beta(Y_s)-b(Y_s)|\leq 2(K+1)\left(\left|\frac{1}{NR_{\max}}-\Lambda\right|\frac{C}{\Lambda}+\left|\frac{M}{N}-\gamma\right|\right).
 \end{eqnarray}
 By the assumptions in Theorem~\ref{difftheorem}, we have
 $$\limsup_{N\rightarrow \infty}\sqrt{N}(\min_{i}\lambda_{i}^{N}-\Lambda)<\infty,\quad 
\limsup_{N\rightarrow \infty}\sqrt{N} \left(\frac{M}{N}-\gamma \right)< \infty.$$
Thus 
\begin{eqnarray}
\limsup_{N\rightarrow \infty}\sqrt{N}\left|\beta(Y_{s}^{N})-b(Y_{s}^{N})\right|&\leq& \limsup_{N\rightarrow \infty}2(K+1)\sqrt{N}\left(\left|\frac{1}{NR_{\max}}-\Lambda\right|\frac{C}{\Lambda}+\left|\frac{M}{N}-\gamma\right|\right) \nonumber \\ &<& \infty. 
\end{eqnarray}
\end{proof}

\begin{proof}[Proof of Lemma \ref{L2bound}]
By Proposition~\ref{driftbound}, $\sqrt{N}|\beta(Y_{s}^{N})-b(Y_{s}^{N})|=O(1)$, then
\begin{equation}
\begin{split}
|D_{t}^{N}|&\leq |D_{0}^{N}|+\sqrt{N}|M_{t}^{N}|+O(1)t+\int_{0}^{t}\sqrt{N} |b(Y_{s}^{N})-b(y_s)|ds\\
&\leq|D_{0}^{N}|+\sqrt{N}|M_{t}^{N}|+O(1)t+\int_{0}^{t}\sqrt{N}L|Y_{s}^{N}-y_s|ds\\
&=|D_{0}^{N}|+\sqrt{N}|M_{t}^{N}|+O(1)t+\int_{0}^{t}L|D_{s}^{N}|ds.
\end{split}
\end{equation}
By Gronwall's Lemma,
$$\sup_{0\leq t\leq T}|D_{t}^{N}|\leq e^{LT}\left(|D_{0}^{N}|+O(1)T+\sup_{0\leq t\leq T}|\sqrt{N}M_{t}^{N}|\right).$$
Then
$$\limsup_{N\rightarrow \infty}\mathbb{E}\left(\sup_{0\leq t \leq T}|D_{t}^{N}|^2\right)\leq e^{2LT}\left[\limsup_{N\rightarrow \infty}\mathbb{E}(|D_{0}^{N}|)+O(1)T+\limsup_{N\rightarrow \infty}\mathbb{E}\left(\sup_{0\leq t\leq T}\sqrt{N}|M_{t}^{N}|\right)\right]^2.$$
By Jensen's inequality and Proposition \ref{bound}, we have that 
$$\left[\mathbb{E}\left(\sup_{0\leq t\leq T}\sqrt{N}|M_{t}^{N}|\right)\right]^2\leq N\mathbb{E}\left(\sup_{0\leq t\leq T}|M_{t}^{N}|^2\right)\leq 4NA(N)T,$$
and that $A(N)=O(\frac{1}{N})$. Therefore
$$\limsup_{N\rightarrow \infty}\mathbb{E}\left(\sup_{0\leq t\leq T}\sqrt{N}|M_{t}^{N}| \right)<\infty.$$
Together with our assumption $\limsup_{N\rightarrow \infty}\mathbb{E}(|D_{0}^{N}|^2)<\infty$, we have
$$\limsup_{N\rightarrow \infty}\mathbb{E}\left(\sup_{0\leq t \leq T}|D_{t}^{N}|^2 \right)<\infty.$$
\end{proof}

\begin{proof}[Proof of Lemma \ref{tightness}]
To prove the tightness of $(D^{N})_{N=1}^{\infty}$ and the continuity of the limit points, we only need to show that the following two tightness conditions hold for each $T>0$ and $\epsilon>0$,
\begin{itemize}
\item[(i)] 
\begin{equation}
\lim_{K\rightarrow \infty}\limsup_{N\rightarrow \infty}\mathbb{P}\left(\sup_{0\leq t\leq T}|D_{t}^{N}|>K \right)=0,
\end{equation}
\item[(ii)] 
\begin{equation}
\lim_{\delta\rightarrow 0}\limsup_{N\rightarrow \infty}\mathbb{P}\left(w(D^{N},\delta,T)\geq \epsilon \right)=0
\end{equation}
\end{itemize}
where for $x\in \mathbb{D}^{d}$,
\begin{equation}
w(x,\delta,T)=\sup\left\{\sup_{u,v\in[t,t+\delta]}|x(u)-x(v)|:0\leq t\leq t+\delta\leq T\right\}.
\end{equation}
By Lemma~\ref{L2bound}, there exists $C_{0}>0$ such that
\begin{eqnarray}
\lim_{K\rightarrow \infty}\limsup_{N\rightarrow \infty}\mathbb{P} \left(\sup_{0\leq t\leq T}|D_{t}^{N}|>K \right) &\leq& \lim_{K\rightarrow \infty}\limsup_{N\rightarrow \infty}\frac{\mathbb{E}\left(\sup_{0\leq t\leq T}|D_{t}^{N}|^2 \right)}{K^2} \nonumber \\
&\leq&  \lim_{K\rightarrow \infty}\frac{C_{0}}{K^2} \nonumber \\
&=&0,
\end{eqnarray}
which proves condition (i).

For condition (ii), we have that

\begin{eqnarray}
D^N_{u} - D^N_{v} &=& \underbrace{\sqrt{N} \cdot ( M^N_{u} - M^N_{v})}_{\text{first term}} + \underbrace{\int^{u}_{v} \sqrt{N} \left( \beta(Y^N_{z}) - b(Y^N_{z})  \right) dz}_{\text{second term}} \nonumber \\&+& \underbrace{\int^{u}_{v} \sqrt{N} \left( b(Y^N_{z}) - b(y_{z})  \right) dz }_{\text{third term}}
\end{eqnarray}
for any $0<t\leq u<v\leq t+\delta\leq T$.  Now it suffices to show that each of the three terms of $D^N_{u} - D^N_{v}$ satisfies condition (ii).  In what follows, we will show that each of the three terms satisfies condition (ii) to complete the proof of tightness. 

 \textbf{For the first term $\sqrt{N} \cdot ( M^N_{u} - M^N_{v})$}, we would like to show that the limiting sample path of $\sqrt{N}M_{t}^{N}$ is a continuous Brownian motion,  by using the martingale central limit theorem.
 
Similar to the proof of Proposition~\ref{drift}, we can show that 
\begin{equation}\label{equation1}
\sup_{t\leq T}\left|\beta_{+}(Y_{t}^{N})-b_{+}(Y_{t}^{N})\right|\xrightarrow{p}0, \quad \sup_{t\leq T}\left|\beta_{-}(Y_{t}^{N})-b_{-}(Y_{t}^{N})\right|\xrightarrow{p}0.
\end{equation}
And by the proof of Proposition~\ref{Lipschitz}, $b_{+}(y), b_{-}(y)$ are also Lipschitz with constant $L$, then by the fact that the composition of Lipschitz functions are also Lipschitz,
\begin{eqnarray}\label{equation2}
\max\left\{\sup_{t\leq T}|b_{+}(Y_{t}^{N})-b_{+}(y_{t})|,\sup_{t\leq T}|b_{-}(Y_{t}^{N})-b_{-}(y_{t})|\right\}\leq L\sup_{t\leq T}|Y_{t}^{N}-y_{t}|.
\end{eqnarray}
By Theorem~\ref{fluid_limit}, 
\begin{equation}\label{equation3}
\sup_{t\leq T}|Y_{t}^{N}-y_{t}|\xrightarrow{p} 0.
\end{equation}
Thus combining Equations (\ref{equation1}), (\ref{equation2}) and (\ref{equation3}), we have 
\begin{eqnarray}
& &\lim_{N\rightarrow \infty}\mathbb{P}\left(\sup_{t\leq T}\left|\boldlangle \sqrt{N}M^N(k)\boldrangle_{t}- \boldlangle M(k)\boldrangle_{t}\right|>\epsilon\right)\nonumber \\
&=&\lim_{N\rightarrow \infty}\mathbb{P}\left(\sup_{t\leq T}\left|\int_{0}^{t}\left(\beta_{+}(Y_{s}^{N})(k)+\beta_{-}(Y_{s}^{N})(k)- b_{+}(y_{s})(k)-b_{-}(y_{s})(k)\right)ds\right|>\epsilon\right)\nonumber\\
&\leq & \lim_{N\rightarrow \infty}\mathbb{P}\left(\sup_{t\leq T}T\left|\beta_{+}(Y_{t}^{N})(k)- b_{+}(Y^{N}_{t})(k)\right|>\epsilon/4\right)+\lim_{N\rightarrow \infty}\mathbb{P}\left(\sup_{t\leq T}T\left|b_{+}(Y_{t}^{N})(k)- b_{+}(y_{t})(k)\right|>\epsilon/4\right) \nonumber\\
&+&\lim_{N\rightarrow \infty}\mathbb{P}\left(\sup_{t\leq T}T\left|\beta_{-}(Y_{t}^{N})(k)- b_{-}(Y^{N}_{t})(k)\right|>\epsilon/4\right)+\lim_{N\rightarrow \infty}\mathbb{P}\left(\sup_{t\leq T}T\left|b_{-}(Y_{t}^{N})(k)- b_{-}(y_{t})(k)\right|>\epsilon/4\right)\nonumber \\
&\leq & \lim_{N\rightarrow \infty}\mathbb{P}\left(\sup_{t\leq T}T\left|\beta_{+}(Y_{t}^{N})(k)- b_{+}(Y^{N}_{t})(k)\right|>\epsilon/4\right)+2\lim_{N\rightarrow \infty}\mathbb{P}\left(\sup_{t\leq T}LT\left|Y_{t}^{N}- y_{t}\right|>\epsilon/4 \right) \nonumber\\
& +&\lim_{N\rightarrow \infty}\mathbb{P}\left(\sup_{t\leq T}T\left|\beta_{-}(Y_{t}^{N})(k)- b_{-}(Y^{N}_{t})(k)\right|>\epsilon/4\right)\nonumber \\
&=& 0,
\end{eqnarray}
for any $\epsilon>0$ and $0\leq k\leq K$. This result implies that
\begin{equation}
\sup_{t\leq T}\left|\boldlangle \sqrt{N}M^N(k)\boldrangle_{t}- \boldlangle M(k)\boldrangle_{t}\right|\xrightarrow{p} 0.
\end{equation}

For the adjacent terms, we have
\begin{eqnarray}
& &\lim_{N\rightarrow \infty}\mathbb{P}\left(\sup_{t\leq T}\left|\boldlangle \sqrt{N}M^N(k),\sqrt{N}M^N(k+1)\boldrangle_{t}- \boldlangle M(k),M(k+1)\boldrangle_{t}\right|>\epsilon\right)\nonumber \\
&=&\lim_{N\rightarrow \infty}\mathbb{P}\left(\sup_{t\leq T}\left|\int_{0}^{t}\left[\sum_{r}\left(\frac{1}{rNR_{\max}}Y_{s}^{N}(r,k+1)+\left(\frac{M}{N}-\sum_{n'}\sum_{r'}n'Y_{s}^{N}(r',n')\right)Y_{s}^{N}(r,k)\right)\right. \right.\right.\nonumber\\
&-& \left.\left.\left.\left(\int_{0}^{1}\frac{\Lambda}{r}dy_{s}(r,k+1)+\int_{0}^{1}\left(\gamma-\sum_{n}\int_{0}^{1}ndy_{s}(r,n)\right) dy_{s}(r,k)\right)\right]ds\right|>\epsilon\right)\nonumber\\
&\leq & \lim_{N\rightarrow \infty}\mathbb{P}\left(\sup_{t\leq T}T\left|\sum_{r}\frac{1}{rNR_{\max}}Y_{t}^{N}(r,k+1)-\int_{0}^{1}\frac{\Lambda}{r}dY_{t}(r,k+1) \right|>\epsilon/3\right)\nonumber \\
&+&\lim_{N\rightarrow \infty}\mathbb{P}\left(\sup_{t\leq T}T\left|\left(\frac{M}{N}-\sum_{n',r'}n'Y_{t}^{N}(r',n')\right)Y_{t}^{N}(r,k)- \int_{0}^{1}\left(\gamma-\sum_{n}\int_{0}^{1}ndY_{t}(r,n)\right) dY_{t}(r,k)\right|>\epsilon/3\right)\nonumber\\
&+&\lim_{N\rightarrow \infty}\mathbb{P}\left(\sup_{t\leq T}2LT\left|Y_{t}^{N}- y_{t}\right|>\epsilon/3 \right)\nonumber \\
&\leq &  \lim_{N\rightarrow \infty}\mathbb{P}\left(\sup_{t\leq T}T\left|\frac{C}{\Lambda NR_{\max}}-\frac{\Lambda}{\Lambda /C}\right|\sum_{r}Y_{s}(r,k+1) >\epsilon/3\right)\nonumber \\
&+&\lim_{N\rightarrow \infty}\mathbb{P}\left(\sup_{t\leq T}T\left|\frac{M}{N}-\gamma\right| \sum_{r}Y_{s}(r,k)>\epsilon/3\right)\nonumber \\
&\leq &  \lim_{N\rightarrow \infty}\mathbb{P}\left(\sup_{t\leq T}\frac{CT}{\Lambda}\left|\frac{1}{ NR_{\max}}-\Lambda \right| >\epsilon/3\right)+\lim_{N\rightarrow \infty}\mathbb{P}\left(\sup_{t\leq T}T\left|\frac{M}{N}-\gamma\right| >\epsilon/3\right)\nonumber \\
&=&0,
\end{eqnarray}
which implies
\begin{equation}
\sup_{t\leq T}\left|\boldlangle \sqrt{N}M^N(k),\sqrt{N}M^N(k+1)\boldrangle_{t}- \boldlangle M(k), M(k+1)\boldrangle_{t}\right|\xrightarrow{p} 0.
\end{equation}
Since for all $|i-j|>1$, 
$$\boldlangle \sqrt{N}M^N(i),\sqrt{N}M^N(j)\boldrangle_{t} =\boldlangle M(i), M(j)\boldrangle_{t}=0.$$
We can conclude that 
\begin{equation}
\sup_{t\leq T}\left|\boldlangle \sqrt{N}M^N(i),\sqrt{N}M^N(j)\boldrangle_{t}- \boldlangle M(i), M(j)\boldrangle_{t}\right|\xrightarrow{p} 0.
\end{equation}
for all $0\leq i,j\leq K$.

We also know that the jump size of $Y^{N}_{t}$ is $1/N$, therefore
\begin{equation}
\lim_{N\rightarrow \infty}\mathbb{E}\left[\sup_{0<t\leq T}\left|\sqrt{N}M^{N}_{t}-\sqrt{N}M^{N}_{t-}\right| \right]=\lim_{N\rightarrow \infty}\mathbb{E}\left[\sup_{0<t\leq T}\left|\sqrt{N}Y^{N}_{t}-\sqrt{N}Y^{N}_{t-}\right| \right]=0.
\end{equation}
By Theorem 1.4 in Chapter 7 of \citet{Ethier2009},  $\sqrt{N}M^{N}_{t}$ converges to the continuous Brownian motion $M_{t}$ in distribution in $\mathbb{D}(\mathbb{R}_{+},\mathbb{R}^{K+1})$. By Prohorov's theorem, $(\sqrt{N}M^{N})_{N=1}^{\infty}$ is tight. This automatically implies the tightness condition ii).

\textbf{For the second term $\int^{u}_{v} \sqrt{N} \left( \beta(Y^N_{z}) - b(Y^N_{z})  \right) dz$}, we have by Proposition \ref{driftbound} that the quantity $ \sqrt{N} \left( \beta(Y^N_{z}) - b(Y^N_{z})  \right) $ is bounded for any value of $z\in [0,T]$.  Therefore, there exists some constant $C_{1}$ that does not depend on $N$ such that
\begin{equation}
\sup_{z\in [0,T]}\sqrt{N} \left| \beta(Y^N_{z}) - b(Y^N_{z})  \right|\leq C_{1}.
\end{equation}
Then
\begin{eqnarray}
& &\lim_{\delta\rightarrow 0}\lim_{N\rightarrow \infty}\mathbb{P}\left(\sup_{u,v\in [0,T],|u-v|\leq \delta}\int^{u}_{v} \sqrt{N} \left| \beta(Y^N_{z}) - b(Y^N_{z})  \right| dz > \epsilon \right)\nonumber \\ 
&\leq &  \lim_{\delta\rightarrow 0}\lim_{N\rightarrow \infty}\mathbb{P}\left(\delta \sup_{z\in [0,T]}\sqrt{N} \left| \beta(Y^N_{z}) - b(Y^N_{z})  \right|  > \epsilon \right)\nonumber \\
&\leq & \lim_{\delta\rightarrow 0}\mathbb{P}\left(\delta C_{1}  > \epsilon \right)\nonumber \\
&=& 0.
\end{eqnarray}
 Thus, we have proved the oscillation bound for the second term.  
 
 Finally \textbf{for the third term} we have that 
\begin{eqnarray}
\int^{u}_{v} \sqrt{N} \left| b(Y^N_{z}) - b(y_{z})  \right| dz & \leq& \int^{u}_{v} \sqrt{N} L\left| Y^N_{z} - y_{z} \right| dz\nonumber \\
&=& \int^{u}_{v} L \cdot \left|D^N_{z}\right| dz \nonumber\\\
&\leq & L\delta \sup_{t\in [0,T]}|D^{N}_{t}|.
\end{eqnarray}
By Lemma~\ref{L2bound},  
\begin{eqnarray}
& &\lim_{\delta\rightarrow 0}\lim_{N\rightarrow \infty}\mathbb{P}\left(\sup_{u,v\in [0,T],|u-v|\leq \delta}\int^{u}_{v} \sqrt{N} \left| b(Y^N_{z}) - b(y_{z})  \right| dz>\epsilon\right)\nonumber\\
&\leq &\lim_{\delta\rightarrow 0}\lim_{N\rightarrow \infty}\mathbb{P}\left(L\delta\sup_{t\in [0,T]}|D^{N}_{t}|>\epsilon\right)\nonumber\\
&\leq & \lim_{\delta\rightarrow 0}\lim_{N\rightarrow \infty}\frac{\mathbb{E}\left(\sup_{t\in [0,T]}|D^{N}_{t}|^2\right)}{(\epsilon/L \delta)^2}\nonumber \\
&\leq &\lim_{\delta\rightarrow 0}\frac{C_{0}(L\delta)^2}{\epsilon^2}\nonumber \\
&=& 0,
\end{eqnarray}
 which implies that the oscillation bound holds for the third term.  
\end{proof}

\begin{proof}[Proof of Theorem \ref{dt_solution}]
To prove the existence and uniqueness of the SDE (\ref{sde}), we show the following two conditions hold:
There exists a constant $H>0$ such that
\begin{itemize}
\item[(1)]Lipschitz condition: for any $D,\tilde{D}\in \mathbb{R}^{K+1}$, any $t\in (t_{0},T)$, 
\begin{equation}
|b'(y_{t})D-b'(y_{t})\tilde{D}|\leq H|D-\tilde{D}|.
\end{equation}
\item[(2)]Linear growth condition: for any $D\in \mathbb{R}^{K+1}$, any $t\in (t_{0},T)$,
\begin{equation}
|b_{+}(y_{t})+b_{-}(y_{t})|\leq H(1+|D|), \quad |b'(y_{t})D|\leq H(1+|D|).
\end{equation}

\end{itemize}
By Proposition~\ref{Lipschitz}, $b(y)$ is Lipschitz, and by our assumption in Theorem~\ref{difftheorem}, $b(y)$ is also continuously differentiable, thus $b'(y)$ is bounded. Then conditions (1)  and (2) follow, which prove that there exist a unique solution to the SDE (\ref{sde}).

Take expectation on both sides of Equation (\ref{dt}), since $\mathbb{E}\left[ \int_{0}^{t}e^{\int_{s}^{t}b'(y_u)du}dM_s\right]=0$, we have $$\mathbb{E}[D_t]=e^{\int_{0}^{t}\mathcal{A}(s)ds} \mathbb{E}[D_0].$$
Therefore
\begin{eqnarray}
D_t-\mathbb{E}[D_t]&=&e^{\int_{0}^{t}\mathcal{A}(s)ds}(D_0-\mathbb{E}[D_0])+\int_{0}^{t}e^{\int_{s}^{t}\mathcal{A}(u)du}dM_s,
\end{eqnarray}
\begin{eqnarray}
\Sigma(t)&=&E[(D_t-\mathbb{E}[D_t])(D_t-\mathbb{E}[D_t])^\top]\nonumber  \\
&=&e^{\int_{0}^{t}\mathcal{A}(s)ds}E[(D_0-\mathbb{E}[D_0])(D_0-\mathbb{E}[D_0])^\top] \left(e^{\int_{0}^{t}\mathcal{A}(s)ds}\right)^\top \nonumber\\&+&
\left(\int_{0}^{t}e^{\int_{s}^{t}\mathcal{A}(u)du}dM_s\right)\left(\int_{0}^{t}e^{\int_{s}^{t}\mathcal{A}(u)du}dM_s\right)^\top\nonumber\\
&=&e^{\int_{0}^{t}\mathcal{A}(s)ds}\Sigma(0)e^{\int_{0}^{t}\mathcal{A^\top}(s)ds}+\int_{0}^{t}e^{\int_{s}^{t}\mathcal{A}(u)du}\mathcal{B}(s)e^{\int_{s}^{t}\mathcal{A^\top}(u)du}ds.
\end{eqnarray}
\end{proof}

\begin{customthm}{\ref{bbound_ex}}
For any $s\geq 0$,
\begin{equation}
\limsup_{N\rightarrow \infty}\sqrt{N}\left|\beta(Y_{s}^{N})-b(Y_{s}^{N})\right|<\infty.
\end{equation}
\end{customthm}
\begin{proof}[Proof of Proposition~\ref{bbound_ex}]
 \begin{equation}
 \begin{split}
 &|\beta(Y_s)-b(Y_s)|\\
 \leq &2\sum_{n}\left|\sum_{r,p,k}\frac{pP^{N}_{\max}}{rR^{N}_{\max}}Y_s(r,n,p,k)-\sum_{k\in \mathcal{K}}\iint_{[0,1]\times [0,1]}\frac{p\Lambda}{r}dY_s(r,n,p,k)\right|\nonumber\\
 &+2\sum_{n}\left|\sum_{r,p,k}pNP^{N}_{\max}\frac{M}{N}Y_s(r,n,p,k)-\sum_{k\in \mathcal{K}}\iint_{[0,1]\times [0,1]}p\mathcal{P}\gamma dY_s(r,n,p,k)\right|\\
 &+2\sum_{n}\left|\sum_{r,p,k}\left(\sum_{n'}\sum_{r',p',k'}n'Y_s(r',n',p',k')\right)Y_s(r,n,p,k)\right.\nonumber\\
  &-\left.\sum_{k\in \mathcal{K}}\iint_{[0,1]\times [0,1]}\left(\sum_{n'}\sum_{k\in \mathcal{K}}\iint_{[0,1]\times [0,1]}n'dY_s(r',n',p',k')\right)dY_s(r,n,p,k)\right|.
 \end{split}
 \end{equation} 
 By Equation (\ref{bound1_ex}) and (\ref{bound2_ex}), 
 \begin{eqnarray}
 |\beta(Y_s)-b(Y_s)|\leq 2(K_{\max}+1)\left(\left|\frac{P_{\max}}{R_{\max}}-\Lambda\right|\frac{C}{\Lambda}+\left|NP_{\max}\frac{M}{N}-\mathcal{P}\gamma\right|\right).
 \end{eqnarray}
 By the assumptions in Theorem~\ref{difftheorem_ex}, we have
 $$\limsup_{N\rightarrow \infty}\sqrt{N}\left(\frac{P_{\max}^{N}}{R_{\max}^{N}}-\Lambda\right)<\infty, \limsup_{N\rightarrow \infty}\sqrt{N}\left(\frac{M}{N}-\gamma\right)< \infty, \limsup_{N\rightarrow \infty}\sqrt{N}\left(NP_{\max}^{N}-\mathcal{P}\right)< \infty.$$
Thus 
\begin{eqnarray}
\limsup_{N\rightarrow \infty}\sqrt{N}\left|\beta(Y_{s}^{N})-b(Y_{s}^{N})\right|&\leq& \limsup_{N\rightarrow \infty}2(K_{\max}+1)\sqrt{N}\left(\left|\frac{P_{\max}}{R_{\max}}-\Lambda\right|\frac{C}{\Lambda}+\left|NP_{\max}\frac{M}{N}-\mathcal{P}\gamma\right|\right) \nonumber \\ &<& \infty. 
\end{eqnarray}
\end{proof}

\begin{customthm}{\ref{adjacent_ex}}
For $k<K_{\max}$,
\begin{equation}
\lim_{N\rightarrow \infty}\mathbb{P}\left(\sup_{t\leq T}\left|\boldlangle \sqrt{N}M^N(k),\sqrt{N}M^N(k+1)\boldrangle_{t}- \boldlangle M(k),M(k+1)\boldrangle_{t}\right|>\epsilon\right)=0.
\end{equation}
\begin{proof}[Proof of Proposition \ref{adjacent_ex}]
For the adjacent terms, we have
\begin{eqnarray}
& &\lim_{N\rightarrow \infty}\mathbb{P}\left(\sup_{t\leq T}\left|\boldlangle \sqrt{N}M^N(k),\sqrt{N}M^N(k+1)\boldrangle_{t}- \boldlangle M(k),M(k+1)\boldrangle_{t}\right|>\epsilon\right)\nonumber \\
&=&\lim_{N\rightarrow \infty}\mathbb{P}\left(\sup_{t\leq T}\left|\int_{0}^{t}\left[\sum_{r,p,K\in\mathcal{K}}\left(\frac{pP_{\max}}{rR_{\max}}Y_{s}^{N}(r,k+1,p,K)+\left(\frac{pNP_{\max}M}{N}-\sum_{n'}\sum_{r',p',K'}n'Y_{s}^{N}(r',n',p',K')\right)\right. \right.\right.\right.\nonumber\\
& & \left.\left. \left. \left. Y_{s}^{N}(r,k,p,K)\right)-\left(\sum_{K\in \mathcal{K}}\iint_{[0,1]\times [0,1]}\frac{p\Lambda}{r}dY_s(r,k+1,p,K)\right. \right.\right.\right.\nonumber\\
&+ & \left.\left. \left. \left. \sum_{K\in \mathcal{K}}\iint_{]0,1]\times [0,1]}p\mathcal{P}\left(\gamma-\sum_{n}\sum_{K\in \mathcal{K}}\iint_{]0,1]\times [0,1]}ndy_{s}(r,n,p,K)\right) dy_{s}(r,k,p,K)\right)\right]ds\right|>\epsilon\right)\nonumber\\
&\leq & \lim_{N\rightarrow \infty}\mathbb{P}\left(\sup_{t\leq T}T\left|\sum_{r,p,K}\frac{pP_{\max}}{rR_{\max}}Y_{t}^{N}(r,k+1,P,K)-\sum_{K\in \mathcal{K}}\iint_{[0,1]\times [0,1]}\frac{p\Lambda}{r}dY_s(r,k+1,p,K) \right|>\epsilon/3\right)\nonumber \\
&+&\lim_{N\rightarrow \infty}\mathbb{P}\left(\sup_{t\leq T}T\left|\sum_{r,p,K\in\mathcal{K}}\left(pNP_{\max}\frac{M}{N}-\sum_{n'}\sum_{r',p',K'}n'Y_{s}^{N}(r',n',p',K')\right)Y_{t}^{N}(r,k,p,K)\right.\right.\nonumber\\
& & \left.\left. - \sum_{K\in \mathcal{K}}\iint_{]0,1]\times [0,1]}p\mathcal{P}\left(\gamma-\sum_{n}\sum_{K\in \mathcal{K}}\iint_{]0,1]\times [0,1]}ndy_{s}(r,n,p,K)\right) dy_{s}(r,k,p,K)\right|>\epsilon/3\right)\nonumber\\
&+&\lim_{N\rightarrow \infty}\mathbb{P}\left(\sup_{t\leq T}2LT\left|Y_{t}^{N}- y_{t}\right|>\epsilon/3 \right)\nonumber \\
&\leq &  \lim_{N\rightarrow \infty}\mathbb{P}\left(\sup_{t\leq T}T\left|\frac{CP_{\max}}{\Lambda R_{\max}}-\frac{\Lambda}{\Lambda /C}\right|\sum_{r,p,K}Y_{s}(r,k+1,p,K) >\epsilon/3\right)\nonumber \\
&+&\lim_{N\rightarrow \infty}\mathbb{P}\left(\sup_{t\leq T}T\left| NP_{\max}\frac{M}{N}-\mathcal{P}\gamma\right| \sum_{r,p,K}pY_{s}(r,k,p,K)>\epsilon/3\right)\nonumber \\
&\leq &  \lim_{N\rightarrow \infty}\mathbb{P}\left(\sup_{t\leq T}\frac{CT}{\Lambda}\left|\frac{P_{\max}}{ R_{\max}}-\Lambda \right| >\epsilon/3\right)+\lim_{N\rightarrow \infty}\mathbb{P}\left(\sup_{t\leq T}T\left|NP_{\max}\frac{M}{N}-\mathcal{P}\gamma\right| >\epsilon/3\right)\nonumber \\
&=&0,
\end{eqnarray}
which implies
\begin{equation}
\sup_{t\leq T}\left|\boldlangle \sqrt{N}M^N(k),\sqrt{N}M^N(k+1)\boldrangle_{t}- \boldlangle M(k), M(k+1)\boldrangle_{t}\right|\xrightarrow{p} 0.
\end{equation}
\end{proof}
\end{customthm}

\bibliographystyle{plainnat}
\bibliography{bike_clt}

\begin{thebibliography}{35}
\providecommand{\natexlab}[1]{#1}
\providecommand{\url}[1]{\texttt{#1}}
\expandafter\ifx\csname urlstyle\endcsname\relax
  \providecommand{\doi}[1]{doi: #1}\else
  \providecommand{\doi}{doi: \begingroup \urlstyle{rm}\Url}\fi

\bibitem[Ames and Pachpatte(1997)]{ames1997inequalities}
William~F Ames and BG~Pachpatte.
\newblock \emph{Inequalities for differential and integral equations}, volume
  197.
\newblock Academic press, 1997.

\bibitem[Benchimol et~al.(2011)Benchimol, Benchimol, Chappert, De~La~Taille,
  Laroche, Meunier, and Robinet]{benchimol2011balancing}
Mike Benchimol, Pascal Benchimol, Beno{\^\i}t Chappert, Arnaud De~La~Taille,
  Fabien Laroche, Fr{\'e}d{\'e}ric Meunier, and Ludovic Robinet.
\newblock Balancing the stations of a self service ``bike hire'' system.
\newblock \emph{RAIRO-Operations Research}, 45\penalty0 (1):\penalty0 37--61,
  2011.

\bibitem[Chemla et~al.(2013)Chemla, Meunier, Pradeau, Calvo, and
  Yahiaoui]{chemla2013self}
Daniel Chemla, Fr{\'e}d{\'e}ric Meunier, Thomas Pradeau, Roberto~Wolfler Calvo,
  and Houssame Yahiaoui.
\newblock Self-service bike sharing systems: simulation, repositioning,
  pricing.
\newblock 2013.

\bibitem[Contardo et~al.(2012)Contardo, Morency, and
  Rousseau]{contardo2012balancing}
Claudio Contardo, Catherine Morency, and Louis-Martin Rousseau.
\newblock \emph{Balancing a dynamic public bike-sharing system}, volume~4.
\newblock Cirrelt Montreal, 2012.

\bibitem[Darling et~al.(2008)Darling, Norris, et~al.]{darling2008differential}
RWR Darling, James~R Norris, et~al.
\newblock Differential equation approximations for markov chains.
\newblock \emph{Probability surveys}, 5\penalty0 (1):\penalty0 37--79, 2008.

\bibitem[DeMaio(2009)]{demaio2009bike}
Paul DeMaio.
\newblock Bike-sharing: History, impacts, models of provision, and future.
\newblock \emph{Journal of Public Transportation}, 12\penalty0 (4):\penalty0 3,
  2009.

\bibitem[Engblom and Pender(2014)]{EP}
Stefan Engblom and Jamol Pender.
\newblock Approximations for the moments of nonstationary and state dependent
  birth-death queues.
\newblock \emph{Submitted for publication to Queueing Systems}, 2014.

\bibitem[Ethier and Kurtz(2009)]{Ethier2009}
Stewart~N Ethier and Thomas~G Kurtz.
\newblock \emph{Markov processes: characterization and convergence}, volume
  282.
\newblock John Wiley \& Sons, 2009.

\bibitem[Freund et~al.(2017)Freund, Henderson, and
  Shmoys]{freund2017minimizing}
Daniel Freund, Shane~G Henderson, and David~B Shmoys.
\newblock Minimizing multimodular functions and allocating capacity in
  bike-sharing systems.
\newblock In \emph{International Conference on Integer Programming and
  Combinatorial Optimization}, pages 186--198. Springer, 2017.

\bibitem[Fricker et~al.(2012)Fricker, Gast, and Mohamed]{fricker2012mean}
Christine Fricker, Nicolas Gast, and Hanene Mohamed.
\newblock Mean field analysis for inhomogeneous bike sharing systems.
\newblock 2012.

\bibitem[George and Xia(2011)]{george2011fleet}
David~K George and Cathy~H Xia.
\newblock Fleet-sizing and service availability for a vehicle rental system via
  closed queueing networks.
\newblock \emph{European journal of operational research}, 211\penalty0
  (1):\penalty0 198--207, 2011.

\bibitem[Ghosh et~al.(2017)Ghosh, Varakantham, Adulyasak, and
  Jaillet]{ghosh2017dynamic}
Supriyo Ghosh, Pradeep Varakantham, Yossiri Adulyasak, and Patrick Jaillet.
\newblock Dynamic repositioning to reduce lost demand in bike sharing systems.
\newblock \emph{Journal of Artificial Intelligence Research}, 58:\penalty0
  387--430, 2017.

\bibitem[Hampshire and Marla(2012)]{hampshire2012analysis}
Robert~C Hampshire and Lavanya Marla.
\newblock An analysis of bike sharing usage: Explaining trip generation and
  attraction from observed demand.
\newblock In \emph{91st Annual meeting of the transportation research board,
  Washington, DC}, pages 12--2099, 2012.

\bibitem[Jian et~al.(2016)Jian, Freund, Wiberg, and
  Henderson]{jian2016simulation}
Nanjing Jian, Daniel Freund, Holly~M Wiberg, and Shane~G Henderson.
\newblock Simulation optimization for a large-scale bike-sharing system.
\newblock In \emph{Proceedings of the 2016 Winter Simulation Conference}, pages
  602--613. IEEE Press, 2016.

\bibitem[Ko and Pender(2017{\natexlab{a}})]{ko2016strong}
Young~Myoung Ko and Jamol Pender.
\newblock Strong approximations for time varying infinite-server queues with
  non-renewal arrival and service processes.
\newblock \emph{Stochastic Models To Appear}, 2017{\natexlab{a}}.

\bibitem[Ko and Pender(2017{\natexlab{b}})]{ko2017diffusion}
Young~Myoung Ko and Jamol Pender.
\newblock Diffusion limits for the (mapt/pht/∞) n queueing network.
\newblock \emph{Operations Research Letters}, 45\penalty0 (3):\penalty0
  248--253, 2017{\natexlab{b}}.

\bibitem[Laporte et~al.(2015)Laporte, Meunier, and Calvo]{laporte2015shared}
Gilbert Laporte, Fr{\'e}d{\'e}ric Meunier, and Roberto~Wolfler Calvo.
\newblock Shared mobility systems.
\newblock \emph{4OR}, 13\penalty0 (4):\penalty0 341--360, 2015.

\bibitem[Massey and Pender(2013)]{Massey2013}
William~A. Massey and Jamol Pender.
\newblock {Gaussian skewness approximation for dynamic rate multi-server queues
  with abandonment}.
\newblock \emph{Queueing Systems}, 75\penalty0 (2-4):\penalty0 243--277,
  February 2013.

\bibitem[Nair and Miller-Hooks(2011)]{nair2011fleet}
Rahul Nair and Elise Miller-Hooks.
\newblock Fleet management for vehicle sharing operations.
\newblock \emph{Transportation Science}, 45\penalty0 (4):\penalty0 524--540,
  2011.

\bibitem[Nair and Miller-Hooks(2016)]{nair2016equilibrium}
Rahul Nair and Elise Miller-Hooks.
\newblock Equilibrium design of bicycle sharing systems: the case of washington
  dc.
\newblock \emph{EURO Journal on Transportation and Logistics}, 5\penalty0
  (3):\penalty0 321--344, 2016.

\bibitem[Nair et~al.(2013)Nair, Miller-Hooks, Hampshire, and
  Bu{\v{s}}i{\'c}]{nair2013large}
Rahul Nair, Elise Miller-Hooks, Robert~C Hampshire, and Ana Bu{\v{s}}i{\'c}.
\newblock Large-scale vehicle sharing systems: analysis of v{\'e}lib'.
\newblock \emph{International Journal of Sustainable Transportation},
  7\penalty0 (1):\penalty0 85--106, 2013.

\bibitem[Nirenberg et~al.(2018)Nirenberg, Daw, and Pender]{nirenberg2018impact}
Samantha Nirenberg, Andrew Daw, and Jamol Pender.
\newblock The impact of queue length rounding and delayed app information on
  disney world queues.
\newblock In \emph{Proceedings of the 2018 Winter Simulation Conference}, pages
  3849--3860. IEEE Press, 2018.

\bibitem[Novitzky et~al.(2019)Novitzky, Pender, Rand, and
  Wesson]{novitzky2019nonlinear}
Sophia Novitzky, Jamol Pender, Richard~H Rand, and Elizabeth Wesson.
\newblock Nonlinear dynamics in queueing theory: Determining the size of
  oscillations in queues with delay.
\newblock \emph{SIAM Journal on Applied Dynamical Systems}, 18\penalty0
  (1):\penalty0 279--311, 2019.

\bibitem[O'Mahony and Shmoys(2015)]{o2015data}
Eoin O'Mahony and David~B Shmoys.
\newblock Data analysis and optimization for (citi) bike sharing.
\newblock In \emph{AAAI}, pages 687--694, 2015.

\bibitem[O'Mahony(2015)]{o2015smarter}
Eoin~Daniel O'Mahony.
\newblock \emph{Smarter tools for (Citi) bike sharing}.
\newblock PhD thesis, Cornell University, 2015.

\bibitem[Pender(2014)]{Pender2014}
Jamol Pender.
\newblock Gram charlier expansion for time varying multiserver queues with
  abandonment.
\newblock \emph{SIAM Journal on Applied Mathematics}, 74\penalty0 (4):\penalty0
  1238--1265, 2014.

\bibitem[Pender(2015)]{Pender2015}
Jamol Pender.
\newblock Nonstationary loss queues via cumulant moment approximations.
\newblock \emph{Probability in the Engineering and Informational Sciences},
  29\penalty0 (01):\penalty0 27--49, 2015.

\bibitem[Pender(2016)]{pender2016sampling}
Jamol Pender.
\newblock Sampling the functional kolmogorov forward equations for
  nonstationary queueing networks.
\newblock \emph{INFORMS Journal on Computing}, 29\penalty0 (1):\penalty0 1--17,
  2016.

\bibitem[Pender and Ko(Fall 2017)]{pender2016approximations}
Jamol Pender and Young~Myoung Ko.
\newblock Approximations for the queue length distributions of time-varying
  many-server queues.
\newblock \emph{INFORMS Journal on Computing}, 29\penalty0 (4):\penalty0
  668--704, Fall 2017.

\bibitem[Pender et~al.(2017)Pender, Rand, and Wesson]{pender2017queues}
Jamol Pender, Richard~H Rand, and Elizabeth Wesson.
\newblock Queues with choice via delay differential equations.
\newblock \emph{International Journal of Bifurcation and Chaos}, 27\penalty0
  (04):\penalty0 1730016, 2017.

\bibitem[Pender et~al.(2018)Pender, Rand, and Wesson]{pender2018analysis}
Jamol Pender, Richard~H Rand, and Elizabeth Wesson.
\newblock An analysis of queues with delayed information and time-varying
  arrival rates.
\newblock \emph{Nonlinear Dynamics}, 91\penalty0 (4):\penalty0 2411--2427,
  2018.

\bibitem[Pfrommer et~al.(2014)Pfrommer, Warrington, Schildbach, and
  Morari]{pfrommer2014dynamic}
Julius Pfrommer, Joseph Warrington, Georg Schildbach, and Manfred Morari.
\newblock Dynamic vehicle redistribution and online price incentives in shared
  mobility systems.
\newblock \emph{IEEE Transactions on Intelligent Transportation Systems},
  15\penalty0 (4):\penalty0 1567--1578, 2014.

\bibitem[Raviv et~al.(2013)Raviv, Tzur, and Forma]{raviv2013static}
Tal Raviv, Michal Tzur, and Iris~A Forma.
\newblock Static repositioning in a bike-sharing system: models and solution
  approaches.
\newblock \emph{EURO Journal on Transportation and Logistics}, 2\penalty0
  (3):\penalty0 187--229, 2013.

\bibitem[Schuijbroek et~al.(2017)Schuijbroek, Hampshire, and
  Van~Hoeve]{schuijbroek2017inventory}
Jasper Schuijbroek, Robert~C Hampshire, and W-J Van~Hoeve.
\newblock Inventory rebalancing and vehicle routing in bike sharing systems.
\newblock \emph{European Journal of Operational Research}, 257\penalty0
  (3):\penalty0 992--1004, 2017.

\bibitem[Shaheen et~al.(2010)Shaheen, Guzman, and
  Zhang]{shaheen2010bikesharing}
Susan Shaheen, Stacey Guzman, and Hua Zhang.
\newblock Bikesharing in europe, the americas, and asia: past, present, and
  future.
\newblock \emph{Transportation Research Record: Journal of the Transportation
  Research Board}, \penalty0 (2143):\penalty0 159--167, 2010.

\end{thebibliography}
\end{document}